\documentclass[a4paper,reqno]{amsart}
\pdfoutput=1
\usepackage[utf8]{inputenc}
\usepackage[T1]{fontenc}
\usepackage{amssymb,mathtools}
\usepackage{mathrsfs}
\usepackage{dsfont}
\usepackage{bbm}

\usepackage[a4paper]{geometry}
\usepackage{lmodern}
\usepackage{color}
\usepackage{comment}
\usepackage[normalem]{ulem}
\usepackage{tikz}

\usepackage[all]{xy}
\usepackage{microtype}

\usepackage{enumitem}
\setlist[enumerate,1]{label=\textup{(\arabic*)}}

\usepackage[lite]{amsrefs}

\renewcommand*{\PrintDOI}[1]{\href{http://dx.doi.org/\detokenize{#1}}{doi: \detokenize{#1}}}

\BibSpec{book}{%
  +{}  {\PrintPrimary}                {transition}
  +{,} { \textit}                     {title}
  +{.} { }                            {part}
  +{:} { \textit}                     {subtitle}
  +{,} { \PrintEdition}               {edition}
  +{}  { \PrintEditorsB}              {editor}
  +{,} { \PrintTranslatorsC}          {translator}
  +{,} { \PrintContributions}         {contribution}
  +{,} { }                            {series}
  +{,} { \voltext}                    {volume}
  +{,} { }                            {publisher}
  +{,} { }                            {organization}
  +{,} { }                            {address}
  +{,} { \PrintDateB}                 {date}
  +{,} { }                            {status}
  +{}  { \parenthesize}               {language}
  +{}  { \PrintTranslation}           {translation}
  +{;} { \PrintReprint}               {reprint}
  +{.} { }                            {note}
  +{.} {}                             {transition}
  +{} { \PrintDOI}                   {doi}
  +{} { available at \url}            {eprint}
  +{}  {\SentenceSpace \PrintReviews} {review}
}

\usepackage
{hyperref}

\numberwithin{equation}{section}

\theoremstyle{plain}
\newtheorem{thm}[equation]{Theorem}
\newtheorem{cor}[equation]{Corollary}
\newtheorem{lem}[equation]{Lemma}
\newtheorem{prop}[equation]{Proposition}

\theoremstyle{definition}
\newtheorem{defn}[equation]{Definition}
\newtheorem{note}[equation]{Notation}
\newtheorem{example}[equation]{Example}

\theoremstyle{remark}
\newtheorem{rem}[equation]{Remark}

\newcommand{\ZZ}{\mathbb{Z}}
\newcommand{\FF}{\mathbb{F}}

\newcommand{\QQ}{\mathbb{Q}}
\newcommand{\NN}{\mathbb{N}}

\newcommand{\CC}{\mathbb{C}}

\newcommand{\ma}{\mathcal{A}}
\def\A{\ma}
\newcommand{\mb}{\mathcal{B}}

\newcommand{\mh}{\mathcal{H}}
\def\H{\mh}

\newcommand*{\nb}{\nobreakdash}
\newcommand*{\Star}{\(^*\)\nobreakdash-}

\newcommand{\Cst}{\mathrm{C}^*}

\newcommand{\idealin}{\mathrel{\triangleleft}} 
\newcommand*{\Bound}{\mathcal{B}}
\newcommand*{\coloneqq}{\mathrel{\vcentcolon=}}
\newcommand{\Hilm}[1][E]{\mathcal{#1}}
\newcommand{\Toep}{\mathcal{T}}
\newcommand{\Toepr}{\mathcal{T}_\lambda}
\newcommand{\Toepu}{\mathcal{T}_u}

\newcommand*\xbar[1]{%
   \hbox{%
     \vbox{%
       \hrule height 0.5pt 
       \kern0.5ex
       \hbox{%
         \kern-0.1em
         \ensuremath{#1}%
         \kern-0.1em
       }%
     }%
   }%
} 

\DeclarePairedDelimiterX{\braket}[2]{\langle}{\rangle}{#1\,\delimsize\vert\,\mathopen{}#2}
\DeclarePairedDelimiterX{\BRAKET}[2]{\langle}{\rangle}{\!\delimsize\langle#1\,\delimsize\vert\,\mathopen{}#2\delimsize\rangle\!}

\DeclarePairedDelimiterX{\setgiven}[2]{\{}{\}}{#1\,{:}\,\mathopen{}#2}

\newcommand{\thmref}[1]{Theorem~\textup{\ref{#1}}}
\newcommand{\secref}[1]{Section~\textup{\ref{#1}}}
\newcommand{\proref}[1]{Proposition~\textup{\ref{#1}}}
\newcommand{\lemref}[1]{Lemma~\textup{\ref{#1}}}
\newcommand{\corref}[1]{Corollary~\textup{\ref{#1}}}

\newcommand{\defref}[1]{Definition~\textup{\ref{#1}}}

\def\W{\mathcal W}
\def\inv{^{-1}}

\def\iqs{Q}
\def\J{\frak J}
\def\I{\mathcal I}
\def\T{\mathcal T}

\def\OO{\mathcal O}

\def\mc{\mathcal C}

\def\dn{\sqrt[3]{19}}
\def\clsp{\operatorname{\overline {span}}}
\def\lsp{\operatorname{span}}

\def\ma{\frak A}

\def\jj{j}
\def\1{\mathbbm 1}

\def\mfc{\mathfrak c}
\def\mfp{\mathfrak p}

\def\mfa{\mathfrak a}

\begin{document}
\title[Toeplitz algebras of semigroups]{Toeplitz algebras of semigroups}

\author{Marcelo Laca}

\address{Department of Mathematics and Statistics, University of Victoria, Victoria, BC V8W 2Y2, Canada}
\thanks{This research was partially supported by the Natural Sciences and Engineering Research Council of Canada, 
Discovery Grant RGPIN-2017-04052}
\email{laca@uvic.ca}

\author{Camila F. Sehnem}

\address{School of Mathematics and Statistics, Victoria University of Wellington, P.O. Box 600, Wellington 6140, New Zealand.}

\thanks{C.F. Sehnem was supported by the Marsden Fund of the Royal Society of New Zealand, grant No. \!\!18-VUW-056}

\email{camila.sehnem@vuw.ac.nz}

\date{17 January 2021, minor changes 12 May 2022}

\begin{abstract} 
To each submonoid $P$ of a group we associate a universal Toeplitz $\Cst$-algebra $\Toepu(P)$ defined via generators and relations; $\Toepu(P)$ is a quotient of  Li's semigroup $\Cst$\nb-algebra $\Cst_s(P)$ and they are isomorphic iff $P$ satisfies independence. We give a partial crossed product realization of $\Toepu(P)$ and  show that
several results known for $\Cst_s(P)$ when $P$ satisfies independence are also valid for $\Toepu(P)$ when independence fails.
At the level of the reduced semigroup $\Cst$\nb-algebra $\Toepr(P)$, we show that nontrivial ideals have nontrivial intersection with the reduced crossed product of the diagonal subalgebra by the action of the group of units of~$P$, generalizing a result of  Li for monoids with trivial unit group. We characterize when the action of the group of units is topologically free, in which case a representation  of $\Toepr(P)$ is faithful  iff  it is jointly proper. This yields a uniqueness theorem that generalizes and unifies several classical results. We provide a concrete presentation for the covariance algebra of the product system over $P$ with one-dimensional fibers in terms of a new notion of foundation sets of constructible ideals. We show that the covariance algebra is a universal analogue of the boundary quotient and give conditions on~$P$ for the boundary quotient to be purely infinite simple. We discuss applications to a numerical semigroup and to the $ax+b$-monoid of an integral domain. This  is particularly interesting in the case of 
nonmaximal orders in  number fields, for which we show independence always fails.
\end{abstract}
\maketitle

\section{Introduction}
Let $P$ be a submonoid of a group and consider the left regular representation $L\colon P \to \mb(\ell^2(P))$ 
determined by $L_p\delta_x = \delta_{px}$ on the usual orthonormal basis $\{\delta_x \mid x\in P\}$ of $\ell^2(P)$. 
The operators $L_p$ associated to $p\in P$ are isometries and generate the Toeplitz $\Cst$\nb-algebra $\Toepr(P)$.
The spatial nature of $\Toepr(P)$ provides many useful  tools for its study, such as the existence of a faithful conditional expectation onto a diagonal subalgebra. 
 But along with this comes the notoriously difficult problem of estimating norms of operators, in this case, of polynomials on the generating isometries and their adjoints. 
 As a result, it is often quite hard to decide whether a given representation of $P$ by isometries generates 
 a homomorphic image of $\Toepr(P)$.  
 One strategy that has been successfully used to get around this problem is to characterize $\Toepr(P)$ by way of generators and relations that replicate distinguished properties of $\Toepr(P)$. 
When such a universal characterization is possible and conditions are given for faithfulness of the resulting representations, one has what is known as a uniqueness theorem.  Examples include celebrated theorems of  Coburn for the natural numbers \cite{Cob}, of Douglas for discrete submonoids of the additive reals \cite{Dou}, and of Cuntz for the free semigroup $\FF_n^+$ \cite{Cun:On}, as well as several  generalizations, see e.g. \cites{Salas:commrange,Nica:Wiener--hopf_operators,LACA1996415,Crisp-Laca,NNN1}.

It has long been clear that the universal $\Cst$\nb-algebra for  isometric representations of a monoid  is often too large to be of much use. Indeed, as Murphy  observed in \cite{Mur:RMJM}, already in the case of two commuting isometries, that is, for representations of $\NN^2$, one obtains a nonnuclear  universal $\Cst$\nb-algebra. In many of the aforementioned situations, however,  it is possible to identify extra conditions that distinguish a specific class of isometric representations and give tractable $\Cst$\nb-algebras.

Substantial progress along this path took place in the early 90's when Nica introduced a $\Cst$\nb-algebra  $\Cst(G,P)$ that is universal for a class of `covariant' representations of a quasi-lattice ordered group $(G,P)$. These are pairs $(G,P)$ consisting of a submonoid $P$ of a group $G$ such that $P \cap P\inv = \{e\}$ and for which the intersection $xP \cap yP$ of any two cones with vertices in $G$ is either empty or equal to another cone; this condition is not exactly Nica's definition of quasi-lattice order, but is equivalent to it, see \cite[Definition~6 and Lemma~7]{CL-JAustralMS}. The covariance relation used by Nica 
stemmed from the observation that the multiplication of range projections of the generating isometries in the left regular 
representation replicates the intersection of cones with vertices in $P$.
Nica showed that  $\Cst(G,P)$ is actually isomorphic to  $\Toepr(P)$ in many interesting cases, including the $\Cst$\nb-algebra generated by~$k$  \Star commuting isometries, which is  universal for Nica covariant representations of $\NN^k$, and the Toeplitz--Cuntz algebra $\T \OO_n$ which is universal for the free monoid $\FF_n^+$. 
 
In general if $P$ is a submonoid of a group $G$, the set 
of cones having vertices in $P$  does not have to be closed under intersection, so Nica-covariance cannot be imposed {\em ipsissimis verbis}. Thus, when the interest in monoids that are not quasi-lattice ordered began to surge in the mid 2000's, sparked mainly by semigroups arising from algebraic number theory, it became clear that a new  idea was needed to go beyond the {\em ad hoc} analysis of examples. The major breakthrough was achieved by Li in \cite{Li:Semigroup_amenability,Li:Semigroup_nuclearity}, and was later summarized and extended in \cite[Chapter 5]{CELY}. Motivated by what happens for $ax+b$\nb-semigroups of algebraic integers, Li's insight was to realize that the idea of mirroring the behaviour of projections in the left regular representation does carry over to the general situation. The key to this was to replace the cones by a collection of subsets of $P$ that he called {\em constructible right ideals}, which form a semilattice under intersection. 

We pause briefly to carry out a concrete computation that illustrates how the constructible right ideals arise naturally, and what their role is in the definition of Li's semigroup $\Cst$\nb-algebra; this computation also motivates the notation we introduce in \secref{sec:neutralwords}. Assume throughout that $P$ is a submonoid of a group and suppose $ p,q,r,s\in P$ satisfy $p\inv q r\inv s = e$. 
Let us work out what the product $L_p^* L_q L_r^* L_s$ really is in the left regular representation by computing it on a basis vector~$\delta_x$  of  the standard orthonormal basis of $\ell^2(P)$. Clearly
$L_p^* L_q L_r^* L_s \delta_x= L_p^*  L_q L_r^*  \delta_{sx} $ vanishes unless $sx \in rP$, in which case
$L_p^*  L_q L_r^*  \delta_{sx} = L_p^*  \delta_{qr\inv sx}$. This, in turn, vanishes unless $qr\inv sx \in pP$, in which case
$L_p^* L_q L_r^* L_s \delta_x =\delta_{p\inv qr\inv sx} =\delta_x$, because $p\inv qr\inv s=e$.
Thus, 
\[L_p^* L_q L_r^* L_s \delta_x = \begin{cases} \delta_x &\text{ if } x\in P \cap s\inv r P \cap s\inv r q\inv pP,\\
0 & \text{otherwise;}
\end{cases}
\]
in other words, $L_p^* L_q L_r^* L_s$ is (the operator of multiplication by) the characteristic function of the set  $P \cap s\inv r P \cap s\inv r q\inv pP$, which is a typical constructible right ideal. To define the semigroup $\Cst$\nb-algebra of a submonoid of a group, \cite[Definition 3.2]{Li:Semigroup_amenability}, Li associates an isometry $v_p$ to each $p\in P$ and a projection $e_S$ to each constructible right ideal $S$,
and then imposes  relations that say that the isometries $v_p$ multiply as the elements of $P$ and that a product
such as $v_p^* v_q v_r^* v_s$ is the projection $e_S$ where $S$ is the constructible right ideal $P \cap s\inv r P \cap s\inv r q\inv pP$ mentioned above, see \defref{def:cstar_s(P)} for the precise statement. 
The $\Cst$\nb-algebra with this presentation has many nice features, and is isomorphic to $\Toepr(P)$ in several interesting examples, notably those arising from $ax+b$\nb-semigroups of algebraic integers. However,  in addition to the unavoidable issue of amenability, the construction is fully satisfactory  only when  $P$ satisfies independence, equivalently, when the characteristic functions of ideals are linearly independent. 

Here we introduce a universal Toeplitz algebra $\Toepu(P)$, for  $P$  a submonoid of a group, that works just as well even if $P$ does not satisfy independence. The presentation of $\Toepu(P)$ includes extra relations that only apply when independence fails, and we show that $\Toepu(P)$ is the quotient of Li's semigroup $\Cst$\nb-algebra under a canonical homomorphism whose kernel reflects the failure of independence. We also relate  $\Toepu(P)$ to various other $\Cst$\nb-algebras associated to semigroups, giving, in particular, a partial crossed product realization. Our main results are a characterization of  faithful representations of $\Toepr(P)$ in terms of the action of units on the diagonal and  a generalized uniqueness theorem for the $\Cst$\nb-algebra generated by a collection of elements satisfying the presentation 
of $\Toepu(P)$. Furthermore, we give a presentation of the covariance algebra of the one-dimensional product system over~$P$
that allows us to realize it as a universal boundary quotient. We also give conditions on~$P$ that are equivalent to topological freeness of the partial action on the boundary, and so are sufficient for the boundary quotient to be purely infinite simple when~$P$ is nontrivial.
In the setting of submonoids of groups, our results represent  significant improvements and in some cases conceptual simplifications of earlier work on $\Cst$\nb-algebras of monoids that have trivial unit group or satisfy independence \cite{CELY}, and of right LCM monoids  \cite{NNN1}.
Throughout our work, we make a point of formulating the presentations and characterizations in terms of the original data, with an eye towards direct applications.

We  describe next the contents of this paper, highlighting the main results along the way.
 In \secref{sec:neutralwords}  we review the set of constructible right ideals, which we view as the range of a map 
 $\alpha \mapsto K(\alpha)$ from words to ideals, through which, e.g., the ideal $P \cap s\inv r P \cap s\inv r q\inv pP$ corresponds to the word $( p,q,r,s)$. We also state explicitly and in detail the properties of this map that are needed in the remaining sections.
In \secref{sec:toeplitzalgebras} we begin by showing that if $v\colon P\to B$ is a unital map from~$P$ to a $\Cst$\nb-algebra~$B$ and the products such as  $v_p^* v_q v_r^* v_s$ only depend  on the constructible ideal associated to the word $(p,q,r,s)$, rather than on the word itself, then $p \mapsto v_p$ is an isometric representation satisfying  important additional properties,  \proref{pro:propertiesofsharpgenerators}. We give the presentation of the universal Toeplitz algebra $\Toepu(P)$ in \defref{def:toeplitz-semigroup},  and we then proceed to examine it in  relation to the reduced Toeplitz algebra $\Toepr(P)$. We establish that $\Toepr(P)$ is a canonical quotient of $\Toepu(P)$, \proref{prop:diagonal-sub}, and in \proref{pro:faithfulonD_u} we show that  a representation of $\Toepu(P)$ is faithful on the diagonal subalgebra  $D_u$ if and only if it satisfies a joint properness condition, \defref{def:jointlyproper}. We conclude that $D_u$ is always  canonically isomorphic to $D_r$, the diagonal in $\Toepr(P)$, \corref{cor:faithfulondiagonal}. We complete \secref{sec:toeplitzalgebras} by showing that a proper subset of the relations defining $\Toepu(P)$ gives the presentation of Li's semigroup $\Cst$\nb-algebra, \proref{pro:equivalentcovs}. This leads to \corref{cor:canonicalSharptoLimap} where we show that $\Toepu(P)$ is a canonical quotient of $\Cst_s(P)$; the two are isomorphic if and only if~$P$ satisfies independence.

In \secref{sec:partialactions} we give a brief introduction to partial actions and their crossed products and review the reduced partial crossed product realization of $\Toepr(P)$ given by Li in \cite{CELY}. We do this explicitly in terms of constructible right ideals, bypassing the use of inverse semigroups. The main result here is \thmref{thm:partial-picture}, where we show that $\Toepu(P)$ is isomorphic to the full partial crossed product of the partial action of $G$ on the diagonal algebra. This allows us to verify that the improvements predicted by Li in \cite[Remark  5.6.46]{CELY} for the full partial crossed product are realized by our universal Toeplitz $\Cst$\nb-algebra, see \thmref{thm:modifiedresult} and \corref{cor:ifamenallequiv}.

In \secref{sec:faithfulness} we study conditions that ensure that a representation of the reduced Toeplitz algebra $\Toepr(P)$ is faithful. We notice first that $\Toepr(P)$ has a copy of the reduced crossed product of the diagonal by the restriction of the partial action of $G$ to the group of units $P^*$.  
The first main result, \thmref{thm:faithfulreps},  shows that a representation of $\Toepr(P)$
is faithful if and only if  its restriction to the crossed product $D_r\rtimes_{\gamma,r} P^*$ is faithful.  Setting $P^* =\{e\}$ recovers \cite[Corollary~5.7.3]{CELY} and assuming $P$ is a right LCM monoid  gives a stronger version of the faithfulness result  \cite[Theorem~7.4]{NNN1}, albeit, under the extra assumption that $P$ embeds in a group.
The second main result in this section is \thmref{thm:idealsifftopfree}, where we show that the partial action of $G$ on $D_r$ is topologically free if and only if the restricted action of $P^*$ is topologically free.  Using a recent result of Abadie--Abadie \cite[Theorem 4.5]{AbaAba} we then conclude that the action of $P^*$ is topologically free if and only if every ideal of $\Toepu(P)$ that has trivial intersection with $D_u$ is contained in the kernel of the canonical map $\Toepu(P) \to \Toepr(P)$. 
 We also give there a criterion in terms of constructible ideals to decide whether the action of $P^*$ on the diagonal is topologically free.
As a consequence,  when the action of $P^*$ is topologically free,  faithfulness of representations of $\Toepr(P)$ is decided by their restrictions to the diagonal, \corref{cor:topfreeimpliesideals}. When we combine these results with the  faithfulness criteria for representations of the diagonal from \secref{sec:toeplitzalgebras}, we obtain \thmref{thm:uniqueiftopfreejointproper}, the third main result of this section, which shows that if $\Toepu(P) \to \Toepr(P)$ is an isomorphism, then $\Toepr(P)$ is the unique $\Cst$\nb-algebra generated by a jointly proper semigroup of isometries satisfying the presentation of  $\Toepu(P)$.

We begin \secref{sec:full-boundary} by showing that there is a canonical homomorphism of 
$\Toepu(P)$ onto the covariance algebra  $\CC\times_{\CC^P}P$ of
the canonical product system over $P$ with one-dimensional fibers from \cite{SEHNEM2019558}.
In \lemref{lem:charac-projections} we identify projections in the diagonal that are in the kernel of this map; these lead  naturally to a  notion of foundation sets, generalizing those 
introduced by  Sims and Yeend in \cite{Sims-Yeend:Cstar_product_systems} for quasi-lattice orders.
In  \corref{cor:generating-projections} we arrive
at an explicit presentation of the covariance algebra in which the extra relations 
are a natural augmentation of the presentation of $\Toepu(P)$ by the new foundation sets. 
This is not a coincidence: the original motivation for our presentation of $\Toepu(P)$ was 
the view that the covariance algebra from \cite{SEHNEM2019558} had to be the universal  boundary quotient of an appropriately defined universal Toeplitz algebra. 
In \proref{prop:non-foundation-sets} we show that the presentation of $\CC\times_{\CC^P}P$ in terms of generalized foundation sets is maximal in the sense that if one imposes  further relations associated to other sets, then the resulting $\Cst$\nb-algebra is trivial.
The main result of \secref{sec:full-boundary} is \thmref{thm: several-characteriz}, where we establish several equivalent descriptions of $\CC\times_{\CC^P}P$. In particular, we show that $\CC\times_{\CC^P}P$ is isomorphic to $\mathrm{C}(\partial\Omega_P)\rtimes G$, where $\partial\Omega_P$ is the boundary of $\Omega_P$ as defined in \cite[Definition~5.7.8]{CELY}. At this point we refer to $\CC\times_{\CC^P}P$ as the full boundary quotient of $\Toepu(P)$. As immediate consequences of our analysis of boundary quotients,  we derive in \corref{cor:qu-isomor} a criterion for when the full boundary quotient is $\Toepu(P)$ itself and then we characterize topological freeness of the partial action of $G$ on the boundary of the diagonal, \thmref{thm:equivalenttopfree}. Our second main result of the section is the characterization of purely infinite simple boundary quotients, \corref{cor:charac-inf-simple}.

In the final four sections we discuss  several classes of examples that illustrate the range of application of our results.
In \secref{sec:numerical} we discuss a specific numerical semigroup studied by Raeburn and Vittadello in \cite{Rae-Vit}, giving a characterization of the left regular $\Cst$-algebra in terms of generators and relations and faithfulness criteria for representations. In \secref{sec:integral} we adapt first our criteria for topological freeness to the $ax+b$-monoid of an integral domain, which we then apply in the following section to orders in number fields. Then we give a direct proof that the associated boundary quotient is purely infinite simple using \corref{cor:charac-inf-simple}, a fact that can also be derived from earlier work of Cuntz and Li \cite{C-LiClay2010}, see also \cite{Li:RingC*}. In \secref{sec:orders} we discuss the $ax+b$-monoids of orders in algebraic number fields. We first obtain a uniqueness theorem for their left regular $\Cst$-algebras generalizing earlier results for rings of algebraic integers \cite{CDLMathAnn2013}. We then prove that the independence condition fails for the multiplicative and $ax+b$-monoids of all nonmaximal orders, establishing them as a rich source of new examples. We also show how our presentation applies in the concrete case of $\ZZ[\sqrt{-3}]$. Finally, in \secref{sec:rightLCM} we show that our results lead to the simplification and strengthening, for  right LCM submonoids of groups, of recent results of Brownlowe, Larsen and Stammeier on uniqueness, simplicity  and pure infiniteness of semigroup $\Cst$\nb-algebras \cite{NNN1}.

\subsection*{Acknowledgments:} This project was started at the workshop {\em Cuntz--Pimsner Cross-Pollination} at the Lorentz Center in Leiden, and we would like to thank the organizers and the institute for providing the opportunity and a wonderful environment for creative interaction.  We are also very grateful for the hospitality of the departments of mathematics at Victoria and  Florian\'opolis during visits in which part of this research was carried out.

\section{Neutral words, quotient sets, and constructible ideals}\label{sec:neutralwords}
We recall here the basic facts leading to the constructible right ideals introduced  in \cite{Li:Semigroup_amenability}. 
Our approach and the notation we use are inspired by those of \cite{SEHNEM2019558}.
Let $P\subset G$ be a submonoid of a group~$G$, and for each $k\in \NN$ consider the set of words of length $2k$ in $P$,
\[
\W(P)^k \coloneqq \{ (p_1, p_2, \cdots, p_{2k-1}, p_{2k}) \mid p_j \in P,  \text{ for } j = 1, 2, \cdots, 2k\}.
\]
When the context makes it clear what  $P$ is,  we write simply $\W$ instead of $\W(P)$. 
By convention we write $\W^0 = \{\emptyset\}$. Using concatenation of words as composition law, it is easy to see that $\W^k \W^l = \W^{k+l}$.
We define  a {\em generalized (iterated) left quotient map}, assigning an element of  $G$ to each word $\alpha \in \W^k$,  
 by 
\[
\alpha = (p_1, p_2, \cdots, p_{2k-1}, p_{2k})\  \longmapsto \ \dot\alpha\coloneqq p_1\inv p_2 \cdots p_{2k-1}\inv p_{2k}.
\]
 
 Thus, $\dot{\W}^k \subset G$ is the set of products of $k$ left quotients of elements of $P$.  Again by convention, we write $\dot\W^0 = \{e\}$, 
 and note that $(\alpha\beta)\dot{} = \dot\alpha\dot\beta$.
For each $\alpha \in \W^k$ we define the reverse word $\tilde\alpha \coloneqq (p_{2k}, p_{2k-1}, \cdots  p_2, p_1) \in \W^k$. It is easy to see that  $\dot{\tilde{\alpha}} = (\dot\alpha)\inv$. We shall say that a word $\alpha$ is {\em neutral} if $\dot \alpha = e$.

\begin{lem} \label{lem:wordsandqotients} 
Suppose $P$ is a submonoid of a group $G$, and assume $P$ generates $G$. Then
$\{e\} \subseteq \dot{\W^1} \subseteq \cdots \subseteq \dot{\W}^k \subseteq \dot{\W}^{k+1} \cdots$
and $\bigcup_k \dot{\W}^k  = G$.

\begin{proof} Since $e\in P$  by assumption, for each $\alpha \in \W^k$ the concatenation $\alpha(e,e)$ is  in $\W^{k+1}$ and satisfies $(\alpha(e,e))\dot{} = \dot\alpha$, proving that $\dot\W^k \subseteq \dot\W^{k+1}$. 
Suppose now $\alpha \in \W^k$ and $\beta \in \W^l$ so that  $\dot\alpha \in \dot\W^k$ and $\dot\beta \in \dot\W^l$. Then 
$\dot\alpha \dot\beta = (\alpha \beta)\dot{} \in \dot\W^{k+l}$, and also $\dot\alpha\inv = \dot{\tilde{\alpha}} \in \dot \W^k$. This shows that the subset $\bigcup_k \dot{\W}^k $ of $G$ contains the products and  inverses of its elements, hence is a subgroup of $G$.
Since  $(e,p)\dot{}= p$ for every $p\in P$ we have
$P \subseteq \dot{\W}^1 \subseteq \bigcup_k \dot{\W}^k $, and clearly any subgroup of $G$ that contains $P$ must contain $\bigcup_k \dot{\W}^k $.
\end{proof}
\end{lem}

\begin{lem}  If $\dot \W^k = \dot{\W}^{k+m}$ for some $m\geq1$, then $\dot \W^k = \dot\W^{k+m}$ for every $m\geq1$, and $\dot \W^k$ is a group. 

\begin{proof}By \lemref{lem:wordsandqotients} we see immediately that $\dot \W^k = \dot{\W}^{k+j}$ for every $j = 1, 2, \ldots, m$. An easy induction argument shows that
the sequence remains constant after $k+m$ too, because
$\dot{\W}^{k+m+1} = \dot \W^{k+1}  \dot{\W}^{m} = \dot \W^{k}  \dot{\W}^{m} = \dot \W^{k+m}  =\dot{\W}^{k}$.
\end{proof}
\end{lem}

\begin{rem}
Recall the celebrated  result of Ore stating that a cancellative monoid $P$ embeds in a group $G$ in such a way that
$ G = P\inv P$  if and only if  $PP\inv \subseteq P\inv P$, namely if and only if every right quotient can be written as a left quotient. Such semigroups $P$ are called right reversible. 
In this case our sequence $(\dot\W^k)_{k\in \NN}$ stabilizes at the first step because
$(P\inv P) (P\inv  P) = P\inv (P P\inv) P \subseteq P\inv (P\inv P) P =P\inv P$. 
Our choice to start the quotients with $P\inv$ introduces an asymmetry. Indeed, 
when $P$ is left reversible, namely, when $P\inv P \subset P P\inv$ it is clear that $G = PP\inv$. But our sequence $\W^k$ stabilizes at the second step because $(P\inv P)( P\inv P)( P\inv P) = P\inv (P P\inv) P P\inv P = P\inv (P\inv P) P P\inv P  = P\inv  P P\inv P $, and thus we would write $G = P\inv  P P\inv P$.
\end{rem}

\begin{defn}  For each word  $\alpha = (p_1, p_2, \ldots, p_{2k}) \in \W^k$ we define the {\em iterated quotient set} 
of $\alpha$  to be the set
\[
\iqs(\alpha) \coloneqq \{ e, \  p_{2k}\inv  p_{2k-1} , \,   p_{2k}\inv  p_{2k-1} p_{2k-2}\inv  p_{2k-3}, \, \ldots, \, p_{2k}\inv  p_{2k-1} p_{2k-2}\inv  p_{2k-3} \cdots p_2\inv p_1  \},
\]
where the last element listed is $\dot{\tilde\alpha} = (\dot\alpha)\inv$. 
We will often drop the  `iterated' and simply say `quotient set'. The apparent reversal of $\alpha $ in defining the iterated quotients
may seem unmotivated at first but is better adapted to the partial actions that will appear later in \secref{sec:partialactions}.  
If we need to refer to the analogous iterated quotient set taken from left to right, we will just use $\iqs(\tilde\alpha) \coloneqq\{e, \,p_1\inv p_2,\, \ldots,\,  p_1\inv p_2 \cdots p_{2k-1}\inv p_{2k}\}$. Here the last element is $\dot\alpha$. 
\end{defn}

\begin{lem}\label{lem:propertiesofquotientsets}  
Suppose $\alpha \in \W^k$ and $\beta \in \W^l$, and, as before, denote by $\tilde\alpha$ the reverse of $\alpha$ and by $\alpha\beta \in \W^{k+l}$  the concatenation of $\alpha$ and $\beta$. Then
\begin{enumerate}
\item  $\iqs(\tilde\alpha) = \dot{\alpha} \iqs(\alpha)$;

\smallskip\item $\iqs(\beta\alpha) =  \iqs (\alpha) \cup \dot{\tilde\alpha} \iqs (\beta) $;

\smallskip\item $\iqs(\tilde\alpha \alpha) = \iqs (\alpha)$;
\end{enumerate}
in particular, if $\dot\alpha =e $, then
\begin{enumerate}\setcounter{enumi}{3}
\item  $\iqs(\tilde\alpha) =  \iqs(\alpha)$; and
\smallskip\item $\iqs(\beta\alpha) = \iqs (\alpha) \cup  \iqs (\beta)  $.

\end{enumerate}
\end{lem}
\begin{proof}
Let $\alpha = (p_1, p_2, \ldots, p_{2k})$. 
Multiplying every element of 
\[
\iqs(\alpha)= \{ e, \  p_{2k}\inv  p_{2k-1} , \,   p_{2k}\inv  p_{2k-1} p_{2k-2}\inv  p_{2k-3}, \, \ldots, \, p_{2k}\inv  p_{2k-1} p_{2k-2}\inv  p_{2k-3} \cdots p_4\inv p_3, \, \dot{\tilde\alpha}  \} 
\]
 on the left by  $\dot{\alpha} =p_1\inv p_2 p_3\inv p_4 \cdots p_{2k-1}\inv p_{2k} $ and simplifying each product, gives 
\[ 
\dot\alpha \iqs(\alpha) = \{ \dot \alpha , p_1\inv p_2 \cdots p_{2k-2}, \ldots ,\  p_1\inv p_2 , \, e \}
\]
which is precisely the set $\iqs(\tilde\alpha)$ listed in reverse. This proves (1). 

 Now let $\beta = (q_1, q_2, \ldots, q_{2l})$. It is easy to see that the first $k$ iterated quotients in 
$\iqs( \beta\alpha )$ are precisely those of $\alpha$, with the last one being $\dot{\tilde\alpha}$, and the following $l$ iterated quotients are those of $\beta$ multiplied by $\dot{\tilde\alpha}$ on the left. This gives (2).

In order to prove (3), set $\beta = \tilde\alpha$, then use first (2) and then (1) to compute
\[
\iqs(\tilde\alpha \alpha) = \iqs (\alpha)  \cup \dot{\tilde\alpha} \iqs (\tilde\alpha) = \iqs (\alpha) \cup  \dot{\tilde\alpha} \dot \alpha \iqs (\alpha)
\] 
which proves (3) because $ \dot{\tilde\alpha} \dot\alpha= e$. Assertions (4) and (5) follow immediately from (1) and (2).
 \end{proof}

 Notice that neutral words suffice to generate all quotient sets. Indeed, by \lemref{lem:propertiesofquotientsets} (3), we may substitute $\alpha$ with the neutral word $\tilde\alpha \alpha$ without changing the iterated quotient set. 

There is a left action of $P$ on words:  if $p\in P$ and  $\alpha \in \W^k$, then $p\alpha \coloneqq (pp_1,pp_2, \ldots pp_{2k})$. This satisfies 
$(p\alpha)\dot{} = \dot\alpha$, \ $(p\alpha)\tilde{} = p \tilde\alpha$, and \   $\iqs(p\alpha) = \iqs(\alpha)$.
Following \cite{SEHNEM2019558}, given a finite subset  $F\subseteq G$, we set 
\[
K_F\coloneqq\underset{g\in F}{\bigcap}gP.
\]
With the usual notation for multiplication of group elements and sets, 
we have $gK_F = K_{gF}$.  We will be interested in the sets $K_F$ arising from taking $F $ to be the iterated quotient set of a word;
this produces subsets  of $P$ because  $e\in \iqs(\alpha)$ for every $\alpha$. In order to lighten the notation
 we will write $K(\alpha)$ instead of $K_{\iqs(\alpha)}$ when possible, so that 
\[
 K(\alpha) \coloneqq  P\cap (p_{2k}\inv p_{2k-1}) P \cap (p_{2k}\inv p_{2k-1}p_{2k-2}\inv p_{2k-3}) P \cap \cdots \cap (\dot{\tilde\alpha}) P.
\]
For further reference we list the following properties of $K(\alpha)$ in relation to concatenation and reversal of words.
\begin{prop}[cf. Section 2.1 of \cite{Li:Semigroup_amenability}]\label{pro:propertiesKalpha}
If $\alpha $ and $\beta$ are in $\W$, then
\begin{enumerate}
\item $ K(\tilde\alpha) = \dot\alpha  K(\alpha)$;
\item $ K( \beta \alpha) =  K(\alpha) \cap \dot {\tilde\alpha}  K(\beta) $;
\item $ K( \tilde\alpha \alpha) =  K( \alpha)$;
\end{enumerate}
in particular, if $\dot\alpha =e$, then
\begin{enumerate}\setcounter{enumi}{3}
\item $  K(\tilde\alpha)=  K(\alpha)$; and
\item $ K( \beta\alpha) =    K(\alpha) \cap K(\beta)  $;
\end{enumerate}
and if, instead,  $\dot{\beta} =e$ and $p\in P$, then
\begin{enumerate}\setcounter{enumi}{5}
\item  $ K((e,p)\beta (p,e))= K(\beta (p,e)) = pK(\beta)$
\item $ K((p,e)\beta (e,p)) = K(\beta (e,p)) = P \cap p\inv K(\beta) $;
\end{enumerate}
\begin{proof} The verification of the first five items is by  straightforward application of  the corresponding properties of the iterated quotient sets from \lemref{lem:propertiesofquotientsets}. We show next how to derive the last two by repeated applications of item (2). Notice first that $K(\beta(p,e)) \subset K((p,e)) = P \cap e\inv pP = pP$ and that $K((e,p)) = P \cap p\inv e P = P$. Then  
\begin{multline*}
 K((e,p)\beta (p,e)) = K(\beta(p,e)) \cap (\beta(p,e))\dot{\tilde{}} \, K((e,p))  =K(\beta(p,e)) \cap p P = \\
 =K(\beta(p,e))  = K(p,e)\cap pK(\beta) 
  = pP \cap pK(\beta) =  p K(\beta), 
\end{multline*} 
 proving (6).
Similarly,
\begin{multline*}
K((p,e)\beta (e,p)) = K(\beta(e,p)) \cap (\beta(e,p))\dot{\tilde{}} \, K((p,e)= K(\beta(e,p)) \cap p\inv  (pP) \\ = K(\beta(e,p)) = K(e,p) \cap p\inv K(\beta)  = P\cap p\inv K(\beta),
\end{multline*} proving (7).
\end{proof}
  \end{prop}
  We record here for later use the following easy consequence of the proposition.
  \begin{cor}
  Suppose $\alpha$ and $\beta$ are words with $\beta$ neutral, then 
$
  K(\alpha \beta \tilde\alpha) = K(\tilde\alpha) \cap \dot\alpha K(\beta)
$.
\end{cor}

The map $\alpha \mapsto K(\alpha)$ is far from injective, but it still  provides a convenient way to parametrize constructible right ideals of $P$ in terms of neutral words. Thus for every monoid $P$ we let  
\[
 \J(P) \coloneqq\{K(\alpha) \mid \alpha \in \W(P)\},
 \]
 dropping the reference to $P$ and writing simply~$\J$ when there is no risk of confusion.  
 We see next that $\J$ is equal to the set of all constructible right ideals as introduced  in Section 2.1 of \cite{Li:Semigroup_amenability}, modulo having to add the empty set, which is not of the form $K(\alpha)$ when $P$ is left reversible. The properties of $K(\alpha)$ listed in \proref{pro:propertiesKalpha} allow us to give direct proofs, and to show that~$\J$ is automatically closed under finite intersections, a fact established in Section 3 of \cite{Li:Semigroup_amenability}.  

\begin{prop}[cf. \cite{Li:Semigroup_amenability}]\label{pro:characterizationconstructibleideals}
The collection 
$ \J =\{K(\alpha) \mid \alpha \in \W\} $
of subsets of $P$ satisfies 
\begin{enumerate}
\item $\J  =  \{K(\alpha) \mid \alpha \in \W, \dot\alpha = e\}
 = \{K(\tilde \alpha \alpha) \mid \alpha \in \W\}$;
\item $P   \in \J$;
\item $K(\alpha) \cap K(\beta) \in \J$ for every $\alpha,\beta\in \W$;
\item $K(\alpha) p \subset K(\alpha)$ for every $\alpha\in W$ and $p\in P$ ($K(\alpha)$ is a right ideal in $P$);
\item $p K(\alpha)  \in \J$ for every $\alpha\in \W$ and $p\in P$; and 
\item $P \cap p\inv K(\alpha) \in \J$ for every $\alpha$ and $p\in P$. 

\end{enumerate}
Moreover, $\J$ is the smallest collection of subsets of $P$ that contains $P$ 
and  is invariant under the left actions of $P$ and of $P\inv$ as in parts \textup{(5)} and  \textup{(6)}. 

\begin{proof}
Since $( \tilde\alpha\alpha) \dot{} =\dot{\tilde \alpha }\dot{\alpha} =  e$ it is obvious that $\J  \supseteq  \{K(\alpha) \mid \alpha \in \W, \dot\alpha = e\}
 \supseteq \{K(\tilde \alpha \alpha) \mid \alpha \in \W\}$. In order to show that $\J  \subseteq \{K(\tilde \alpha \alpha) \mid \alpha \in \W\}$, 
suppose $\alpha \in \W^k$ 
and recall that  $K(\alpha) = K(\tilde\alpha\alpha)$ by \proref{pro:propertiesKalpha}(3). This proves part (1).
 
Part (2) is obvious because $P = K((p,p))$ for each $p\in P$, and part (3) follows easily from \proref{pro:propertiesKalpha}(5) since
 part (1) allows us to work with neutral words.

Part (4) follows from the observation that  the factor $p$ is absorbed by $P$ on the right in each term of the intersection
that defines  $K_F$.

Parts (5) and (6) now follow directly from \proref{pro:propertiesKalpha}(6) and (7), respectively.
 
Up to this point we have verified that $\J$ is a collection of right ideals in $P$ that contains $P$, is closed under intersections, and is invariant under the left actions by $P$ and $P\inv$ given in parts (5) and (6). 

To complete the proof we need to show that $\J$ is contained in any collection of subsets of $P$ that contains $P$ and is invariant under the left actions by $P$ and $P\inv$ given in parts (5) and (6).  This is done by an easy induction argument based upon  rewriting $K(\alpha)$ as
\[K(\alpha) =P\cap  p_{2k}\inv(p_{2k-1} (P\cap p_{2k-2}\inv (p_{2k-3}(\cdots  (P\cap p_{2}\inv (p_{1}P)) \cdots )))). \qedhere\]
 \end{proof}
 \end{prop}

We close this section with a brief discussion of what happens when  $P$ embeds into different groups;
see the argument around \cite[equation~(37)]{Li:Semigroup_amenability} and also \cite[Lemma 3.9]{SEHNEM2019558}.
One issue is that, in principle, the subset $K(\alpha)$ of $P$
could depend on the specific embedding of $P$ in a group. That this is not the case is implicit in
\proref{pro:characterizationconstructibleideals}, because the original definition of constructible ideals in
\cite{Li:Semigroup_amenability} does not use a group at all. 
Nevertheless, we wish to give a direct proof of it in the present context of submonoids of groups.

 Recall first that when $P$ embeds as a submonoid of any group, then there exists
 a universal group $ \Hilm[G](P)$ generated by a canonical copy of $P$. Every embedding $P\hookrightarrow G$ extends to a unique group homomorphism $\gamma\colon \Hilm[G](P)\to G$. 
 
\begin{lem}\label{lem:indep-group} Suppose $P$ is a submonoid of a group $G$. Let $\gamma\colon \Hilm[G](P)\to G$ be the unique group homomorphism extending $P\hookrightarrow G$.
For each $\alpha=(p_1,\ldots,p_{2k})\in\Hilm[W]^k$ we have
\begin{enumerate}
\item $K(\alpha) = \emptyset$ whenever $\gamma(\dot{\alpha}) = e$ in $G$ and $\dot{\alpha} \neq e$ in $ \Hilm[G](P)$;

\smallskip\item $\gamma K(\alpha)=K(\gamma(\alpha)) \coloneqq P\cap \gamma(p_{2k}^{-1}p_{2k-1})P\cap 
 \ldots \cap \gamma(p_{2k}^{-1}p_{2k-1}p_{2k-2}^{-1}\cdots p_2^{-1}p_1)P$, where we may identify 
 $P$ and $\gamma(P)$.
\end{enumerate}
\begin{proof} 
For the first assertion, suppose that $\gamma(\dot{\alpha})=e$ in~$G$ and $K(\alpha) \neq \emptyset$,
and let
\[
s\in K(\alpha)=P\cap p_{2k}^{-1}p_{2k-1}P\cap p_{2k}^{-1}p_{2k-1}p_{2k-2}^{-1}p_{2k-3}P \cap \ldots \cap  p_{2k}^{-1}p_{2k-1}p_{2k-2}^{-1}\cdots p_2^{-1}p_1P.
\]
In particular, $s\in \dot{\tilde \alpha }P$, so there exists a unique $t\in P$ such that $\dot{\tilde{\alpha}} t=s$ in $ \Hilm[G](P)$. Applying $\gamma$ on both sides of the equality and using $\gamma(\dot{\tilde\alpha}) = \gamma(\dot{\alpha})\inv = e$, we obtain 
 \[
 \gamma(t)=\gamma(\dot{\tilde{\alpha}})\gamma(t)=\gamma\big(\dot{\tilde{\alpha}} t\big)=\gamma(s).
 \]
  This implies that $s=t$ because $\gamma$ is injective on~$P$, which forces $\dot{\alpha} = \dot{\tilde{\alpha}}\inv= e $ in $\Hilm[G](P)$.
  
  For the second assertion, notice first that  $ \gamma(K(\alpha)) \subseteq K(\gamma(\alpha))$ is clear. Let~$r\in P$ be such that $\gamma(r)\in K({\gamma(\alpha)})$.  There is a unique~$s_1\in P$ satisfying~$\gamma(r)=\gamma(p_{2k}^{-1}p_{2k-1})\gamma(s_1)=\gamma(p_{2k}^{-1}p_{2k-1}s_1)$. This entails $\gamma(p_{2k}r)=\gamma(p_{2k-1}s_1)$ and hence $r=p_{2k}^{-1}p_{2k-1}s_1$ in~$\Hilm[G](P)$ since~$\gamma$ is injective on~$P$. But~$\gamma(r)$ also lies in $\gamma(p_{2k}^{-1}p_{2k-1}p_{2k-2}^{-1}p_{2k-3})P$. So there is a unique $s_2\in P$ so that $$\gamma(r)=\gamma(p_{2k}^{-1}p_{2k-1}p_{2k-2}^{-1}p_{2k-3})\gamma(s_2).$$ Then $\gamma(s_1)=\gamma(p_{2k-2}^{-1}p_{2k-3}s_2)$ and we conclude as above that $s_1=p_{2k-2}^{-1}p_{2k-3}s_2$ in~$\Hilm[G](P)$. Thus $$r=p_{2k}^{-1}p_{2k-1}s_1=p_{2k}^{-1}p_{2k-1}p_{2k-2}^{-1}p_{2k-3}s_2\in p_{2k}^{-1}p_{2k-1}p_{2k-2}^{-1}p_{2k-3}P.$$  Continuing with this procedure, we deduce that~$r\in K(\alpha)$. Therefore~$K(\gamma(\alpha))=K(\alpha)$ as asserted.
\end{proof}
\end{lem}

  \section{Toeplitz \texorpdfstring{$\mathrm{C}^*$}{C*}-algebras for submonoids of groups}\label{sec:toeplitzalgebras}
For each  monoid $P$ that embeds in a group we introduce here a new universal Toeplitz $\Cst$\nb-algebra $\Toepu(P)$, defined via a conceptually simple set of relations. Our choice of relations implies that the linear combinations of  projections associated to constructible ideals always behave exactly as they do in the $\Cst$\nb-algebra $\Toepr(P)$ generated by the left regular representation. 
This is particularly relevant when the underlying monoid does not satisfy independence. 
When we remove some of the relations defining $\Toepu(P)$ we obtain a new presentation of the $\Cst$\nb-algebra $\Cst_s(P)$ from \cite[Definition 3.2]{Li:Semigroup_amenability}. Hence $\Toepu(P)$ is canonically a quotient of $\Cst_s(P)$, and we will see that the two coincide if and only if~$P$ satisfies independence.

\subsection{Presentation of a universal Toeplitz \texorpdfstring{$\mathrm{C}^*$}{C*}-algebra}  The following notation will be very useful throughout.

\begin{note} When $w\colon P\to B$ is a map from~$P$ to a $\Cst$\nb-algebra~$B$ and  $\alpha=(p_1,p_2,\ldots,p_{2k-1},p_{2k})$ is a word in $\W$, we will denote by~$\dot{w}_\alpha$ the product 
\[
\dot{w}_\alpha\coloneqq  w_{p_1}^*w_{p_2}\ldots w_{p_{2k-1}}^*w_{p_{2k}}\in B.
\]
It is easy to see that  $\dot{w}_\alpha \dot{w}_\beta = \dot{w}_{\alpha\beta}$ and  $(\dot{w}_\alpha)^* = \dot{w}_{\tilde\alpha} $. Recall that when  $\alpha \in \W$ and $\dot{\alpha} =e$ we say $\alpha$ is a neutral word.
\end{note}
Next
we draw some consequences from  assuming that the alternating products factor through the constructible ideals. 
 
 \begin{prop} Let~$P$ be a submonoid of a group~$G$. Suppose
 $w \colon P\to B$ is a map from $P$ to a $\Cst$\nb-algebra $B$ such that  
 \label{pro:propertiesofsharpgenerators}
\begin{enumerate}\label{def:toeplitztype_Cstaralgebra}
\item[\textup{(T1)}] $w_e=1$;
\smallskip \item[\textup{(T2)}]  $\dot{w}_{\alpha} = 0$ \ if $K(\alpha)=\emptyset$ with $\dot\alpha =e$; and
\smallskip\item[\textup{(T3)}] 
$
\dot{w}_\alpha=\dot{w}_\beta
$
\  if  $\alpha$ and $\beta$ are neutral words in $\W$ such that $K(\alpha)=K(\beta)$.
\end{enumerate}
Then $w$ also has the following properties:
\begin{enumerate}\label{def:semigroup_Cstaralgebra}
\smallskip \item[\textup{(1)}] $w_p$ is an isometry for all~$p\in P$;

\smallskip \item[\textup{(2)}] $w_pw_q=w_{pq}$ for all $p,q\in P$;
 
\smallskip \item[\textup{(3)}] the set $\{\dot{w}_\alpha\mid \alpha\in\Hilm[W], \ \ \dot\alpha=e\}$ does not depend on the embedding $P\hookrightarrow G$ and is a commuting family of projections that is closed under multiplication and contains the identity; and

\smallskip \item[\textup{(4)}] if $F$ is a finite set of neutral words in~$\W$ with 
$\bigcup_{\beta\in F} K(\beta)= K(\alpha)$ 
for some neutral word $\alpha \in \W$ such that $K(\alpha) \in \{K(\beta)\mid \beta \in F\}$, then
\begin{equation}\label{eqn:property(4)}
\dot{w}_{\alpha} = \sum_{\emptyset \neq A\subset F} (-1)^{|A|+1} \prod_{\beta \in A} \dot{w}_{\beta}.
\end{equation}

\end{enumerate}

\begin{proof} We remark that the key ideas in the proof 
appear  in  \cite{Li:Semigroup_amenability}, especially equations (32) and (33).

Property (1) holds  because $K((p,p)) = P \cap p\inv pP = P = K((e,e))$, so $w_p^*w_p=w_e^* w_e =1$ by conditions (T3) and (T1). In order to prove property (2), notice that 
\begin{equation*}\begin{aligned} (w_pw_q-w_{pq})^*(w_pw_q-w_{pq})&=w_q^*w_p^*w_pw_q-w_q^*w_p^*w_{pq}-w_{pq}^*w_pw_q+w_{pq}^*w_{pq}\\&=2-w_q^*w_p^*w_{pq}-w_{pq}^*w_pw_q.
\end{aligned}
\end{equation*} 
The product $w_q^*w_p^*w_{pq}$ is associated to the word $\alpha = (q,e,p,pq)$, which satisfies $\dot \alpha = e$ and
 $\iqs(\alpha) = \{e,\, (pq)\inv p, \, (pq)\inv p e\inv q)\} = \{e, q\inv\}$, so that $K((q,e,p,pq)) = P\cap q\inv P = P = K(e,e)$. Hence $w_q^*w_p^*w_{pq} = w_e^*w_e = 1$
by conditions (T3) and (T1);
  similarly  $w_{pq}^*w_pw_q = 1$, using $\tilde\alpha$. It follows that $w_pw_q-w_{pq}=0$, proving property (2).

The  set $\{\dot{w}_\alpha\mid \alpha\in\Hilm[W], \ \ \dot\alpha=e\}$ does not depend on the embedding $P\hookrightarrow G$ because of \lemref{lem:indep-group}(1) and condition (T2). It is also obviously closed under multiplication because $\dot{w}_\alpha \dot{w}_\beta = \dot{w}_{\alpha\beta}$ and $\alpha\beta$ is neutral whenever $\alpha$ and $\beta$ are, and it certainly contains 
 $1 = \dot{w}_{(e,e)}$. In order to prove the remainder of property (3), let $\alpha$ and $\beta$ be neutral words in $\W$ and recall that
  $K(\alpha \beta) = K(\alpha) \cap K(\beta) = K(\beta\alpha)$ by \proref{pro:propertiesKalpha}(5). 
  Then condition (T3) applied to the neutral words $\alpha\beta $ and $\beta \alpha$ yields
  \[
  \dot{w}_\alpha \dot{w}_\beta = \dot{w}_{\alpha\beta} = \dot{w}_{\beta\alpha} =  \dot{w}_\beta \dot{w}_\alpha.
  \]
  To see that $\dot{w}_\alpha$ is a projection for each neutral word $\alpha$, recall that $\dot{w}_\alpha^*=\dot{w}_{\tilde{\alpha}}$ and that $K(\tilde{\alpha}\alpha)=K(\alpha)$ by \proref{pro:propertiesKalpha}(3). Hence condition (T3) yields   \[
\dot{w}_\alpha^*  \dot{w}_\alpha =  \dot{w}_{\tilde\alpha} \dot{w}_\alpha= \dot{w}_{ \tilde\alpha \alpha} =\dot{w}_\alpha ,
 \]
 which implies that each $\dot{w}_\alpha$ is a projection, completing the proof of property (3). 
     
Assume now that $F$ is a finite set of neutral words in $\W$ with 
$\bigcup_{\beta\in F} K(\beta)= K(\alpha)$ 
for some neutral word $\alpha \in \W$ such that $K(\alpha) \in \{K(\beta)\mid \beta \in F\}$. Then  
$\dot{w}_\alpha \dot{w}_{\beta} = \dot{w}_{\alpha\beta}  = \dot{w}_{\beta}= \dot{w}_{\beta} \dot{w}_\alpha  $ for every $\beta \in F$ because $K(\alpha\beta) = K(\alpha) \cap K(\beta) = K(\beta)$. Since 
$\dot{w}_\alpha - \dot{w}_{\beta} =0$ for some $\beta \in F$ by condition (T3), an easy expansion  yields
\begin{equation}\label{eqn:INC/EXCidentity}
0 = \prod_{\beta \in F}(\dot{w}_\alpha - \dot{w}_{\beta}) = \dot{w}_\alpha + \sum_{\emptyset \neq A\subset F} (-1)^{|A|} \prod_{\beta\in A} \dot{w}_{\beta} 
\end{equation}
 where the products over all subsets $A$ of $F$ can be taken in any order because of property (3).  This proves equation \eqref{eqn:property(4)}.
\end{proof} 
\end{prop}
 \begin{rem}  \label{rem:alternative2} 
  \proref{pro:propertiesofsharpgenerators}  highlights the  significance of Li's constructible right ideals. 
If we view (T1) and (T2) as `calibrations',  we are essentially just asking that the map $\alpha \mapsto \dot w_\alpha$ factor through $\alpha \mapsto K(\alpha)$, and yet we obtain a very strong algebraic structure for the collection $\{w_p\mid p\in P\}$ as a consequence. 
 We point out  that the order in which the projections are multiplied or the words concatenated does not affect the formulas in \eqref{eqn:INC/EXCidentity} because of part (3) and condition (T3) of \proref{pro:propertiesofsharpgenerators}.
\end{rem}

\begin{defn}\label{def:toeplitz-semigroup} Let $P$ be a submonoid of a group $G$.  We define the {\em universal Toeplitz algebra} of~$P$, denoted by $\Toepu(P)$, to be the universal $\Cst$\nb-algebra with generators $\{t_p\colon p\in P\}$ subject to the relations
\begin{enumerate}\label{def:toeplitztype_Cstaralgebra}
\item[\textup{(T1)}] $t_e=1$;

\smallskip \item[\textup{(T2)}]  $\dot{t}_{\alpha} = 0$ \  if $K(\alpha)=\emptyset$ with $\dot\alpha=e$;

\smallskip\item[\textup{(T3)}]
$
\dot{t}_\alpha=\dot{t}_\beta
$
\  if  $\alpha$ and $ \beta$ are neutral words  such that $K(\alpha)=K(\beta)$;

\smallskip\item[\textup{(T4)}] $ 
 \prod_{\beta \in F} (\dot{t}_\alpha - \dot{t}_{\beta})  =  0$ \  if $F$ is a finite  set of neutral words such that $K(\alpha)  =  \bigcup_{\beta \in F} K(\beta)$ for some neutral word $\alpha$.
\end{enumerate}
\end{defn}

\begin{rem}\label{rem:alternative1} 
We would like to make a few comments concerning \defref{def:toeplitz-semigroup}.
\begin{enumerate}
\item Relation (T3) is simply the special case of (T4) for $|F|=1$. Further, with the convention that the union over the empty set is empty, if we interpret an empty product in \textup{(T4)} as being equal to $\dot t_\alpha$ 
then we can also  derive (T2) from (T4).
We have chosen to include (T2) and (T3)  explicitly here mainly for clarity and to facilitate the comparison with further constructions. It seems also plausible that the verification of specific cases would probably have to go through (T2) and (T3) anyway. 

\item Notice that if $P$ is left reversible then  $K(\alpha)$ is never empty, making (T2) vacuous. So (T2) only applies when $P$ is not left reversible. 

\item Further insight into relation (T4) is gained from noticing its relation to independence.
  Recall  from \cite{Li:Semigroup_amenability} that the semigroup $P$ is said to satisfy the {\em independence condition}
if the union $\bigcup_{\beta \in F} K(\beta)$ of constructible ideals in $P$ is a constructible ideal itself only when $ \bigcup_{\beta \in F} K(\beta) = K(\alpha)$ for some neutral word $\alpha \in \W$ such that $K(\alpha) \in \{K(\beta)\mid \beta \in F\}$, equivalently, when the characteristic functions of constructible right ideals are linearly independent.
This, in turn, implies that the product in (T4) vanishes for every $w$ satisfying (T3), because one of the factors is zero.
 Thus, in particular, semigroups that satisfy independence and (T3) also satisfy (T4) automatically. The full implication of this observation is spelled out  in \corref{cor:canonicalSharptoLimap} below. 

\item It is also helpful to notice that in the particular case when the ideals $K(\beta)$ happen to be mutually disjoint,  the only nonzero terms in equation \eqref{eqn:INC/EXCidentity} are those for which the subset $A$ of $F$ is a singleton. In this case
(T4) simply reduces to  the familiar relation $\dot{t}_\alpha = \bigoplus_{\beta\in F}  \dot{t}_{\beta}$ involving a sum of mutually orthogonal projections.

\item Equations like those appearing in (T4) and  \eqref{eqn:property(4)}  go back  at least as far as the characterization of the diagonal in Cuntz--Krieger algebras of infinite matrices, and  can be imposed in various combinations to obtain different quotients. For instance,  relation (T4) applies only to ideals that fail to satisfy independence. Eventually, when we characterize the universal boundary quotient in \secref{sec:full-boundary} we will see how to enlarge the set of relations in a maximal way, \proref{prop:non-foundation-sets}. This will correspond to Exel's tightness condition \cite[Definition~2.6]{MR2534230} and to Donsig and Milan's cover-to-join relations \cite[Section~2]{DM14}, see also the discussion in \cite{Exel21}.
\end{enumerate}
  \end{rem}
We are indebted to Sergey Neshveyev for bringing  to our attention, after an earlier version of this paper was circulated, that the C*-algebras of left cancellative small categories initially defined in terms of two groupoids by Spielberg in \cite{Spi20} are also characterized there in terms of generators and relations, \cite[Theorem~9.4]{Spi20}.
Specifically, for group embeddable monoids, the presentation of Spielberg's $C^*(G_2)$  amounts to the presentation given in \defref{def:toeplitz-semigroup}, see \cite[Proposition~2.6]{NS22}.

We list next several equivalent formulations of relation (T4) that are helpful to understand better its meaning and motivation. In particular,  relation (T4) is  a stronger version of the property stated in
 \proref{pro:propertiesofsharpgenerators}(4) because it applies whenever the union of a finite collection  of constructible ideals is a constructible ideal, regardless of whether this ideal is a member of the collection or not.

\begin{lem} \label{lem:equivalentto(T4)}
Suppose  $ B$ is a unital  $\Cst$\nb-algebra and $t \colon P \to  B$ is a map that satisfies relations \textup{(T1)--(T3)}.
The following are equivalent:
\begin{enumerate}
\item relation \textup{(T4)};

\item  relation \textup{(T4)} restricted to the special cases where independence fails, that is,  for neutral words $\alpha\in\W$ and finite sets $F\subset\W$ such that $K(\alpha)= \bigcup_{\beta\in F} K(\beta)$ and $K(\alpha) \neq K(\beta)$ for all $\beta \in F$;

\item the expanded version of relation \textup{(T4)},
\begin{equation} \label{eqn:condition(T4)}
\dot{t}_\alpha = \sum_{\emptyset \neq A \subset F}(-1)^{|A|+1}  \prod_{\beta\in A} 
\dot{t}_{\beta}
\end{equation}
 if $F\subset\W $ is a finite set of neutral words and $\alpha \in \W$ is a neutral word that satisfies $K(\alpha) = \bigcup_{\beta\in F} K(\beta)$,
 equivalently, restricted to special cases where independence fails.
 
\item  the expanded version of relation \textup{(T4)} written in terms of concatenation of words: 
\begin{equation} \label{eqn:condition(T4)words}
\dot{t}_\alpha = \displaystyle\sum_{\emptyset \neq A \subset F}(-1)^{|A|+1}  
\,\ \dot{t}_{\prod_{\beta\in A}\beta}
\end{equation} f
or $F$ and $\alpha$ as in \textup{(T4)}, equivalently, restricted to the special cases where independence fails.
\end{enumerate}

\begin{proof} 
That (1) $\implies$ (2) is immediate. To see that  (2) $\implies$ (1), we only need to notice that in the cases that are left out of (2), that is to say, for each neutral word $\alpha$  and finite set $F$ of neutral words such that $K(\alpha) = \bigcup_{\beta\in F} K(\beta) = K(\beta_0)$ for some $\beta_0\in F$, the product in (T4) clearly vanishes because it includes the factor $(\dot t_\alpha - \dot t_{\beta_0})$, which is trivial by (T3). The remaining equivalences are easy to see from the expansion in equation \eqref{eqn:INC/EXCidentity}.
\end{proof}
\end{lem}

\subsection{The left regular representation} The (reduced) Toeplitz $\Cst$\nb-algebra $\Toepr(P)$
is, by definition,  the $\Cst$\nb-algebra generated by the image of the left regular representation $L$ of $P$  on $\ell^2(P)$, which is given by 
$L_p\delta_q = \delta_{pq}$ on the standard orthonormal basis elements.
The range projection $L_pL_p^*$ of the generating isometry $L_p$ is the
multiplication operator by the characteristic function $\1_{pP}$ of the set $pP$. Since the representation of  $\ell^\infty(P)$ as multiplication operators on $\ell^2(P)$ is isometric, there is no harm in confusing these multiplication operators with the functions they represent. We will use the following  basic observation from \cite{Li:Semigroup_amenability}.

\begin{lem}\cite[Definition 2.12 and Lemma 3.1]{Li:Semigroup_amenability}
\label{lem:1_KalphaspansD_r}
 Let $P $ be a submonoid of a group $G$, and let $\1_{K(\alpha)}\in\ell^\infty(P)$ denote the characteristic function of~$K(\alpha)$. Then
$\dot{L}_\alpha = \1_{K(\alpha)}$ for every neutral word $\alpha$ and 
$D_r  \coloneqq \clsp\{\1_{K(\alpha)} \in \Toepr(P)  \mid \dot\alpha =e\} $ is a $\Cst$\nb-subalgebra of $\ell^\infty(P)$, which we call the {\em reduced diagonal}.
\end{lem}

 There is also a  {\em full diagonal} subalgebra $D_u\coloneqq \clsp\{\dot{t}_\alpha \in \Toepu(P)\mid \dot\alpha =e\} $ 
 at the level of the universal Toeplitz algebra. This does not depend on the embedding $P\hookrightarrow G$ and is a commutative unital $\Cst$\nb-subalgebra of $\Toepu(P)$ by \proref{pro:propertiesofsharpgenerators}(3).
 
 \begin{prop}\label{prop:diagonal-sub} Let $P$ be a submonoid of a group $G$, and denote by $\{t_p\mid p\in P\}$ the set of canonical  generators for $\Toepu(P)$. Then
 the map $t_p \mapsto L_p$ extends to a surjective \Star homomorphism $\lambda^+\colon \Toepu(P)\to \Toepr(P)$.
  
\begin{proof} Since $\dot{L}_\alpha = \1_{K(\alpha)}$ it is easy to see that relations (T2) and (T3) are satisfied by $L$, and relation (T1) is obvious for $L$.
In order to prove that $L$ also satisfies (T4), let $F$ be a finite set of neutral words such that the union of their corresponding right ideals is a constructible right ideal, that is, such that $\bigcup_{\beta\in F} K(\beta) = K(\alpha) $ for some 
neutral word $\alpha$. Then $\emptyset = \bigcap_{\beta\in F} (K(\alpha) \setminus K(\beta) )$, and hence
\[
0 =  \prod_{\beta\in F} (\1_{K(\alpha)} - \1_{K(\beta)}), 
\]
which shows that $L$ satisfies relation  (T4). The resulting canonical homomorphism
$\lambda^+:\Toepu(P) \mapsto \Toepr(P)$ is surjective  because the image contains the generating isometries $L_p$ for $p\in P$.
\end{proof}
\end{prop}

It is clear from \proref{prop:diagonal-sub} that the restriction $\lambda^+\!\restriction_{D_u}$ maps  $D_u$  onto  $D_r$ and we would like to show next that 
this restriction is in fact an isomorphism. We will obtain this as a corollary of our characterization of the representations of $\Toepu(P)$ that are faithful on $D_u$. The following fact is probably well known, but we have not been able to find a general reference. We include the concrete statement here for completeness and to set the notation.

\begin{lem}[cf. Lemma~1.4 of \cite{LACA1996415}]\label{lem:orthogonaldecomp}
Suppose $F$ is a  finite set of mutually commuting projections in a unital $\Cst$\nb-algebra and let $\lambda_X \in \CC$ for each $X \in F$.
For every subset $A$ of $F$ define
\[
Q_A\coloneqq \prod_{X\in A} X \prod_{X\in F\setminus A} (1-X), 
\]
which includes the extreme cases $Q_\emptyset = \prod_{F} (1-X)$ and $Q_F= \prod_F X$.

Then  $1 = \sum_{A\subset F} Q_A $ is a decomposition of the identity into mutually orthogonal projections, 
\[
\sum_{X\in F} \lambda_X X = \sum_{\emptyset \neq A \subset F} ({\textstyle\sum_{X\in A} \lambda_X}) Q_A
\]
and
\[
\Big\|\sum_{X\in F} \lambda_X X\Big\| = \max\left\{ \big|{\textstyle\sum_{X\in A} \lambda_X}\big| \mid \emptyset \neq  A\subset F, \ \ Q_A \neq 0 \right\}
\]
\begin{proof} The proof follows, mutatis mutandis, from the proof of \cite[Lemma~1.4]{LACA1996415}, keeping the product  $\prod_{X\in A} X $ instead of substituting it by $X_{\sigma A}$. 
\end{proof}
\end{lem}
\begin{defn}\label{def:jointlyproper}
A map $w\colon P \to B$  into a unital $\Cst$\nb-algebra $B$ satisfying  (T1)--(T4) of \defref{def:toeplitz-semigroup}, is said to be {\em jointly proper} if 
\begin{equation}\label{eqn:jproper}
\prod_{\alpha\in F}(1 - \dot{w}_\alpha ) \neq 0
\end{equation}
 for every finite collection $F$ of neutral words such that $K(\alpha)$ is a proper constructible right ideal for each $\alpha \in F$.
 By extension, we will also say that a unital \Star homomorphism $\rho$ of  $\Toepu(P)$ into a $\Cst$\nb-algebra $B$ is {\em jointly proper} if
\[
\prod_{\alpha\in F}(1 - \rho(\dot{t}_\alpha) ) \neq 0
\]
 for every finite collection $F$ as above.
 Because of (T3), it suffices to verify that \eqref{eqn:jproper} holds for  words $\alpha$ in a collection large enough to generate all constructible ideals.

\end{defn}
\begin{example}Recall that an isometry $V$  is called {\em proper} if $V V^* \neq 1$. So Coburn's theorem \cite{Cob} can be rephrased by saying that the $\Cst$\nb-algebra generated by a proper isometry is canonically unique.  In order to see why the stronger condition of joint properness might be necessary for a uniqueness result, let $S$ be the  unilateral shift on $\ell^2(\NN)$ and consider the two  isometries $V\coloneqq S\oplus 1$  and $W \coloneqq 1 \oplus S$ defined on $\ell^2(\NN) \oplus \ell^2(\NN)$.  Then $V$ and $W$ are proper isometries that $*$-commute. Thus the isometric representation $T:\NN^2 \to \mathbb B( \ell^2(\NN)\oplus \ell^2(\NN))$ defined by $T_{(1,0)}\coloneqq V$ and $T_{(0,1)}\coloneqq W$ is covariant in the sense of Nica, equivalently, satisfies (T3), but is not jointly proper.  Denoting as usual by $c$ the $\Cst$\nb-algebra of convergent sequences, we see that  the diagonal in $\Cst(T_{\NN^2})$ is isomorphic to $c\oplus c$ while the diagonal in $\Toepu (\NN^2)$ is $c \otimes c$. 

Other familiar examples of proper isometries that are not jointly proper arise from isometric representations  $V: \FF_n^+ \to \mh$ with $n<\infty$. In this case  Nica-covariance  means that the generating isometries $V_j$ have mutually orthogonal ranges. If $\sum_j V_j V_j^* =1$, or, equivalently, if  $\prod_j(1-V_jV_j^*) =0$, then the resulting representation of $\Toep\!\mathcal O_n$ is not jointly proper. In fact, such a representation  factors through $\mathcal O_n$, whose diagonal is isomorphic to the continuous functions on infinite path space, while the diagonal in $\Toep\!\mathcal O_n$ itself is isomorphic to the continuous functions on the space of finite and infinite paths.
\end{example}
\begin{rem}
If $R$ and $S$ are constructible ideals with $R\subset S$, then $\1_R \leq \1_S$ and hence $\1 -\1_S \leq \1-\1_R$.
Hence the verification of whether a representation is jointly proper can sometimes be reduced to finite subsets of larger ideals, which are generally associated to shorter words. 
\end{rem}

\defref{def:jointlyproper} generalizes to constructible right ideals the condition used in \cite[Proposition~2.3]{LACA1996415} to characterize  representations that are faithful on the diagonal algebra  
when $P$ is the positive cone in a quasi-lattice ordered group. So it should not be too much of a surprise that 
such a characterization is possible for general submonoids of groups too.

\begin{prop}\label{pro:faithfulonD_u}
Suppose $P$ is a submonoid of the group $G$. 
A unital \Star homomorphism $\rho\colon\Toepu(P) \to B$ into a  $\Cst$\nb-algebra $B$ is faithful on the full diagonal subalgebra $D_u$
if and only if  it is jointly proper.
\begin{proof}
Evaluation at $e\in P$ shows that  $\prod_{\alpha\in F}(\1 - \1_{K(\alpha)})\neq 0$ in $D_r$; 
so the representation $\lambda^+$ arising from the map $L\colon P \to \Toepr(P)$ 
is jointly proper. Hence the same is true for the identity representation of $\Toepu(P)$ arising from the universal map 
$t: P \to \Toepu(P)$, 
so the condition is necessary.

For sufficiency, suppose $\rho$ is jointly proper and let $a\neq 0$ be an element  of the form 
\[
a\coloneqq \sum_{\alpha\in F} \lambda_\alpha \dot{t}_\alpha \in D_u,
\]  where $F\subset\Hilm[W]$ is a finite collection of  
neutral words.
 We will show that $\rho(a) \neq 0$ in $B$. 
 Since $a$ is a linear combination of commuting projections, 
 \lemref{lem:orthogonaldecomp} gives a nonempty subset $A$ of $F$ such that the projection
\[
Q_A\coloneqq\prod_{\alpha \in A} \dot{t}_\alpha \prod_{\beta \in F\setminus A} (\1 -\dot{t}_\beta)
\]
is nonzero and satisfies $\| Q_A a\| = \|a\| = | \sum_{\alpha \in A} \lambda_\alpha|$. 

If the set $\bigcap_{\alpha\in A} K(\alpha) \setminus \bigcup_{\beta  \in F\setminus  A} K(\beta)$ 
were empty, then we would have
\[
\bigcap_{\alpha\in A} K(\alpha) = \bigcup_{\beta  \in F\setminus  A} \Big(K(\beta) \cap \bigcap_{\alpha\in A} K(\alpha)\Big),
\]
 and equation \eqref{eqn:condition(T4)words} would realize 
$\prod_A \dot{t}_\alpha = \dot{t}_{\prod_{\alpha\in A} \alpha} $ as a linear combination of subprojections of the projections $\dot{t}_\beta$ for $\beta \in F\setminus A$. This would force $Q_A =0 $, contradicting the choice of $Q_A$. So there exists $p \in \bigcap_{\alpha\in A} K(\alpha) \setminus \bigcup_{\beta  \in F\setminus  A} K(\beta)$, and we may define a projection by
 \[
 Q \coloneqq t_p t_p^* Q_A t_p t_p^* = t_p\prod_{\beta\in F\setminus A} (\1 -  (t_p^*\dot{t}_{\beta}t_p))t_p^*=
 t_p \,  \Big(\prod_{\beta \in F\setminus A} (\1 - \dot{t}_{(p,e)\beta(e,p)}) \Big) t_p^*.
  \]
Clearly $Q$ is a   subprojection of $Q_A$ 
and $Q\neq 0$ because through the left regular representation
we have $\lambda^+(Q) (p) = \prod_{\beta\in F\setminus A} \big(\1 - \1_{K(\beta)\cap pP} \big)(p) = 1$.
Hence $aQ =  \big( \sum_{\alpha \in A} \lambda_\alpha \big)  Q$, and thus $\|aQ\| = \|a\| = |\sum_{\alpha \in A} \lambda_\alpha|$.
 
Passing to the representation $\rho$,  we get
\[
\rho(Q) = \rho(t_p)\,  \Big(\prod_{\beta\in F\setminus A} (\1 - \rho(\dot{t}_{(p,e)\beta(e,p)}) \Big) \rho(t_p)^*,
\]
where the middle factor is of the type that appears in the joint properness condition \eqref{eqn:jproper}.
From \proref{pro:propertiesKalpha}(6) we know that
 \[
 K((p,e)\beta(e,p)) =  P \cap p\inv  K(\beta) 
 \]
and since $p\notin K(\beta)$ by construction, we conclude that the ideal $K((p,e)\beta(e,p))$ is proper for each $\beta\in F \setminus A$. Thus,
our assumption \eqref{eqn:jproper}, together with the fact that $\rho(t_p)$ is an isometry, imply that $\rho(Q) \neq 0$. 
  But since $\rho(a) \rho(Q) =\rho(a Q) = \big( \sum_{\alpha \in A} \lambda_\alpha \big)  \rho(Q)$ and $ \big|\big( \sum_{\alpha \in A} \lambda_\alpha \big)\big| = \|a\| \neq 0$, we must have $\rho(a) \neq 0$ as wanted.
 
 To establish the proposition, observe that for each $\cap$-closed finite subcollection $\mc$ of constructible ideals,
 \[
 \A(\mc) \coloneqq \clsp\{\dot{t}_\alpha\mid \alpha\in\Hilm[W],\,K(\alpha)\in \mc, \dot{\alpha}=e\}
 \]
 is a finite dimensional $\Cst$\nb-subalgebra of $D_u$ spanned by a finite set of projections arising from constructible right ideals. To see this, for each $S\in \mc$, choose a neutral word $\alpha_S$ with $K(\alpha_S)=S$. Then it follows from relation (T3) of \defref{def:toeplitz-semigroup} that $ \A(\mc)=\lsp\{\dot{t}_{\alpha_S}\mid S\in\mc\}$ as claimed. Now notice that
 $D_u = \lim_{\mc} \A(\mc)$, with the limit taken over the $\cap$-closed finite subcollections of constructible ideals 
 directed by inclusion. From the above we deduce that the representation $\rho$
 is faithful on each $\A(\mc)$, and so it is also faithful on $D_u$. 
   \end{proof}
\end{prop}
\begin{cor}\label{cor:faithfulondiagonal}
The restriction of $\lambda^+$ to the full diagonal $D_u$ gives a canonical isomorphism $D_u \cong D_r  $.
A representation of $\Toepr(P)$ is faithful  on $D_r$ if and only if it is jointly proper.
\begin{proof}
The first line of the proof of \proref{pro:faithfulonD_u} 
verifies that $\lambda^+$ is jointly proper, so the first assertion follows by \proref{pro:faithfulonD_u}. 
Once we know that~$D_u$ is canonically isomorphic to $D_r$, the second assertion follows directly also from \proref{pro:faithfulonD_u}.
\end{proof}
\end{cor}

Let $E_r\colon\Toepr(P)\to D_r$  be the restriction of the canonical diagonal conditional expectation from $\Bound(\ell^2(P))$ onto~$\ell^\infty(P)$. It is determined   
by $$\braket{E_r(b)\delta_p}{\delta_q}=\begin{cases} \braket{b\delta_p}{\delta_p}&\text{ if }p=q,\\
0 & \text{otherwise}
\end{cases}$$ 
for $b\in \Toepr(P)$ and is a faithful conditional expectation.

\begin{cor}\label{cor:reg-kernel} Let $P$ be a submonoid of a group. Then $E_u\coloneqq (\lambda^+\restriction_{D_u})^{-1}\circ E_r\circ\lambda^+$ is a conditional expectation from $\Toep_u(P)$ onto $D_u$ that vanishes on the subspace $$B_g\coloneqq\clsp\{\dot{t}_\alpha\mid \alpha\in\Hilm[W],\, \dot{\alpha}=g\}\subset\Toepu(P)$$ whenever $g\neq e$. Moreover, $$\ker\lambda^+=\{b\in \Toep_u(P)\mid E_u(b^*b)=0\}.$$
\begin{proof} The first assertion is clearly true. For the last one, notice that $E_u(b^*b)=0$ if and only if $\lambda^+(b^*b)=0$ because $E_r$ is faithful. Now the result follows because $\lambda^+(b^*b)=0$ if and only if $\lambda^+(b)=0$.
\end{proof}
\end{cor}

\subsection{Li's semigroup \texorpdfstring{$\mathrm{C}^*$}{C*}-algebra} Next we wish to compare $\Toepu(P)$ with the full semigroup $\Cst$\nb-algebra of a submonoid of a group, denoted by $\Cst_s(P)$ in \cite{Li:Semigroup_amenability}. 
We begin by recalling that definition.
\begin{defn} (cf.\cite[Definition 3.2]{Li:Semigroup_amenability})
\label{def:cstar_s(P)} Let $P$ be a submonoid of a group $G$. The \emph{semigroup $\Cst$\nb-algebra} of~$P$, denoted by~$\Cst_s(P)$, is the universal $\Cst$\nb-algebra generated by a family of isometries~$\{v_p\mid p\in P\}$ and projections~$\{e_S\mid S\in\J \cup\{\emptyset\}\}$ such that
\begin{enumerate}\label{def:semigroup_Cstaralgebra}
\item[\textup{(i)}] $v_pv_q=v_{pq}$ \  whenever $p,q\in P$;
\smallskip\item[\textup{(ii)}] $e_{\emptyset}=0$;
\smallskip\item[\textup{(iii)}] $\dot{v}_\alpha=e_{S}$ \  whenever $S\in \J$ and $\alpha\in \W$ satisfy $\dot\alpha=e$ and $ K(\alpha) = S$.
\end{enumerate}
\end{defn}
The family $\{e_S\mid S\in\J \cup\{\emptyset\}\}$ replicates the semilattice structure of the
corresponding family of subsets of $P$.
This property is included as part of \cite[Definition~2.2]{Li:Semigroup_amenability} for general semigroups, but for submonoids of groups
it is a consequence of the relations listed above. 
Perhaps surprisingly, one can achieve the same effect by requiring a lot less, as shown in the next proposition.

\begin{prop} \label{pro:equivalentcovs}
Let $P$ be a submonoid of a group and let  
$\Cst_{\#}(P)$ be the universal $\Cst$\nb-algebra generated by a family of elements~$\{w_p\mid p\in P\}$ subject to the relations \textup{(T1)--(T3)} of \proref{pro:propertiesofsharpgenerators}.
Then $\Cst_{\#}(P)$ is canonically isomorphic to $\Cst_s(P)$, i.e. there is an isomorphism that maps $w_p$ to $v_p$.
\begin{proof}  
We know from \proref{pro:propertiesofsharpgenerators} that any family $\{w_p\mid p\in P\}$ satisfying  \textup{(T1)} and \textup{(T3)}
consists of a semigroup of isometries, so there is no concern about existence of the universal object $\Cst_{\#}(P)$.
Denote by ${v}_p$ and ${e}_S$ the generating isometries and projections of $\Cst_s(P)$. 
Since $1 = {v}_e^* {v}_e =  {v}_e^* {v}_e  {v}_e =  {v}_e$ by the first relation in \defref{def:cstar_s(P)} we see that  
\textup{(T1)}  holds in $\Cst_s(P)$. 
If $\alpha\in\Hilm[W]$ is a neutral word with $K(\alpha) =\emptyset$, then $\dot{v}_\alpha=  {e}_{K(\alpha)} =0 $, from which it follows that relation \textup{(T2)} in the presentation of $\Cst_{\#}(P)$ 
holds in $\Cst_s(P)$.
A similar argument shows that  \textup{(T3)} also holds.  Hence there is a \Star homomorphism $\Cst_{\#}(P)\to\Cst_s(P)$ that maps the 
element~$w_p\in \Cst_{\#}(P)$ to the corresponding isometry $ {v}_p\in\Cst_s(P)$. This 
homomorphism is surjective by Corollary 2.10 and Lemma~3.3 of \cite{Li:Semigroup_amenability}. 

To obtain the inverse map, recall that the family of constructible right ideals of~$P$ is given by 
\[
 \big\{ K(\alpha) \mid \alpha \in \W, \ \dot\alpha = e\}
\] 
by \proref{pro:characterizationconstructibleideals}. So for each $S\in \J$ we may choose $\alpha_S=(p_1,p_2,\ldots,p_{2k})$ with $\dot{\alpha_S}=e$
such that $S = K(\alpha_S)$. Then
 \[
\dot{w}_{\alpha_S}=w_{p_1}^*w_{p_2}\cdots w_{p_{2k-2}} w_{p_{2k-1}}^*w_{p_{2k}}
\] 
is a projection by \proref{pro:propertiesofsharpgenerators}(3). 
Because  of relation \textup{(T3)}, { the projection $\dot{w}_{\alpha_S}$ in~$\Cst_{\#}(P)$ does not depend on the choice of the neutral word~$\alpha_S$ representing the right ideal~$S$.
It follows from the definition of the~$\dot{w}_{\alpha_S}$'s and \proref{pro:propertiesofsharpgenerators} that the family of projections $\{\dot{w}_{\alpha_S}\mid S\in \J\}$ in $\Cst_{\#}(P)$, together with the family of isometries 
$\{w_p\mid p\in P\}$, satisfy the conditions of \defref{def:cstar_s(P)}. 
Hence the maps $ {v}_p \mapsto w_p$, $e_S\mapsto \dot{w}_{\alpha_S}$ and $e_\emptyset\mapsto 0$ extend to a homomorphism $\Cst_s(P)\to\Cst_{\#}(P)$, which is obviously the inverse of the
 one determined above by $ w_p \mapsto  {v}_p$.}
\end{proof}
\end{prop}

\begin{cor} \label{cor:canonicalSharptoLimap} Suppose $P$ is a submonoid of a group. Then the map $v_p \mapsto t_p$  extends to a canonical surjective \Star homomorphism $\lambda^{\#}\colon\Cst_s(P)\to\Toepu(P)$, which is an isomorphism if and only if~$P$ satisfies independence.
\begin{proof} That $v_p\mapsto t_p$ extends to a surjective \Star homomorphism $\lambda^{\#}\colon\Cst_s(P)\to\Toepu(P)$ follows from Proposition~\ref{pro:equivalentcovs} since the relations defining $\Toepu(P)$ include those defining $\Cst_{\#}(P)$. Suppose that $\lambda^{\#}$ is an isomorphism. Then the restriction of~$\lambda^{\#}$ to the diagonal subalgebra $$D_s=\overline{\mathrm{span}}\{e_S\mid S\in \J\}$$ is faithful. Hence $(\lambda^+\circ\lambda^{\#})\restriction_{D_s}\colon D_s\to D_r$ is an isomorphism by Proposition~\ref{prop:diagonal-sub}. But $\lambda^+\circ\lambda^{\#}$ is precisely the left regular representation $\lambda\colon \Cst_s(P)\to \Toepr(P)$. So $P$ satisfies independence by \cite[Corollary~2.27]{Li:Semigroup_amenability}.

From \proref{pro:equivalentcovs} we know that (T1)--(T3) hold in $\Cst_s(P)$ for general $P$. When $P$ satisfies independence,  condition (2) of \lemref{lem:equivalentto(T4)} is void, hence automatically satisfied. Thus \lemref{lem:equivalentto(T4)} implies that (T4) holds in $\Cst_s(P)$. So by the universal property of $\Toepu(P)$ the map that sends $t_p$ to~$v_p$ for each~$p\in P$ extends to a homomorphism
 $\Toepu(P)\to\Cst_s(P)$, which is the inverse of~$\lambda^\#$.
\end{proof}
\end{cor}

Another semigroup $\Cst$\nb-algebra, denoted $\Cst_s{}^{(\cup)}(P)$, is mentioned in passing, right before Subsection 3.1 of \cite{Li:Semigroup_amenability}. Its presentation is not given explicitly, but 
 the notation makes it clear that $\Cst_s{}^{(\cup)}(P)$ is meant to be the quotient of the $\Cst$\nb-algebra $\Cst{}^{(\cup)}(P)$ from \cite[Definition 2.4]{Li:Semigroup_amenability} by the ideal generated by the relation \textup{III$_G$} in \cite[Definition 3.2]{Li:Semigroup_amenability}. We can deduce from \cite[Lemma~3.3]{Li:Semigroup_amenability} that $\Cst_s{}^{(\cup)}(P)$ coincides with the $\Cst$\nb-algebra with the same defining relations as $\Cst_s(P)$ but with extra generators in the presentation. These come from the projections $\{e_X\mid X\in\J^{(\cup)}\},$ where $$\J^{(\cup)}=\big\{\bigcup_{R\in\mc}R\mid \emptyset\neq\mc\subset  \J,\, |\mc|<\infty\big\}.$$  Our  (T4) implies relation II$^{(\cup)}$(iv) in \cite[Definition 2.4]{Li:Semigroup_amenability} in the special case when $X$ and $Y$ are constructible ideals whose  union is also a constructible ideal. By \cite[Proposition~2.24]{Li:Semigroup_amenability}, it is also reasonable to expect that $\Cst_s{}^{(\cup)}(P)$ should  have the same property as our 
$\Toepu(P)$, of being a quotient of $\Cst_s(P)$
that is isomorphic to it when independence holds.  Thus, it is natural to wonder whether  $\Cst_s{}^{(\cup)}(P)$ and $\Toepu(P)$ are one and the same.
On the path to decide this question we establish next some equivalent forms of relation (T4).

\begin{lem} \label{lem:equivalentto(iv)}
Suppose that  $w:P \to B$ is a map of $P$ into a $\Cst$\nb-algebra $B$ that satisfies relations \textup{(T1)} and \textup{(T2)}.
The following are equivalent:
\begin{enumerate}
\medskip\item $w:P \to B$ satisfies relation \textup{(T4)};

\medskip\item $\displaystyle \sum_{\beta\in F} \lambda_\beta \dot{w}_{\beta} = 0$ \  whenever $F$ is a finite set of neutral words in~$\W$ and the linear combination ${\displaystyle\sum_{\beta\in F} \lambda_\beta \1_{K(\beta)}}$ vanishes in {$D_r$};

\medskip\item  $ \displaystyle\sum_{\emptyset \neq B\subset F} (-1)^{|B|} \prod_{\beta\in B} \dot{w}_{\beta} = \sum_{\emptyset \neq C\subset H} (-1)^{|C|} \prod_{\gamma\in C} \dot{w}_{\gamma}$  
\  whenever  $F$ and $H$ are finite sets of neutral words in $\W$ such that 
${ \bigcup_{\beta\in F}  K(\beta) = \bigcup_{\gamma \in H}K(\gamma) }$; and 

\medskip\item  $ \displaystyle\prod_{\beta\in F} (1 -\dot{w}_{\beta}) =  \prod_{\gamma \in H} (1 -\dot{w}_{\gamma})$  
\  \  \  whenever  $F$ and $H$ are finite sets of neutral words in $\W$ such that 
$ \bigcup_{\beta\in F}  K(\beta) = \bigcup_{\gamma \in H} K(\gamma) $.
\end{enumerate}
\begin{proof} Notice first that each one of the conditions (1)--(4) imply that (T3) holds.
We will prove (3)$\implies$(1)$\implies$(2)$\implies$(3)$\iff$(4). 
Setting $F=\{\alpha\}$ in  (3) gives  (1), and 
 a standard application of the expansion in \eqref{eqn:INC/EXCidentity} shows that  (3) and (4) are equivalent. This takes care of the first and the last (double) implications.

Assume now (2) holds and  ${ \bigcup_{\beta\in F} K(\beta) = \bigcup_{\gamma \in H} K(\gamma) }$. Then 
\[
\sum_{\emptyset \neq B\subset F} (-1)^{|B|+1} \1_{K(\prod_{i\in B}\beta)}  = \sum_{\emptyset \neq C\subset H} (-1)^{|C|+1} \1_{K(\prod_{k\in C}\gamma)}
\]
in $D_r$. So by condition (2) the corresponding equality will hold in~$B$ with $\dot{w}_{K(\prod_B\beta)}$ in place of $\1_{K(\prod_B\beta)}$ and $\dot{w}_{K(\prod_C\gamma)}$ in place of $\1_{K(\prod_C \gamma)}$.
This gives (3) and establishes (2)$\implies$(3).

In order to see that (1)$\implies$(2), assume that $w$ satisfies (T4), and let $\rho_w$ be the resulting representation of $\Toepu(P)$.
If $\sum_{\beta\in F} \lambda_\beta \1_{K(\beta)} = 0$ in $D_r$, then $\sum_{\beta\in F} \lambda_\beta \dot{t}_{\beta} = 0$ in $D_u$ because of
\corref{cor:faithfulondiagonal}, so necessarily 
\[\sum_{\beta\in F} \lambda_\beta \dot{w}_{\beta}= \rho_w\big(\sum_{\beta\in F} \lambda_\beta \dot{t}_{\beta}\big) = 0.\qedhere
\]
\end{proof}
\end{lem}

\begin{prop} \label{pro:charactofToepu} The following $\Cst$\nb-algebras are canonically isomorphic:
\begin{enumerate}
\item the universal Toeplitz $\Cst$\nb-algebra $\Toepu(P)$;

\smallskip\smallskip\item the $\Cst$\nb-algebra with presentation  \textup{(T1)--(T3)}, and

\textup{(T4lc):\ } $\sum_{\alpha \in F}\lambda_\alpha \dot{t}_{\alpha}=0$ \ whenever  
$\sum_{\alpha \in F}\lambda_\alpha\1_{K(\alpha)}$ vanishes in $\ell^\infty(P)$;

\smallskip\smallskip\item the $\Cst$\nb-algebra $\Cst_s{}^{(\cup)}(P)$  
from \cite[Section 3]{Li:Semigroup_amenability}  defined as the quotient of $\Cst{}^{(\cup)}(P)$
by the ideal $\langle \dot{v}_\alpha - e_{K(\alpha)} \mid   \alpha \in \W, \dot \alpha =e\rangle$ from relation \textup{III$_G$} in \cite[Definition~3.2]{Li:Semigroup_amenability}.

\end{enumerate}
\begin{proof}
That $\Toepu(P)$ is canonically isomorphic to the $\Cst$\nb-algebra with the presentation \textup{(T1)--(T3)} and \textup{(T4lc)} given in item (2) follows from  the equivalence of conditions (1) and (2) in \lemref{lem:equivalentto(iv)}. From \proref{pro:equivalentcovs}  we know that (T1)--(T3)  hold for the isometries in $\Cst_s{}^{(\cup)}(P)$. By \cite[Corollary~2.22]{Li:Semigroup_amenability} 
the map $D^{(\cup)} \to D_r$ that sends a generating projection to the characteristic function of the corresponding right ideal is an isomorphism. So condition (2) from \lemref{lem:equivalentto(iv)} also
holds in $\Cst_s{}^{(\cup)}(P)$. Thus, the map ${t}_p \mapsto v_p$ extends to a canonical \Star homomorphism of $\Toepu(P)$ to $\Cst_s{}^{(\cup)}(P)$. 

For the inverse of the above \Star homomorphism, let $\{t_p\mid p\in P\}$ be the canonical generating elements of $\Toepu(P)$.
If $X = \bigcup_{\beta\in F} K(\beta)\in\J^{(\cup)}$,  we may define 
\[
\epsilon_X \coloneqq \sum_{\emptyset \neq B\subset F} (-1)^{|B|+1} \prod_{\beta\in B} \dot{t}_\beta =
1 -  \prod_{\beta\in F} (1 - \dot{t}_\beta) 
\]
because the right hand side above only depends on $X$ by parts (3) and (4) of \lemref{lem:equivalentto(iv)}. It is now easy to verify that the pair of maps $(t,\epsilon) $ satisfies the relations defining  $\Cst_s{}^{(\cup)}(P)$. 
Relation III$_G$ is simply the definition of $\epsilon_X$ with $X = K(\alpha)$. The relation $t_p\epsilon_Xt_p^*=\epsilon_{pX}$ can be derived from relation III$_G$ since it holds when $X$ is in~$\J$ (see \cite[Lemma~3.3]{Li:Semigroup_amenability}). The remaining relations to be verified only involve the $\epsilon_X$'s and come from the presentation of 
$\Cst{}^{(\cup)}(P)$. These relations are satisfied in $D_r$ as observed right after \cite[Definition~2.4]{Li:Semigroup_amenability}. But we know from \corref{cor:faithfulondiagonal} that  there is a canonical isomorphism $D_u \cong D_r$, so the relations are satisfied by the set of projections $\{\epsilon_X\mid X\in\J^{(\cup)}\}$ in $D_u$ as well. Hence  $\Toepu(P)$ and $\Cst_s{}^{(\cup)}(P)$ are canonically isomorphic. 
\end{proof}
\end{prop}

 \section{Semigroup \texorpdfstring{$\mathrm{C}^*$}{C*}-algebras as partial crossed products} \label{sec:partialactions}
A partial action of~$G$ on $D_r$ is constructed in \cite[Section~5.5.2]{CELY},  
 and the corresponding reduced partial crossed product is shown to be isomorphic to
 $\Toepr(P)$,  \cite[Theorem~5.6.41]{CELY}. Here we aim to
show that the  full partial crossed product of that action is always canonically isomorphic to our~$\Toepu(P)$.  
For ease of reference and to establish our notation we will describe the partial action of~$G$ on~$D_r$ explicitly in terms of the constructible right ideals of~$P$. This will make our study of faithful representations for~$\Toepr(P)$ and of simplicity of the boundary quotient more accessible. We point out, nevertheless, that many of the results of this section could also be extracted from \cite[Section~5]{CELY}.
 
\subsection{Partial action basics} We begin with some basic facts concerning partial actions and partial crossed products.
\begin{defn}[\cite{Exel:Partial_dynamical}*{Definition~11.4}]\label{defn: partial-action}
A \emph{partial action} of a discrete group~$G$ on a $\Cst$\nb-algebra~$A$ is a pair $\gamma=(\{A_g\}_{g\in G},\{\gamma_g\}_{g\in G})$, where
$\{A_g\}_{g\in G}$ is a collection of closed two-sided ideals of~$A$ and $\gamma_g\colon A_{g^{-1}}\to A_g$ is a \Star isomorphism for each~$g\in G$, such that for all~$g,h\in G$

\begin{enumerate}
\item[\rm{(1)}] $A_e=A$ and $\gamma_e$ is the identity on~$A$;

\item[\rm{(2)}] $\gamma_g(A_{g^{-1}}\cap A_h)\subseteq A_{gh}$;

\item[\rm{(3)}] $\gamma_g\circ \gamma_h=\gamma_{gh}$ on~$A_{h^{-1}}\cap A_{(gh)^{-1}}$.
\end{enumerate}
\end{defn}

We recall the construction of full and reduced partial crossed products based on full and reduced cross-sectional $\Cst$\nb-algebras of Fell bundles. Further details can be found in~\cite{Exel:Partial_dynamical}. Let $\gamma=(\{A_g\}_{g\in G},\{\gamma_g\}_{g\in G})$ be a partial action of~$G$ on~$A$. We build a Fell bundle $\Hilm[B]_\gamma=(B_{\gamma_g})_{g\in G}$ over~$G$ as follows. We set $B_{\gamma_g}\coloneqq A_g$ as a complex Banach space. For $a\in A_g$, we write $a\delta_g$ to identify the element in $B_{\gamma_g}$ corresponding to~$a$. The multiplication map $$B_{\gamma_g}\times B_{\gamma_h}\to B_{\gamma_{gh}}$$ is then given by 
\begin{equation}\label{eq:multiplicationmap}
(a\delta_g)\cdot (b\delta_h)\coloneqq \gamma_g(\gamma_{g^{-1}}(a)b)\delta_{gh}, \qquad a\in A_g, b\in A_h, g,h\in G.
\end{equation} This is well defined by condition (2) of Definition~\ref{defn: partial-action}. The resulting multiplication operation on~$\Hilm[B]_{\gamma}$ is associative. For each $g\in G$, we define an involution $^*\colon B_{\gamma_g}\to B_{\alpha_{g^{-1}}}$ by $$(a\delta_g)^*\coloneqq \gamma_{g^{-1}}(a^*)\delta_{g^{-1}}, \qquad a\in A_g.$$ Then
$\Hilm[B]_\gamma=(B_{\gamma_g})_{g\in G}$ is a Fell bundle whose unit fiber algebra is~$A$ (see \cite[Proposition~16.6]{Exel:Partial_dynamical}).

\begin{defn} The Fell bundle $\Hilm[B]_{\gamma}=(B_{\gamma_g})_{g\in G}$ constructed above is called the \emph{semidirect product bundle} relative to~$\gamma=(\{A_g\}_{g\in G},\{\gamma_g\}_{g\in G})$. The \emph{partial crossed product} of $A$ by $G$ under $(\{A_g\}_{g\in G},\{\gamma_g\}_{g\in G})$, denoted by $A\rtimes_{\gamma}G$, is the (full) cross-sectional $\Cst$\nb-algebra of~$\Hilm[B]_{\gamma}=(B_{\gamma_g})_{g\in G}$. The \emph{reduced partial crossed product} $A\rtimes_{\gamma, r}G$ is defined to be the reduced cross-sectional $\Cst$\nb-algebra of~$\Hilm[B]_{\gamma}$.
\end{defn}

Recall that a map $v\colon G\to B$ from $G$ to a unital $\Cst$\nb-algebra~$B$ is said to be a \emph{\Star partial representation} of~$G$ in~$B$ if $v_g$ is a partial isometry for each~$g\in G$ with~$v_e=1$, and the set of partial isometries $\{v_g\mid g\in G\}$ satisfies the relations $$v_{g}^*=v_{g^{-1}}\qquad\text{ and }\qquad v_{g}v_hv_{h^{-1}}=v_{gh}v_{h^{-1}},$$ for all~$g, h\in G$. Let $\gamma=(\{A_g\}_{g\in G},\{\gamma_g\}_{g\in G})$ be a partial action. A \emph{covariant representation} of $(\{A_g\}_{g\in G}, \{\gamma_g\}_{g\in G})$ in $B$ is a pair $(\pi, v)$, where $v\colon G\to B$ is a \Star partial representation and $\pi\colon A\to B$ is a \Star homomorphism, such that for all $g\in G$ and $a\in A_{g^{-1}}$, $$v_g\pi(a)v_{g^{-1}}=\pi(\gamma_g(a)).$$ A covariant representation $(\pi,v)$ of~$(\{A_g\}_{g\in G},\{\gamma_g\}_{g\in G})$ yields a representation $\pi\times v\colon A\rtimes_\gamma G\to B$ induced by the formula $$(\pi\times v)(a\delta_g)=\pi(a)v_g,$$ for $g\in G$ and $a\in A_g$. By \cite[Theorem~13.2]{Exel:Partial_dynamical}, the map $(\pi, v)\mapsto \pi\times v$ gives a one-to-one correspondence between nondegenerate covariant representations of~$(\{A_g\}_{g\in G},\{\gamma_g\}_{g\in G})$ on~$\Hilm[H]$ such that $v_gv_{g^{-1}}$ is the orthogonal projection onto $\pi(A_g)\Hilm[H]=\clsp\{\pi(a)\xi\mid \xi\in\Hilm[H], a\in A_g\}$ and nondegenerate representations of the partial crossed product $A\rtimes_\alpha G$ on $\Hilm[H]$.

\subsection{Toeplitz algebras as partial crossed products}  Suppose that $P$ is a submonoid of a group~$G$. 
By \cite{CELY}*{Section~5.5.2} there is a partial action of~$G$ on~$D_r$ with $D_r\rtimes_r G\cong\Toepr(P)$. We wish to describe this partial action explicitly in
terms of words  and their constructible ideals. For each~$g\in G$, let $$A_{g^{-1}}\coloneqq\overline{\mathrm{span}}\{\1_{K(\alpha)}\mid \alpha\in\Hilm[W],\,\dot{\alpha}=g\}\subseteq D_r.$$ So $A_e=D_r$ by  \lemref{lem:1_KalphaspansD_r}. Also, notice that $A_{g^{-1}}$ is precisely the ideal $D_{g^{-1}}$ defined in \cite[p. 188]{CELY} because in~$\Toepr(P)$ $$\1_{K(\alpha)}=\1_{K(\tilde{\alpha}\alpha)}=\dot{L}_{\tilde{\alpha}}\dot{L}_\alpha=\dot{L}_\alpha^*\dot{L}_\alpha.$$

\begin{prop}[cf. \cite{CELY}*{Section~5.5.2}]\label{prop:partial-const} Let $P$ be a submonoid of a group~$G$. For each $g\in G$, there is a unique \Star isomorphism $\gamma_g\colon A_{g^{-1}}\to A_g$ given on a projection $\1_{K(\alpha)}\in A_{g^{-1}}$ by $$\gamma_g(\1_{K(\alpha)})=\1_{gK(\alpha)}=\1_{K(\tilde{\alpha})}.$$ 
Moreover, $\gamma = (\{A_g\}_{g\in G},\{\gamma_g\}_{g\in G})$ is a partial action of~$G$ on~$D_r$.
\begin{proof} We view $\Bound(\ell^2(P))$ as a closed $\Cst$\nb-subalgebra of~$\Bound(\ell^2(G))$ using the canonical embedding of~$\ell^2(P)$ as a closed subspace of~$\ell^2(G)$. We claim that $$\lambda_g\1_{K(\alpha)}\lambda_{g^{-1}}=\1_{K(\tilde{\alpha})},$$ where $\lambda\colon G\to\Bound(\ell^2(G))$ is the left regular representation of~$G$. To show this, let $h\in G$. Since $\lambda_{g^{-1}}(\delta_h)=\delta_{g^{-1}h}$, it follows that $$(\lambda_g \1_{K(\alpha)}\lambda_{g^{-1}})(\delta_h)=
\begin{cases}
\delta_{h}& \text{if } g^{-1}h\in K(\alpha),\\
0 &\text{otherwise}.
\end{cases}$$
Since  $g=\dot{\alpha}$, \proref{pro:propertiesKalpha} yields $gK(\alpha)=K(\tilde{\alpha})$. 
Hence the automorphism $ X \mapsto \lambda_g X \lambda_{g^{-1}}$ of $\Bound(\ell^2(G))$ restricts to a \Star isomorphism $\gamma_g\colon A_{g^{-1}}\to A_g$ determined by  $\gamma_g(\1_{K(\alpha)}) =\1_{K(\tilde{\alpha})}$.

To see that $A_{g^{-1}}$ is an ideal of~$D_r$, let $\alpha=(p_1,\ldots, p_{2k})$ be such that $\dot{\alpha}=g$. Let $\beta=(q_1,\ldots,q_{2l})\in\Hilm[W]$ with $\dot{\beta}=e$. It follows from \proref{pro:propertiesKalpha}(5) with the roles of $\alpha$ and $\beta$ exchanged that 
\begin{equation*}
\begin{aligned}
\1_{K(\alpha)}\1_{K(\beta)}=\1_{K(\beta)}\1_{K(\alpha)}=\1_{K(\beta)\cap K(\alpha)}=\1_{K(\alpha\beta)}.
\end{aligned}
\end{equation*} This lies in $A_{g^{-1}}$ because $$\dot{\alpha\beta}=\dot{\alpha}\dot{\beta}=\dot{\alpha}=g.$$

Let us now prove that $(\{A_g\}_{g\in G},\{\gamma_g\}_{g\in G})$ satisfies axiom (ii) of Definition \ref{defn: partial-action}. That is, for all $g,h\in G$, one has $$\gamma_g(A_{g^{-1}}\cap A_h)\subseteq A_{gh}.$$ Let $\alpha=(p_1,\ldots,p_{2k})$ and $\beta=(q_1,\ldots, q_{2l})$ be words in~$P$ with $$\dot{\alpha}=g\qquad\text{ and }\qquad\dot{\beta}=h^{-1},$$ so that $\1_{K(\alpha)}\in A_{g^{-1}}$ and $\1_{K(\beta)}\in A_h$. Again we view $\Bound(\ell^2(P))$ as a $\Cst$\nb-subalgebra of $\Bound(\ell^2(G))$ using the canonical inclusion $\ell^2(P)\hookrightarrow\ell^2(G)$. Thus for all $k\in G$, $$\gamma_g(\1_{K(\alpha)}\1_{K(\beta)})(\delta_k)=(\lambda_g\1_{K(\alpha)}\1_{K(\beta)}\lambda_g^*)(\delta_k)=\begin{cases}
\delta_{k}& \text{if } g^{-1}k\in K(\alpha)\cap K(\beta),\\
0 &\text{otherwise}.
\end{cases}$$ 

Now we compute \begin{equation*}
\begin{aligned} g (K(\alpha)\cap K(\beta))=g K(\alpha)\cap gK(\beta)=K(\tilde{\alpha})\cap gK(\beta).
\end{aligned}
\end{equation*} 
Replacing $g$ by $\dot{\alpha}$ in the above and using \proref{pro:propertiesKalpha}, we deduce that  $$ g (K(\alpha)\cap K(\beta))=K(\beta\tilde{\alpha}).$$ Hence $\gamma_g(\1_{K(\alpha)}\1_{K(\beta)})=\1_{K(\beta\tilde{\alpha})}$. Since 
$$\dot{\beta\tilde{\alpha}}=\dot{\beta}\dot{\tilde{\alpha}}=h^{-1}g^{-1}=(gh)^{-1},$$ it follows that $\gamma_g(\1_{K(\alpha)}\1_{K(\beta)})\in A_{gh}$ and so $\gamma_g(A_{g^{-1}}\cap A_h)\subseteq A_{gh}$ as desired. Axiom (iii) of Definition \ref{defn: partial-action} follows from the computation $$\gamma_{gh}(b)=\lambda_{gh}b\lambda_{gh}^*=\lambda_g\lambda_hb\lambda_h^*\lambda_g^*=\lambda_g\gamma_h(b)\lambda_g^*=\gamma_g(\gamma_h(b))$$ for all $b\in A_{g^{-1}}\cap A_h.$ We then conclude that $(\{A_g\}_{g\in G},\{\gamma_g\}_{g\in G})$ is a partial action of~$G$ on~$D_r$. 
\end{proof}
\end{prop}

\begin{rem} Observe that $A_{g^{-1}}=\{0\}$ if and only if $g^{-1}P\cap P=\emptyset$. One direction is obvious because $K(\alpha) \subset g^{-1}P\cap P$ whenever $\dot{\alpha}=g$. For the converse assume $ g^{-1}P\cap P\neq\emptyset$ and take  $p,q \in P$ such that $g^{-1}q=p$; then $\alpha = (e,q,p,e)$ satisfies $\dot{\alpha}= qp\inv = g$ and $\1_{K(\alpha)} = \1_{pP} \in A_{g^{-1}}$. Also notice that if~$p\in P$, then the ideal $A_p$ is the corner determined by the projection $\1_{pP}$ and $A_{p^{-1}} = D_r$.
\end{rem}

\begin{lem}\label{lem:generating-set} Let $P$ be a submonoid of a group $G$.  Let $(\{A_g\}_{g\in G},\{\gamma_g\}_{g\in G})$ be the partial action of~$G$ on~$D_r$ from \proref{prop:partial-const}. Let $\alpha=(p_1,p_2,\ldots,p_{2k})\in\Hilm[W]^k$ with $\dot{\alpha}=g^{-1}$. Then {$$\1_{K(\alpha)}\delta_g=\dot{\delta}_{\tilde{\alpha}}\coloneqq\delta_{p_{2k}^{-1}}\1_{p_{2k-1}P}\delta_{p_{2k-1}}\ldots \delta_{p_2^{-1}}\1_{p_1 P}\delta_{p_1}.$$} In particular, for every~$g\in G$, one has
$$B_{\gamma_g}=\overline{\mathrm{span}}\{ \dot{\delta}_{\beta}\mid\beta\in\Hilm[W], \dot{\beta}=g\}$$ and the full and reduced partial crossed products $D_r\rtimes_\gamma G$ and  $D_r\rtimes_{\gamma, r} G$ are generated as $\Cst$\nb-algebras by the semigroup of isometries~$\{\1_{pP}\delta_p\mid p\in P\}$.
\begin{proof} Let $\tilde{\pi}\colon D_r\rtimes_\gamma G\to\Bound(\Hilm[H])$ be a nondegenerate representation of~$ D_r\rtimes_\gamma G$ on a Hilbert space~$\Hilm[H]$. Let $(\pi, w)$ be the unique nondegenerate covariant representation of~$(\{A_g\}_{g\in G},\{\gamma_g\}_{g\in G})$ on~$\Hilm[H]$ such that $w_gw_{g^{-1}}$ is the orthogonal projection onto $\pi(A_g)\Hilm[H]$ and $\tilde{\pi}=\pi\times w$.  We will prove by induction on~$k$ that \begin{equation*}\label{eq:monomials}\tilde{\pi}(\1_{K(\alpha)}\delta_g)=\pi(\1_{K(\alpha)})w_g=w_{p_{2k}^{-1}}w_{p_{2k-1}}\ldots w_{p_2^{-1}}w_{p_1}=\dot{w}_\alpha^*\end{equation*} for all $\alpha=(p_1,p_2,\ldots,p_{2k})\in \Hilm[W]^k$ and $g=\dot{\tilde{\alpha}}$.

The base case $k=0$ only occurs if $g=e$, and clearly $\pi(\1_{K(\alpha)})w_e=\pi(\1)w_e=w_e$ in this case. Suppose~$k=1$, so that $\alpha=(p_1,p_2)$. Then 
\begin{equation*}
\begin{aligned}
\pi(\1_{K(\alpha)})w_{p_2^{-1}p_1}&=\pi(\gamma_{p_2^{-1}}(\1_{p_1P}\1_{p_2P}))w_{p_2^{-1}p_1}\\&=w_{p_2^{-1}}\pi(\1_{p_1P}\1_{p_2P})w_{p_2}w_{p_2^{-1}p_1}\\&= w_{p_2^{-1}}\pi(\1_{p_2P})\pi(\1_{p_1P})w_{p_1}=w_{p_2^{-1}}w_{p_1}.
\end{aligned}
\end{equation*}We used above that $w\colon G\to \Bound(\Hilm[H])$ is a \Star partial representation and $\1_{pP}$ is the unit of the ideal $A_{p}$, so that $w_{p}w_{p^{-1}}=\pi(\1_{pP})$ for all~$p\in P$.

Now fix $k>1$ and assume as induction hypothesis that \begin{equation*}\tilde{\pi}(\1_{K(\beta)}\delta_h)=\pi(\1_{K(\beta)})w_h=\dot{w}_\beta^*\end{equation*} for all $\beta\in\Hilm[W]^{k-1}$ and $h=\dot{\tilde{\beta}}$. Let $\alpha\in\Hilm[W]^k$ and $g=\dot{\tilde{\alpha}}$. Set $\alpha'=(p_1,p_2,\ldots,p_{2k-1}, e)\in\Hilm[W]^k$. Notice that $$\1_{K(\alpha)}=\gamma_{p_{2k}^{-1}}(\1_{K(\alpha')}\1_{p_{2k}P}).$$ Hence \begin{equation*}\begin{aligned}
\tilde{\pi}(\1_{K(\alpha)}\delta_g)=\pi(\1_{K(\alpha)})w_g&=w_{p_{2k}^{-1}}\pi(\1_{K(\alpha')})\pi(\1_{p_{2k}P})w_{p_{2k}}w_g\\&=w_{p_{2k}^{-1}}\pi(\1_{K(\alpha')})w_{p_{2k}g}\\&=w_{p_{2k}^{-1}}\pi(\1_{K(\alpha')})w_{\dot{\tilde{\alpha'}}}.
\end{aligned}
\end{equation*}
 Observe that we still have $\alpha'\in \Hilm[W]^k$. Put $\alpha''\coloneqq (p_1,p_2,\ldots,p_{2k-3},p_{2k-2})$. So $\alpha''\in\Hilm[W]^{k-1}$. Also, $$\pi(\1_{K(\alpha')})=\pi(\gamma_{p_{2k-1}}(\1_{K(\alpha'')}))=w_{p_{2k-1}}\pi(\1_{K(\alpha'')})w_{p_{2k-1}^{-1}}.$$ Therefore 
\begin{equation*}
\begin{aligned}w_{p_{2k}^{-1}}\pi(\1_{K(\alpha')})w_{\dot{\tilde{\alpha'}}}&=w_{p_{2k}^{-1}}w_{p_{2k-1}}\pi(\1_{K(\alpha'')})w_{p_{2k-1}^{-1}}w_{\dot{\tilde{\alpha'}}}\\&=w_{p_{2k}^{-1}}w_{p_{2k-1}}\pi(\1_{K(\alpha'')})w_{p_{2k-1}^{-1}}w_{p_{2k-1}}w_{\dot{\tilde{\alpha''}}}\\&=w_{p_{2k}^{-1}}w_{p_{2k-1}}\pi(\1_{K(\alpha'')})w_{\dot{\tilde{\alpha}}''}.
\end{aligned}
\end{equation*} We can now apply our induction hypothesis to~$\alpha''$ to conclude that $$\tilde{\pi}(\1_{K(\alpha)}\delta_g)=\pi(\1_{K(\alpha)})w_g=w_{p_{2k}^{-1}}w_{p_{2k-1}}\ldots w_{p_2^{-1}}w_{p_1}$$ as asserted. Since we can always take a covariant pair $(\pi, w)$ such that~$\pi$ is a faithful representation of~$D_r$ and, for all~$p\in P$, one has $w_{p^{-1}}=(\pi\times w)(\delta_{p^{-1}})$ and $w_p=(\pi\times w)(\1_{pP}\delta_p)$, we deduce that $$\1_{K(\alpha)}\delta_g=\delta_{p_{2k}^{-1}}\1_{p_{2k-1}P}\delta_{p_{2k-1}}\ldots \delta_{p_2^{-1}}\1_{p_1 P}\delta_{p_1}$$ in~$B_{\gamma_g}=A_g\delta_g$. Hence \begin{equation*}\begin{aligned}B_{\gamma_g}&=\overline{\mathrm{span}}\{ \delta_{q_1^{-1}}\1_{q_2P}\delta_{q_2}\ldots \delta_{q_{2l-1}^{-1}}\1_{q_{2l}P}\delta_{q_{2l}}\mid l\geq 0, q_1^{-1}q_2\cdots q_{2l-1}^{-1}q_{2l}=g\}\\&=\overline{\mathrm{span}}\{\dot{\delta}_\beta\mid\beta\in\Hilm[W], \dot{\beta}=g\}\end{aligned}\end{equation*} for all~$g\in G$. The last assertion in the statement follows because $\bigoplus_{\substack{g\in G}}B_{\gamma_g}$ is dense in~$D_r\rtimes_\gamma G$ and $D_r\rtimes_{\gamma, r}G$. This finishes the proof of the lemma.
\end{proof}
\end{lem}

By  \cite[Theorem~5.6.41]{CELY} the reduced Toeplitz $\Cst$\nb-algebra $\Toepr(P)$ is canonically isomorphic to the reduced partial crossed product 
$D_r\rtimes_{\gamma, r}G$.
  We show next that the full version of this isomorphism holds for  $\Toepu(P)$. As a byproduct of our construction, we also recover  the reduced result. 

\begin{thm}\label{thm:partial-picture} Let $P$ be a submonoid of a group $G$.  Let $(\{A_g\}_{g\in G},\{\gamma_g\}_{g\in G})$  be the partial action of~$G$ on~$D_r$ from \proref{prop:partial-const}. Then the map $t_p\mapsto\1_{pP}\delta_p$ induces an isomorphism $$\Toep_u(P)\cong D_r\rtimes_{\gamma}G.$$ In addition, $L_p\mapsto \1_{pP}\delta_p$ gives rise to an isomorphism between reduced $\Cst$\nb-algebras $\Toepr(P)\cong D_r\rtimes_{\gamma, r} G.$
\begin{proof} To see that the map that sends $t_p$ to $\1_{pP}\delta_p$ induces a surjective \Star homomorphism $\psi\colon \Toep_u(P)\to D_r\rtimes_\gamma G$, notice that $$\1_{K(\alpha)}\delta_e = \1_{K(\tilde{\alpha})}\delta_e= \delta_{p_1^{-1}}\1_{p_2P}\delta_{p_2}\ldots \delta_{p_{2k-1}^{-1}}\1_{p_{2k}P}\delta_{p_{2k}}=\dot{\delta}_\alpha$$ in~$D_r\rtimes_\gamma G $ whenever $\alpha=(p_1,\ldots,p_{2k-1},p_{2k})\in\Hilm[W]^k$ satisfies $\dot{\alpha}=e$. Hence the defining relations (T1)--(T4) of Definition \ref{def:toeplitz-semigroup} are satisfied in~$D_r\rtimes_\gamma G $ and so the map $t_p\mapsto \1_{pP}\delta_p$ extends to a \Star homomorhism $\psi\colon \Toep_u(P)\to D_r\rtimes_\gamma G$. This is surjective by \lemref{lem:generating-set}.

In order to construct an inverse for~$\psi$, for each $g\in G$, consider the subspace of~$\Toep_u(P)$ given by $$B_g= \clsp\{\dot{t}_\alpha\mid \alpha\in\Hilm[W],\, \dot{\alpha}=g\}.$$ Then $\psi\colon \Toep_u(P)\to D_r\rtimes_\gamma G$ is faithful when restricted to~$B_g$ for all~$g\in G$, because it is so on~$B_e=D_u$ and $b^*b\in B_e$ for all~$b\in B_g$. Put $\psi_g\coloneqq\psi_{\restriction_{B_g}}$. Then $\psi_g\colon B_g\to A_g\delta_g$ is an isomorphism by \lemref{lem:generating-set}. Thus we can define a representation $\psi'$ of $(\{A_g\}_{g\in G},\{\gamma_g\}_{g\in G})$ in $\Toep_u(P)$ by $$\psi'(a\delta_g)=\psi_g^{-1}(a\delta_g)$$ for all $g\in G$ and $a\in A_g$. Since $B_gB_h\subseteq B_{gh}$ and $B_g^*=B_{g^{-1}}$, it follows that $\psi_g^{-1}(b)^*=\psi^{-1}_{g^{-1}}(b^*)$ and $\psi_g^{-1}(b)\psi_h^{-1}(c)=\psi_{gh}^{-1}(bc)$ for all $g,h\in G$, $a\in \lambda(B_g)$, $b\in\lambda(B_h)$.  Thus \begin{equation*}\begin{aligned}
\psi'(a\delta_g)\psi'(b\delta_h)&=\psi_g^{-1}(a\delta_g)\lambda_h^{-1}(b\delta_h)\\&=\psi_{gh}^{-1}(a\delta_g\cdot b\delta_h)\\&=\psi'(a\delta_g\cdot b\delta_h).
\end{aligned}
\end{equation*} Similarly, one can show that $\psi'$ preserves the involution operation $^*\colon A_g\delta_g\to A_{g^{-1}}\delta_{g^{-1}}$. Hence it gives rise to a \Star homomorphism $\tilde{\psi'}\colon D_r\rtimes_\gamma G\to \Toep_u(P)$ by universal property of~$ D_r\rtimes_\gamma G$. Since $D_r\rtimes_{\gamma}G$ is generated as a $\Cst$\nb-algebra by the set of isometries $\{\1_{pP}\delta_p\mid p\in P\}$, we see that $\tilde{\psi'}$ is the inverse of~$\psi$ as desired. 

It remains to establish the isomorphism $\Toepr(P)\cong D_r\rtimes_{\gamma, r}G$. Let $\Lambda\colon D_r\rtimes_{\gamma}G\to D_r\rtimes_{\gamma, r}G$ be the left regular representation associated to the partial action $(\{A_g\}_{g\in G}, \{\gamma_g\}_{g\in G})$ and let $E_{\Lambda}\colon D_r\rtimes_{\gamma}G\to D_r\delta_e$ be the corresponding conditional expectation. By \cite[Proposition~19.7]{Exel:Partial_dynamical}, $$\ker\Lambda=\{c\in  D_r\rtimes_{\gamma}G\mid E_\Lambda(c^*c)=0\}.$$ Hence \corref{cor:reg-kernel} and the commutativity of the diagram
$$\xymatrix{
     \Toep_u(P)\ar@{->}[r]^{\psi} \ar@{->}[d]_{E_u}&
   D_r\rtimes_{\gamma}G \ar@{->}[d]_{E_\Lambda}\\
    D_u\ar@{->}[r]^{\psi_{\restriction_{D_u}}} &
   D_r\delta_e
    }$$ yield $\psi(\ker\lambda^+)=\ker\Lambda$. Thus $\Toepr(P)\cong D_r\rtimes_{\gamma, r}G$ via an isomorphism that identifies the canonical generating elements.
\end{proof}
\end{thm}
\begin{rem}
It is pointed out in \cite[Remark  5.6.46]{CELY} that instead of defining the full semigroup $\Cst$\nb-algebra 
 as an inverse semigroup $\Cst$\nb-algebra, one could focus on the full $\Cst$\nb-algebra of the partial transformation groupoid $G\ltimes \Omega_P$ or, equivalently, the full partial crossed product $D_r\rtimes_\gamma G$. When we combine  \thmref{thm:partial-picture} with  \proref{pro:charactofToepu}, we see that the latter $\Cst$\nb-algebra is
 canonically isomorphic to the  $\Cst$\nb-algebra $\Cst_s{}^{(\cup)}(P)$ mentioned in \cite[Section~3]{Li:Semigroup_amenability}. The reason for the isomorphism is that  both coincide with our $\Toepu(P)$. 
As suggested also in \cite[Remark  5.6.46]{CELY}, the full partial crossed product version would  yield stronger, `independence-free', versions  of \cite[Theorem 5.6.44]{CELY} and \cite[Corollary  5.6.45]{CELY}, see below. Arguably, choosing such a path is more justified and the stronger results are more  appealing now that we have introduced the $\Cst$\nb-algebra $\Toepu(P)$ via a transparent presentation.
\end{rem}

\begin{thm}[cf. \cite{CELY}*{Theorem 5.6.44}]\label{thm:modifiedresult}
Suppose $P$ is a submonoid of a group $G$ and consider the following conditions:
\begin{enumerate}
\item $\Toepu(P)$ is nuclear;
\item $\Toepr(P)$ is nuclear;
\item  the groupoid $G\ltimes \Omega_P$ is amenable;
\item the left regular representation $\lambda^+\colon \Toepu(P)\to \Toepr(P)$ is faithful.
\end{enumerate}
Then \textup{(1) $\iff$ (2) $\iff$ (3) $\implies$ (4).} 
If $G$ is exact, then \textup{(4) $\implies$ (1),} and all conditions are equivalent.
\begin{proof} By \thmref{thm:partial-picture}, $\Toepu(P)$  and  $\Toepr(P)$ are the full and the reduced  crossed product of the partial action of $G$ on $\Omega_P$, and hence are respectively isomorphic to the full and reduced groupoid $\Cst$\nb-algebras of the partial transformation groupoid $G\ltimes \Omega_P$. So the first statement follows directly from \cite[Theorem 5.6.7]{CELY}. If $G$ is exact and~$P$ satisfies independence, the implication (4)$\implies$(1)  has already been obtained, for $\Cst_s(P)$, in \cite[Corollary~5.5]{BFS2020}; for general $P$ we use \thmref{thm:partial-picture} and \cite[Theorem~4.10]{BFS2020}.
\end{proof}
\end{thm}

\begin{cor} [cf. \cite{CELY}*{Corollary 5.6.45}]\label{cor:ifamenallequiv}
If the monoid $P$ embeds in an amenable group $G$, then all the conditions in \thmref{thm:modifiedresult} hold.
\begin{proof}  
Since $G$ is amenable, all the conditions in \thmref{thm:modifiedresult} are equivalent. As indicated in the  proof of  \cite[Corollary  5.6.45]{CELY},  
it also follows that the groupoid $G\ltimes \Omega_P$ is amenable  by Theorem~20.7 and Theorem~20.10 in~\cite{Exel:Partial_dynamical}. This proves the corollary.
\end{proof}

\end{cor}

\begin{rem}
We can also see now that the conclusion of \cite[Theorem  5.6.42]{CELY}  holds, without the assumption of independence, for the $\Cst$\nb-algebra $\Toepu(P)$ instead of $\Cst_s(P)$.
The proof goes along the same lines, but relies on our \thmref{thm:modifiedresult} instead of  \cite[Theorem 5.6.44]{CELY}.
\end{rem}

\section{Faithful representations  of $\Toepr(P)$} \label{sec:faithfulness}
Our main purpose in this section is to study faithfulness of representations of~$\Toepr(P)$, for which we use the partial crossed product picture of~$\Toepr(P)$ as described in Section~\ref{sec:partialactions}.  The first result reduces the question of whether a representation of $\Toepr(P)$ is faithful to whether its restriction to the crossed product of~$D_r$ by the action of the group of units is faithful. This generalizes earlier results, from~\cite{CELY}, valid for trivial unit group  and from \cite{NNN1} about right LCM monoids. It turns out that topological freeness of the partial action of $G$ is equivalent to that of its restriction to the group of units, and we characterize this in terms of the action of units on  constructible right ideals. We finish the section by deriving a general uniqueness theorem for the $\Cst$\nb-algebra generated by a collection of elements satisfying  
the presentation of $\Toepu(P)$. 

\subsection{A characterization of faithful representations.} If $P$ is embedded as a submonoid in a group~$G$, then  the group of units of $P$ is the subgroup $P^* \coloneqq P\cap P\inv$ of $G$.
 The partial action of $G$ restricts to an action of $P^*$ on
the diagonal subalgebra $D_r$, and the crossed product $D_r\rtimes_{\gamma, r} P^*$ embeds canonically in the partial crossed product $\Toepr(P) \cong D_r \rtimes_{\gamma, r} G$. 
This observation plays a crucial role in the following characterization of faithful representations of $\Toepr(P)$.

\begin{thm}\label{thm:faithfulreps}
Every nontrivial ideal of {$\Toepr(P) \cong D_r \rtimes_{\gamma, r} G$} has nontrivial intersection with the subalgebra  $D_r\rtimes_{\gamma, r} P^*$. In other words, a representation of $\Toepr(P)$ is faithful if and only if it is faithful on $D_r\rtimes_{\gamma, r} P^*$. 
\begin{proof}  Let $E_r\colon\Toepr(P) \to D_r$ be the canonical faithful conditional expectation of $\Toepr(P)$ onto the diagonal subalgebra. In order to prove the theorem, it suffices to show that if a representation $\rho \colon\Toepr(P) \to \Bound(\H)$ is faithful on the reduced crossed product {$D_r\rtimes_{\gamma, r} P^*$, then there is a
conditional expectation $\varphi_\rho$, defined on the image of $\rho$ and having range $\rho(D_r)$, 
so that the square}
\begin{equation}\label{eqn:commdiagram}
    \centering
    \begin{tikzpicture}[scale=0.6]
    \node at (-4,4) {$\Toepr(P) \cong D_r \rtimes_{\gamma, r} G$\ \ \ \ };
    \node at (4,4) {$\rho(\Toepr(P))$};
    \node at (-4,0) {$D_r$};
    \node at (4,0) {$\rho(D_r)$};
    
    \draw[->] (-1.5,4) -- (2,4);
    \draw[->] (-2,0) -- (2,0);
	\draw[->] (-4,3) -- (-4,1);
    \draw[->] (4,3) -- (4,1);
  
    \node at (-4.5, 2) {$E_r$};
    \node at (4.5, 2) {$\varphi_\rho$};
    \node at (0, 4.5) {$\rho$};
    \node at (0, .5) {$\rho\restriction_{D_r}$};
      \end{tikzpicture}
\end{equation} 
commutes.
The usual argument then completes the proof: if $\rho(b) =0$, then $(\varphi_\rho\circ \rho)(b^*b) =0$, and hence 
$(\rho\restriction_{D_r} \circ E_r)(b^*b) =0$. Since $E_r(b^*b) \in D_r$, this implies that $E_r (b^*b) =0$. Thus $b^*b =0$ and $b=0$ because $E_r$ is faithful.

We denote by $A_g^c$ the dense \Star subalgebra of~$A_g$ spanned by the set $\{\1_{K(\alpha)}\mid \dot{\alpha}=g^{-1}\}$. Thus $\bigoplus_{\substack{g\in G}}A_g^c\delta_g$ is a dense \Star subalgebra of~$ D_r \rtimes_{\gamma,r} G$. In order to show that the conditional expectation~$\varphi_\rho$ exists we show that for each (finite) linear combination 
$\sum_{g\in F} a_g \delta_g $ in $\bigoplus_{\substack{g\in G}}A_g^c\delta_g$
(in which we may assume there is a term $a_e$ by setting it to be zero if necessary), there exist an element $p\in P$ and a projection $Q \in D_r$
such that 
\begin{enumerate}
\item $| a_e(p) |= \| a_e \| $;
\item $Q (p) = 1$;
\item $Q a_g\delta_g Q = 0 $ for every $g\in F\setminus pP^* p\inv$; and
\item $\1_{pP} Q = Q= Q \1_{pP}$.
\end{enumerate}
We relegate the proof of existence of~$p$ and~$Q$ to  \lemref{lem:pQproperties} below. 
Supposing for now that $p$ and $Q$ are as above, we have the following estimate.
\begin{eqnarray*}
\Big\| \rho\big(\sum_{g\in F} a_g \delta_g\big)\Big\| &\geq& \Big\|\rho(Q) \sum_{g\in F} \rho\big(a_g \delta_g\big) \rho(Q)\Big\| \\\\
&=& \Big\| \sum_{g\in F \cap pP^* p\inv} \rho\big(\1_{pP}a_g Q\delta_gQ\1_{pP}\big)\Big\| \qquad\text{because of (3)}\\\\
&=& \Big\| \sum_{g\in F \cap pP^* p\inv} \rho\big(\1_{pP}\delta_p\delta_{p^{-1}}a_g Q\delta_gQ\1_{pP}\delta_p\delta_{p^{-1}}\big)\Big\| 
\end{eqnarray*}
We continue by changing the summation index from $g\in F \cap pP^* p\inv$ to $u \coloneqq p\inv g p \in p\inv F p \cap P^*$ and using the multiplication rule \eqref{eq:multiplicationmap} for generators of the partial crossed product. 
\begin{eqnarray*}
\Big\| \rho\big(\sum_{g\in F} a_g \delta_g\big)\Big\| &\geq&\Big\| \sum_{u\in p\inv F p \cap P^*} \rho\big(\1_{pP}\delta_p\delta_{p^{-1}}a_{pup\inv} Q \delta_{pup\inv} Q\1_{pP}\delta_p\delta_{p^{-1}}\big)\Big\| \\\\
&=&\Big\| \sum_{u\in p\inv F p \cap P^*} \rho(\1_{pP}\delta_p)\rho\big(\delta_{p^{-1}}a_{pup\inv} Q \delta_{pup\inv} Q\1_{pP}\delta_p\big)\rho(\delta_{p^{-1}})\Big\| \\\\
&=&\Big\| \sum_{u\in p\inv F p \cap P^*} \rho(\1_{pP}\delta_p)\rho\big((\gamma_{p^{-1}}(\1_{pP}a_{pup\inv} Q) \delta_{up\inv}) Q\1_{pP}\delta_p\big)\rho(\delta_{p^{-1}})\Big\| \\\\
&=&\Big\| \sum_{u\in p\inv F p \cap P^*} \rho(\1_{pP}\delta_p)\rho\big(\gamma_{up\inv}(\gamma_{pu\inv}(\gamma_{p^{-1}}(\1_{pP}a_{pup\inv} Q))Q\1_{pP}) \delta_{u}\big)\rho(\delta_{p^{-1}})\Big\| \\\\
&=&\Big\| \sum_{u\in p\inv F p \cap P^*} \rho\big(\gamma_{up\inv}(\gamma_{pu\inv}(\gamma_{p^{-1}}(\1_{pP}a_{pup\inv} Q))Q\1_{pP}) \delta_{u}\big)\Big\|.
\end{eqnarray*}
Since this sum is in the crossed product by the action of $P^*$ and  $\rho$ is assumed to be faithful there,
\begin{eqnarray*}
\Big\| \rho\big(\sum_{g\in F} a_g \delta_g\big)\Big\| &\geq& \Big\| \sum_{u\in p\inv F p \cap P^*} \gamma_{up\inv}(\gamma_{pu\inv}(\gamma_{p^{-1}}(\1_{pP}a_{pup\inv} Q))Q\1_{pP}) \delta_{u}\Big\|  \\\\ &\geq&\Big\| \gamma_{p\inv}(\1_{pP}a_{e} Q\1_{pP})\Big\|
\qquad \text{ because $E_r$ is contractive,} \\\\
&=&\Big\|\1_{pP}a_{e} Q\1_{pP}\Big\| 
\\\\
&=& \|a_e\| \qquad \text{because of (4)  and (1),} \\\\
&=&\|\rho(a_e)\| \qquad \text{ because $\rho$ is faithful on $D_r$}.
\end{eqnarray*}

Thus, the map $\sum_{g\in F} \rho(a_g \delta_g) \mapsto \rho(a_e)$ is well defined and contractive on a dense \Star subalgebra of $\rho(\Toepr(P))$. So it extends uniquely by continuity to give a conditional expectation $\varphi_\rho : \rho(\Toepr(P)) \to  \rho(D_r)$ such that the diagram
\eqref{eqn:commdiagram} commutes.
\end{proof}
\end{thm}
The following lemma can be extracted from the proof of~\cite[Theorem~5.7.2]{CELY}; we formulate it explicitly because it is useful in a couple of places.
\begin{lem}\label{lem:minilemma}
 Let $g\in G$ and $p\in P$. The following are equivalent.
 \begin{enumerate}
 \item $gpP=pP$;
 \item $g\in pP^*p^{-1}$;
  \item $gp \in pP^*$;
  \end{enumerate}
 \begin{proof}
  Suppose  that $gpP=pP$ and take $x, y\in P$ such that $gp=px$ and $p=gpy$. Multiplying the first identity on the right by $y$, we obtain $gpy=pxy$ and so $xy=e$. Since $P$ is contained in a group, we deduce that $x$ and $y$ are invertible, that is, $x,y\in P^*$. Thus $g=pxp^{-1}\in pP^*p^{-1}$. This proves that (1)$\implies$(2).
The converse holds because if $g=pxp^{-1}$ with $x\in P^*$, then $gpP=pxP=pP$. Clearly (3) is just a reformulation of (2).
 \end{proof}
  \end{lem}

\begin{lem}\label{lem:pQproperties}
Let $P$ be a submonoid of a group $G$.  Let $F\subseteq G$ be a finite set and let $a=\sum_{\substack{g\in F}}a_g\delta_g$ be an element of the dense \Star subalgebra $\bigoplus_{g\in G}A^c_g\delta_g$ of $D_r\rtimes_{\gamma,r}G$. Then there exist a point $p\in P$ and a projection $Q\in D_r$ 
with the properties \textup{(1)}--\textup{(4)} listed in the proof of  \thmref{thm:faithfulreps}. 
\begin{proof} The proof is an adaptation of the strategy of \cite[Lemma~3.2]{LACA1996415} combined with the main idea of the proof of~\cite[Theorem~~5.7.2]{CELY}. Let $A\subseteq \J$ be a finite collection of constructible right ideals and let $\{\lambda_S\mid S\in A\}\subset \CC$ be scalars such that $a_e=\sum_{S\in A}\lambda_S \1_S$. Because~$a_e$ is a finite linear combination of projections in $\ell^\infty(P)$,  there is $p\in P$ such that $\|a_e\|=|a_e(p)|$. Consider the subset of~$A$ given by $$F_p\coloneqq \{S\in A\mid p\in S\}  $$ 
and put $$Q_{F_p}\coloneqq \prod_{\substack{S\in F_p}}\1_S\prod_{\substack{S'\in A\setminus F_p}}(\1-\1_{S'}).$$ Note that $Q(p)=1$ because $S'\in A\setminus F_p$ implies $(\1-\1_{S'})(p)=1$. We are going to modify $Q_{F_p}$ by taking a subprojection $Q\leq Q_{F_p}$ with $Q(p)=1$ and $Qa_g\delta_gQ=0$ for all~$g\in F\setminus pP^*p\inv$. To do so, we need to find, in the context of a general submonoid of a group, the correct analogues of the elements $ad_{x,y}$ used in the setting of quasi-lattice orders to define a projection~$Q$ right after \cite[equation~(3.6)]{LACA1996415}.

Let  $a=\sum_{\substack{g\in F}}a_g\delta_g$ be as in the statement of the lemma. Take $g\in F\setminus pP^*p\inv$. By \lemref{lem:minilemma}, $gpP\neq pP$ so we have either $gpP\cap pP\subsetneq pP$, or $gpP\cap pP\subsetneq gpP$. This latter situation is equivalent to $pP\cap g^{-1}pP\subsetneq pP$. Now since $a_g$ lies in the linear span~$\{\1_{K(\alpha)}\mid \dot{\alpha}=g^{-1}\}$, we can find~$m\in \NN$ and words $\alpha_1,\ldots, \alpha_m\in \Hilm[W]$ with $\dot{\alpha_i}=g^{-1}$ such that $a_g=\sum_{\substack{i=1}}^m\lambda_i\1_{K(\alpha_i)}$, where $\lambda_i\in\CC$ for all $i\in\{1,\ldots, m\}$. For each~$i=1,\ldots, m$, we define 
$$d^p_{\alpha_i}=\begin{cases}
\1_{ K((p,e)\alpha_i)}& \text{if } gpP\cap pP\subsetneq pP,\\
\1_{K((p,e)\tilde{\alpha_i})}&  \text{if } g^{-1}pP\cap pP\subsetneq pP.
\end{cases}$$ 
We claim that $d^p_{\alpha_i}(p)=0$ for all~$i\in\{1,\ldots, m\}$. Indeed, if $gpP\cap pP\subsetneq pP$, we have that $d^p_{\alpha_i}(p)= 0$ for all~$i=1,\ldots,m$ because $p\not \in gpP$ and $K((p,e)\alpha_i)\subseteq gpP\cap P$. In case $g^{-1}pP\cap pP\subsetneq pP$, we see that $d^p_{\alpha_i}(p)= 0$ because $K((p,e)\tilde{\alpha_i})\subseteq g^{-1}pP\cap P$ and $p\not\in g^{-1}pP$. 

We set $Q_g\coloneqq\prod_{\substack{i=1}}^m(\1-d^p_{\alpha_i})$. We pause here to show that  $Q_g\1_{pP}a_g\delta_g\1_{pP}Q_g=0$.  Suppose that $gpP\cap pP\subsetneq pP$. Using the multiplication rule~\eqref{eq:multiplicationmap} in the partial crossed product $D_r\rtimes_{\gamma, r}G$, we compute 
\begin{equation*}
\begin{aligned}
Q_g\1_{pP}\1_{K(\alpha_i)}\delta_g\1_{pP} Q_g&=Q_g\1_{pP}\gamma_g(\gamma_{g^{-1}}(\1_{K(\alpha_i)})\1_{pP})\delta_gQ_g\\&=\1_{pP}Q_g\gamma_g(\1_{K(\tilde{\alpha_i})}\1_{K(e,p,p,e)})\delta_gQ_g\\&=\1_{pP}Q_g\gamma_g(\1_{K(\tilde{\alpha_i}(e,p,p,e))})\delta_gQ_g
\\&=\1_{pP}Q_g\1_{K((e,p,p,e)\alpha_i)}\delta_gQ_g\\&=\1_{pP}Q_g\1_{K((p,e)\alpha_i)}\delta_gQ_g.
\end{aligned}
\end{equation*} 
Since $Q_g$ has a factor $\1-\1_{K((p,e)\alpha_i)}$, we deduce that $Q_g\1_{pP}\1_{K(\alpha_i)}\delta_g\1_{pP}Q_g=0$. 

Assume we are in the case $g^{-1}P\cap pP\subsetneq pP$. Then
\begin{equation*}
\begin{aligned} Q_g\1_{pP}\1_{K(\alpha_i)}\delta_g\1_{pP}Q_g&=Q_g\1_{K(\alpha_i(e,p,p,e))}\delta_g\1_{pP}Q_g\\&=Q_g\1_{K(\alpha_i(e,p,p,e))}\1_{K(\alpha_i(e,p,p,e))}\delta_g\1_{pP}Q_g\\&=Q_g\1_{K(\alpha_i(e,p,p,e))}\delta_g\gamma_{g^{-1}}(\1_{K(\alpha_i(e,p,p,e))})Q_g\1_{pP}\\&=Q_g\1_{K(\alpha_i(e,p,p,e))}\delta_g\1_{K((e,p,p,e)\tilde{\alpha_i})}Q_g\1_{pP}\\&=Q_g\1_{K(\alpha_i(e,p,p,e))}\delta_g\1_{K((p,e)\tilde{\alpha_i})}Q_g\1_{pP}.
\end{aligned}
\end{equation*} Again this is zero because $Q_g$ has a factor $\1-\1_{K((p,e)\tilde{\alpha_i})}$. Hence $$Q_g\1_{pP}a_g\delta_g\1_{pP}Q_g=\sum_{\substack{i=1}}^m\lambda_iQ_g\1_{pP}\1_{K(\alpha_i)}\delta_g\1_{pP}Q_g=0.$$

Finally, we set $$Q\coloneqq Q_{F_p}\cdot\1_{pP}\cdot\prod_{\substack{g\in F}\setminus (pP^*p^{-1})}Q_g.$$ Then $Q$ is projection in~$D_r$ since it is a finite product of projections in~$D_r$. Also, $Q$ is a subprojection of~$Q_{F_p}$ with~$Q(p)=1$ and so $\|Qa_eQ\|=\|a_e\|$. That $Qa_g\delta_g Q=0$ for all $g\in F\setminus pP^*p^{-1}$ follows from the computation above. This completes the proof of the lemma.
\end{proof}
\end{lem}

\begin{rem}
When we apply  \thmref{thm:faithfulreps}  to a right LCM monoid $P$ that embeds in a group, we recover the group embeddable case of 
 \cite[Theorem~7.4]{NNN1} without having to assume condition (C1) of \cite[Definition 2.6]{NNN1}.
 Our result is about the reduced Toeplitz $\Cst$\nb-algebra $\Toepr(P)$, but so is \cite[Theorem~7.4]{NNN1}, in view of \cite[Corollary 7.5]{NNN1}. 
 The necessary and sufficient conditions match because the inner core $\mc_I $ is naturally isomorphic to $D_r\rtimes_r P^*$ by a standard argument using the assumed faithful conditional expectation from $\mc_I$ onto $D_r$. 
\end{rem}

In the right LCM case, condition (C1) and the extra assumptions on the quotient semigroup $P/P^*$ 
are needed to produce a conditional expectation in the proof of \cite[Theorem~7.4]{NNN1}. Thus,
it is natural to wonder whether a faithful conditional expectation from $\Toepr(P)$ to $D_r\rtimes_r P^*$ always exists when $P$ is a submonoid of a group.
We show next  that this is indeed the case.
  
\begin{prop} Let $P$ be a submonoid of a group~$G$. View $\Toepr(P)$ as the reduced partial crossed product $D_r\rtimes_{\gamma,r}G$. Then the map 
\[
a\delta_g \mapsto \begin{cases}a\delta_g & \text{ if } g\in P^*, \\ 0 & \text{ otherwise,}\end{cases}
\]
 extends, by linearity and continuity, to a faithful conditional expectation of $D_r \rtimes_{\gamma, r} G$ onto $D_r \rtimes_{\gamma, r} P^*$.
\begin{proof} Let $\lambda\colon G\to \Hilm[B](\ell^2(G))$ be the left regular representation of~$G$. Let $Q_*\in\Hilm[B](\ell^2(G))$ denote the orthogonal projection of~$\ell^2(G)$ onto $\ell^2(P^*)$. Because $P^*$ is a group, we have $$
Q_*\lambda_gQ_* = \begin{cases}\lambda^{P^*}_g& \text{ if } g\in P^*, \\ 0 & \text{ otherwise,}\end{cases}
$$ where $\lambda^{P^*}\colon P^*\to \Hilm[B](\ell^2(P^*))$ denotes the left regular representation of~$P^*$. 

Now  let $\rho\colon D_r\rtimes_{\gamma,r}G\to \Hilm[B](\Hilm[H])$ be a faithful representation of $D_r\rtimes_{\gamma,r}G$ on a Hilbert space $\Hilm[H]$. Observe that
\[
(1\otimes Q_*)(\rho(a\delta_g)\otimes \lambda_g)(1\otimes Q_*)= \begin{cases}\rho(a\delta_g)\otimes \lambda_g^{P^*}& \text{ if } g\in P^*, \\ 0 & \text{ otherwise.}\end{cases}
\] By Fell's absorption principle \cite[Proposition~18.4]{Exel:Partial_dynamical}, the map $$a\delta_g\in A_g\delta_g\mapsto \rho(a\delta_g)\otimes\lambda_g\in\Hilm[B](\Hilm[H]\otimes\ell^2(G))$$ yields a representation of $\gamma=(\{A_g\}_{g\in G},\{\gamma_g\}_{g\in G})$ whose integrated form factors through the reduced partial crossed product $D_r\rtimes_{\gamma,r}G$. Let $\tilde{\rho}\colon D_r\rtimes_{\gamma,r}G\to\Hilm[B](\Hilm[H]\otimes\ell^2(G))$ be the induced homomorphism. Then $\tilde{\rho}$ is faithful because $\rho$ is injective on~$D_r$. A similar reasoning shows that the map $$a\delta_u\in D_r\rtimes_{\gamma,r}P^*\mapsto \rho(a\delta_u)\otimes\lambda^{P^*}_u\in \Hilm[B](\Hilm[H]\otimes \ell^2(P^*))$$ induces a faithful representation of $D_r\rtimes_{\gamma,r}P^*$ on~$\Hilm[H]\otimes \ell^2(P^*)$. We obtain a map $E_*\colon D_r\rtimes_{\gamma,r}G\to D_r\rtimes_{\gamma,r}P^*$ by sending an element $b$ to $(1\otimes Q_*)\tilde{\rho}(b)(1\otimes Q_*)$ and then identifying the result with the corresponding element of~$D_r\rtimes_{\gamma,r}P^*$. This is the desired conditional expectation. It is faithful because the map obtained by composing $E_*$ with the canonical conditional expectation of $D_r\rtimes_{\gamma, r}P^*$ onto~$D_r$ is precisely the usual diagonal conditional expectation $E_r\colon \Toepr(P)\to D_r$.
\end{proof}
\end{prop}

\subsection{The action of $P^*$ on the spectrum.} The underlying reason why the strategy of \cite[Section~3]{LACA1996415} works here is that the set of characters determined by evaluation at points in~$P$ is dense in~$\Omega_P$. 
 However, a quick comparison to the original result from \cite{LACA1996415} reveals a modification; indeed, for general submonoids of groups, \thmref{thm:faithfulreps} only reduces faithfulness of representations of $\Toepr(P)$  to faithfulness on the subalgebra $D_r \rtimes_{\gamma,r} P^*$ instead of on $D_r$.  
 If we still wish to know whether a representation of  $\Toepr(P)$ is faithful by looking at its restriction to  $D_r$,  we must rely on
 the topological freeness of the action of~$P^*$. Recall that 
the action of the discrete group $P^*$ on the compact space  $\Omega_P$ is said to be {\em topologically free} if for every element $e\neq x\in P^*$, 
the set of fixed points $\{\chi\in \Omega_P\mid x\cdot \chi = \chi\}$ has empty interior. Topological freeness for partial actions is defined similarly in~\cite{ELQ}.

In  order to decide whether the action of $P^*$ on $\Omega_P$ is topologically free it is helpful to review first  the description of the spectrum 
$\Omega_P$ of the diagonal~$D_r$ given in \cite[Corollary~5.6.28]{CELY}.
View~$\J$ as a semilattice with multiplication given by intersections of constructible ideals. The space $\hat{\J}$ of characters on~$\J$ is described in~\cite[p. 184]{CELY}, following a general construction for the semilattice of idempotents in an inverse semigroup \cite{MR2534230, MR2419901}: it consists of nonzero functions $ \J\to \{0,1\}$ that are compatible with the semilattice structure of~$\J$. That is, a function $\chi\colon\J\to\{0,1\}$ belongs to~$\hat{\J}$ if it is not identically~$0$ and $\chi(R\cap S)=\chi(R)\chi(S)$ for all $S,R\in \J$, 
where the multiplication in~$\{0,1\}$ is inherited from the multiplication in~$\CC$.  In the case that $\emptyset\in \J$
we require $\chi(\emptyset)=0$. 
The topology on~$\hat{\J}$ is the one  induced by pointwise convergence, so $\hat \J$ is  a compact Hausdorff space

By \cite[Corollary~5.6.28]{CELY}, the spectrum of~$D_r$ is the subspace $\Omega_P$ of~$\hat{\J}$ given by characters $\chi\colon \J\to\{0,1\}$ satisfying the following additional property: if $\chi(S)=1$  and $\{S_i\mid i=1,\ldots,n\}\subset\J$ are such that $S=\bigcup_{i=1}^nS_i$, then there is $i\in\{1,\ldots,n\}$ with $\chi(S_i)=1$.  
Equivalently, $\Omega_P$ is the set of points $\chi \in \hat\J$ such that the relation (T4) holds at 
$\chi$, namely, such that $\prod_i \big(\chi(S) - \chi(S_i) \big) =0$ whenever 
the constructible ideals $S$ and $ S_i$ with $ i=1,\ldots,n$ satisfy $S=\bigcup_{i=1}^nS_i$. Notice that $\Omega_P$ is closed in~$\hat{\J}$, hence compact
and that if we define $\omega_p(S) = \1_S(p)$ for $S\in \J$, then $\{\omega_p \mid p\in P\}$ is a dense subset of $\Omega_P$ \cite[Lemma~5.7.1]{CELY}.

We will also need the basis for the topology on $\Omega_P$ described in  \cite[equation~(4)]{Li:Semigroup_nuclearity}. This basis is
 also mentioned in \cite[p. 199]{CELY}, where we believe there is a typo in the negated inequality, which should read
 $e\not \leq e_i$ for the basic open set to contain the point $\chi_e$.
\begin{lem}\label{lem:basis-specP}  Let $P$ be a submonoid of a group. For each nonempty constructible right ideal $S\in\J$ and each finite (possibly empty) collection 
$\mc\subset \J$ of nonempty constructible right  ideals such that $S\not\subset \bigcup_{R\in\mc}R$, let 
 \[
V(S;\mc)\coloneqq\{\chi \in\Omega_P\mid \chi(S)=1;  \chi(R)=0\text{ for }R\in \mc\}.
 \]
 Then the collection $\{ V(S;\mc)\}$ indexed by the pairs ${(S,\mc)}$ is a basis for the topology of~$\Omega_P$
 consisting  of nonempty open sets.  
 When $P$ is not left reversible, we may assume $\mc$ to be nonempty.
\begin{proof} View $\Omega_P$ as a closed subspace of $\{0,1\}^{\J}$ with the relative topology of pointwise convergence. A basis for this topology is given by the open sets $N(A,B) $ indexed by  disjoint pairs of  finite subsets $A$ and $B$ of $\J$ and defined by
\[
N(A,B) \coloneqq \{ \omega\in \Omega_P\mid \omega(S)  =1 \text{ for all } S\in A \text{ and }  \omega(R)  = 0 \text{ for all } R\in B\}.
\] If we let $S_A \coloneqq \bigcap_{S\in A} S$, with $S_A =P$ for $A =\emptyset$, then $\omega(S_A) = \prod_A \omega(S)$, so we may rewrite 
$$N(A,B)= \{\omega\in \Omega_P\mid \omega(S_A)  =1  \text{ and }  \omega(R)  = 0 \text{ for }R\in B\}.$$ 
This shows that $N(A,B)=V(S_A;B)$. Now when  $S_A \subset \bigcup_{R \in B} R$, we have $S_A = \bigcup_{R \in B} (R\cap S_A)$. Since the corresponding relation (T4) 
 holds at every $\omega \in \Omega_P$, that is, $\prod_{R \in B} \big(\omega(S_A) - \omega(R\cap S_A) \big) =0$, we see that $N(A,B)=\emptyset$ in this case. 
 On the other hand, when $S_A \not\subset \bigcup_{R \in B} R$ we may choose $p\in S_A \setminus\bigcup_{R \in B} R$, in which case 
  $\omega_p \in N(A,B)$ and thus $N(A,B) =  V(S_A; B)$ is nonempty.
\end{proof}
\end{lem}

The partial action $\gamma=(\{A_g\}_{g\in G}, \{\gamma_g\}_{g\in G})$ of~$G$ on~$D_r$ from Proposition~\ref{prop:partial-const}
induces a partial action of $G$ by partial homeomorphisms of $\Omega_P$, which we now describe. Following~\cite{CELY}, for each~$g\in G$, we identify the spectrum of~$A_{g^{-1}}$ with the subspace of $\Omega_P$ given by $$U_{g^{-1}}=\{ \chi\in \Omega_P\mid \chi(K(\alpha))=1\text{ for some } \alpha\in\Hilm[W], \text{ with } \dot{\alpha}=g\},$$ see \cite[p. 189]{CELY} and \cite[Lemma~5.6.40]{CELY}. Then $\{U_g\}_{g\in G}$ is a family of open subspaces of~$\Omega_P$. By abuse of notation we also denote by~$\gamma_g$ the bijection from $\{K(\alpha)\mid \alpha\in\Hilm[W], \dot{\alpha}=g\}\subset \J$ onto $\{K(\beta)\mid \beta\in\Hilm[W], \dot{\beta}=g^{-1}\}\subset\J$ that sends $K(\alpha)$ to $K(\tilde{\alpha})$. Define a map $\hat{\gamma}_g\colon U_{g^{-1}}\to U_g$ by $\hat{\gamma}_g(\chi)=\chi\circ\gamma_{g^{-1}}$. Then $\hat{\gamma}=(\{U_g\}_{g\in G},\{\hat{\gamma}_g\}_{g\in G})$ is a partial action of~$G$ on~$\Omega_P$. This gives rise to the transpose partial action $\hat{\gamma}^*=(\{\mathrm{C}_0(U_{g})\}_{g\in G}, \{\hat{\gamma}_g^*\}_{g\in G})$ on $\mathrm{C}(\Omega_P)$, where $\hat{\gamma}_g^*$ is given by 
\begin{equation}\label{eq:spec-formula} f\in\mathrm{C}_0(U_{g^{-1}})\mapsto f\circ \hat{\gamma}_{g^{-1}}\in\mathrm{C}_0(U_{g}).
\end{equation} 
 It is then clear that the Gelfand transform $D_r \cong \mathrm{C}(\Omega_P)$ intertwines $\gamma$ and $\hat{\gamma}^*$.  

\begin{thm}\label{thm:idealsifftopfree}
 Let $P$ be a submonoid of a group~$G$. The following are equivalent:
 \begin{enumerate}
 \item the partial action of $G$ on $\Omega_P$ is topologically free;
 \item the action of $P^*$ on $\Omega_P$ is topologically free;
 \item if $u\in P^*\setminus\{e\}$ and $\mc$ is a finite collection of proper constructible right ideals, 
  then there exists $t \in P \setminus \bigcup_{R\in \mc} R$ such that $utP \neq tP$ (or, equivalently, such that $ut\notin tP^*$);
\item every  ideal of $\Toepu(P)$ that has trivial intersection with $D_u$ is contained in the kernel of the left regular representation.
 \end{enumerate}
\begin{proof} The equivalence  (1)$\iff$(4) is from \cite[Theorem 4.5]{AbaAba}, and the implication  (1)$\implies$(2) is obvious. In order to prove that (2)$\implies$(3),  assume that the  action of $P^*$ is topologically free and 
let $u\in P^*\setminus \{e\}$ be a nontrivial unit. Let $\mc$ be a finite collection of proper constructible ideals.
Since the action of $P^*$ is topologically free, the  nonempty basic open set $V(P; \mc)$ must contain a point  $\chi$ that is not fixed by $\hat\gamma_u$.
By density, we may assume that such a point is of the form $\omega_t $ for some $t\in P \setminus \bigcup_{R\in \mc} R$. Then $\omega_{ut}  = \hat\gamma_u(\omega_t)  \neq \omega_t$, which means that there is a constructible ideal $S$ that contains one of $tP$ and $utP$ but not both, and this translates into
$utP \neq tP$. That $utP \neq tP$ is equivalent to $ut\notin tP^*$ is \lemref{lem:minilemma}. 

We finish the proof by showing that (3)$\implies$(1). 
Suppose $g\in G\setminus\{e\}$. It suffices to show that every nonempty basic open subset $V(S;\mc)$ contained in $U_{g\inv}$
contains  a point that is not fixed by $\hat\gamma_g$.  Since $V(S;\mc) \subset U_{g\inv}$, we know that $S\setminus \bigcup_{R\in \mc} R \subset g\inv P$.
Choose $q \in S\setminus \bigcup_{R\in \mc} R$; then $\omega_q \in V(qP;\mc) \subset V(S;\mc) $.  If $\hat\gamma_g(\omega_q) \neq \omega_q$ we are done.
If $\hat\gamma_g(\omega_q) = \omega_q$, then $gqP = qP$ and \lemref{lem:minilemma} shows that $g = qu q\inv$ 
for some nontrivial unit $u\in P^* \setminus \{e\}$.
Since  the ideal $(q\inv R) \cap P$ is proper for each $R\in \mc$, we may apply condition (3)  to $u$ and the collection $\mc' = \{(q\inv R) \cap P\mid R\in \mc\}$
to get $t\in P\setminus \bigcup_{R\in \mc} (q\inv R) \cap P$ with  $utP \neq tP$. This means that 
$(q\inv g q) tP \neq tP$, or $g (qtP) \neq qtP$. Since $q \in g\inv P$, we have $\hat\gamma_g(\omega_{qt}) = \omega_{gqt} \neq  \omega_{qt}$.
So the point $\omega_{qt}$ is not fixed by $\hat\gamma_g$. Since $qt \in qP \setminus \bigcup q (q\inv R \cap P)$, it follows that $\omega_{qt}$
is in $V(qP,\mc)$ and hence in $V(S;\mc)$. This shows that the set of the fixed points of $\hat\gamma_g$ has empty interior and completes the proof.
\end{proof}
\end{thm}
\begin{cor}\label{cor:topfreeimpliesideals}
The equivalent  conditions of \thmref{thm:idealsifftopfree} imply that every nontrivial ideal of $\Toepr(P)$ has nontrivial intersection with $D_r$. 
If $\lambda^+ \colon \Toepu(P) \to \Toepr(P)$ is an isomorphism, then the converse  holds.
 \begin{proof}  Suppose $J$ is a nontrivial ideal in $\Toepr(P)$. By \thmref{thm:faithfulreps}, $J\cap (D_r\rtimes_{\gamma, r} P^*)$ is a nontrivial ideal in $D_r\rtimes_{\gamma, r}P^*$. Since the action of $P^*$ is topologically free,  \cite[Theorem 2]{AS} implies that  the ideal 
 $J \cap D_r = (J\cap (D_r\rtimes_{\gamma, r} P^*)) \cap D_r$ is nontrivial as wanted.   When the kernel of the left regular representation  $\lambda^+ \colon \Toepu(P) \to \Toepr(P)$ is trivial,  
 the converse simply becomes the implication (4)$\implies$(1) in \thmref{thm:idealsifftopfree}, which is from \cite[Theorem 4.5]{AbaAba}.
 \end{proof}
\end{cor}

 When we combine the results from \secref{sec:toeplitzalgebras} with topological freeness of the action of $P^*$,
we obtain the following uniqueness theorem for $\Cst$-algebras generated by jointly proper representations of $P$.
\begin{thm}\label{thm:uniqueiftopfreejointproper}
Let $P$ be a submonoid of a group~$G$. Suppose that any of the equivalent conditions of \thmref{thm:idealsifftopfree} hold 
and that the conditional expectation $E_u\colon\Toepu(P) \to D_u$ is faithful.
Let $\{W_p: p\in P\}$ be a family of elements of a $\Cst$\nb-algebra satisfying the presentation \textup{(T1)--(T4) }given in  \defref{def:toeplitz-semigroup}. 
Then there is a homomorphism $\pi_W\colon\Toepr(P)\to\Cst(W)$ such that $\pi_W(L_p)=W_p$ for all~$p\in P$, and $\pi_W$ is an isomorphism if and only if~$W$ is jointly proper.
\begin{proof} By the universal property of $\Toepu(P)$, there is a representation $\rho_W\colon\Toepu(P) \to \Cst(W)$ such that $\rho(t_p) = W_p$ for all $p\in P$. Since $E_u\colon\Toepu(P) \to D_u$ is faithful, the left regular representation $\lambda^+\colon\Toepu(P)\to\Toepr(P)$ is an isomorphism by \corref{cor:reg-kernel}. Thus  $\pi_W \coloneqq \rho_W \circ (\lambda^+)\inv$ is the required homomorphism.

It follows that $W$ is jointly proper if $\pi_W$ is faithful because the identity representation of $\Toepr(P)$ is obviously jointly proper. Assume now $W$ is jointly proper. Then $\pi_W$ is faithful on the diagonal subalgebra $D_r$ by \corref{cor:faithfulondiagonal}, and so it is faithful on $\Toepr(P)$ as well by \corref{cor:topfreeimpliesideals}.
\end{proof}
\end{thm}

\begin{rem}
\corref{cor:topfreeimpliesideals} generalizes \cite[Corollary 5.7.3]{CELY} 
to monoids with nontrivial units, provided their action is topologically free, for which we give a criterion in terms of the semigroup itself  in  \thmref{thm:idealsifftopfree}(3). 
We would also like to  note that there is a relation between \thmref{thm:uniqueiftopfreejointproper} and the uniqueness result  \cite[Theorem~4.3]{NNN1} and  postpone its discussion until \secref{sec:rightLCM}.
\end{rem}

\section{Strong covariance and a full boundary quotient}\label{sec:full-boundary}

In this section we  analyze a universal boundary quotient for $\Toepu(P)$. We will see that this differs from the boundary quotient $\partial\Toepr(P)$ in the sense of Li \cite[Definition~5.7.9]{CELY} only by the failure of an amenability condition. Our approach involves the covariance algebra $\CC\times_{\CC^P}P$  associated to the canonical product system~$\CC^P$ over $P$,  see \cite[Theorem~3.10]{SEHNEM2019558}.  
As we study representations of $\CC\times_{\CC^P}P$, a notion of  foundation sets emerges naturally,  generalizing the foundation sets from  \cite{Sims-Yeend:Cstar_product_systems} and \cite[Definition~3.4]{Crisp-Laca}. Thus, 
our findings extend \cite[Proposition 5.6]{Sims-Yeend:Cstar_product_systems} to submonoids of groups that are not  quasi-lattice ordered. The partial crossed product structure of $\Toepu(P)$ is not needed to define the covariance algebra, but it will be used to give sufficient (and in some cases necessary) conditions for simplicity of $\partial\Toepr(P)$. 

\subsection{Strongly covariant representations} Let $P$ be a submonoid of a group~$G$. Following~\cite{Sims-Yeend:Cstar_product_systems}, we let $\CC^P=(\CC_p)_{p\in P}$ denote the canonical product system over~$P$ with one-dimensional fibers. That is, $\CC_p=\CC$ carries the natural structure of a Hilbert space over~$\CC$ and the left action of~$\CC$ given by multiplication, and the multiplication map $\CC_p\otimes_{\CC}\CC_q\cong\CC_{pq}$ is also just multiplication in~$\CC$. Thus every isometric representation of~$P$ in a unital $\Cst$\nb-algebra~$B$ extends uniquely by linearity to a nondegenerate representation of~$\CC^P$ in~$B$, and every (nondegenerate) representation of~$\CC^P$ in~$B$ arises this way. Since each fiber of $\CC^P$ is a copy of~$\CC$, the Fock space of $\CC^P$ can be naturally identified with $\ell^2(P)$. Under this identification, the Fock representation of~$\CC^P$ corresponds to the left regular representation of~$P$ on~$\ell^2(P)$. We refer to \cite{Fowler:Product_systems} for further details on product systems. 

 We aim to find necessary and sufficient conditions for a representation $\rho \colon\Toep_u(P)\to B$ to factor through the covariance algebra~$\CC\times_{\CC^P}P$ of $\CC^P$ \cite{SEHNEM2019558}. First of all, we need to ensure that $\CC\times_{\CC^P}P$ is indeed a quotient of~$\Toep_u(P)$, via a \Star homomorphism that identifies the canonical generating isometries. 

Let us give a description of strongly covariant isometric representations of~$P$ that is equivalent to \cite[Definition~3.2]{SEHNEM2019558}
in the special case of $\CC^P$, and also recall how~$\CC\times_{\CC^P}P$ is constructed. For each finite set~$F\subseteq G$, we define a subset~$\Delta_F$ of~$P$ as follows: an element~$p\in P$ lies in~$\Delta_F$ if and only if for all $g\in F$, either $pP\cap gP=\emptyset$  or $pP\cap gP=pP$. This gives rise to a closed subspace $\ell^2(\Delta_F)\subset \ell^2(P)$. Observe that~$\ell^2(\Delta_F)$ corresponds to the Hilbert $\CC$\nb-module $\CC^P_{F}$ from  \cite[equation~(3.1)]{SEHNEM2019558}. If $F_1\subset F_2$ are finite subsets of~$G$, then $\Delta_{F_2}\subset \Delta_{F_1}$ and so $\ell^2(\Delta_{F_2})$ is a closed subspace of~$\ell^2(\Delta_{F_1})$. 

\begin{defn}\label{def:covariance-alg} We let $F$ range in the directed set formed by all finite subsets of~$G$ ordered by inclusion and define an ideal in~$D_r$ by 
$$D_r^{\rm{o}}\coloneqq\{b\in D_r\mid \lim_{\substack{F}}\|b\restriction_{\ell^2(\Delta_F)}\|=0\}. $$
We will say that an isometric representation $w\colon P\to B$ of~$P$ in a $\Cst$\nb-algebra~$B$ is \emph{strongly covariant} if the map that sends $\dot{L}_{\alpha}\in D_r$ to $\dot{w}_\alpha\in B$ extends to a well-defined \Star homomorphism $D_r\to B$ that factors through the quotient $D_r/D_r^{\rm{o}}$. The \emph{covariance algebra}~$\CC\times_{\CC^P}P$ is the universal $\Cst$\nb-algebra for strongly covariant isometric representations of~$P$.
\end{defn}

 Let $\jj\colon P\to \CC\times_{\CC^P}P$ be the universal representation of~$P$ in $\CC\times_{\CC^P}P$. We will simply write $\jj_p$ instead of $\jj_p(1)$ for the range of $1\in \CC_p$ under the inclusion $\jj_p\colon \CC_p\to \CC\times_{\CC^P}P$. So $p\mapsto \jj_p$ is an isometric representation of~$P$ in~~$\CC\times_{\CC^P}P$.

\begin{note} We will denote by~$D_u^{\rm{o}}$ the ideal of~$D_u\subset\Toep_u(P)$ isomorphic to~$D_r^{\rm{o}}$ via the left regular representation. 
\end{note}

We will gradually work towards a more concrete description of the ideal $D_u^{\rm{o}}$, and thus also of strongly covariant representations. We begin by showing that $\CC\times_{\CC^P}P$ is a quotient of~$\Toepu(P)$ and by giving a characterization of strongly covariant isometric representations in terms of a dense subalgebra of~$D_u$. 

\begin{lem}\label{lem: lc-strg} Let $\jj \colon P\to \CC\times_{\CC^P}P$ be the universal representation of~$P$ in $\CC\times_{\CC^P}P$. Then there is a surjective \Star homomorphism $q_u\colon \Toep_u(P)\to \CC\times_{\CC^P}P$ that sends a generating element $t_p\in \Toep_u(P)$ to~$\jj_p\in \CC\times_{\CC^P}P$. A representation $\rho\colon \Toep_u(P)\to B$ factors through $\CC\times_{\CC^P}P$ if and only if $$\rho\big(\sum_{\substack{\alpha}\in A}\lambda_\alpha\dot{t}_\alpha\big)=0$$ for every finite collection of neutral words $A\subset\Hilm[W]$ and scalars $\{\lambda_\alpha\mid\alpha\in A\}\subset \CC$ such that $\sum_{\substack{\alpha}}\lambda_\alpha\dot{t}_\alpha\in D_u^{\rm{o}}$.
\begin{proof} We will first prove that $\jj\colon P\to  \CC\times_{\CC^P}P$ satisfies relations (T1)--(T4) of Definition \ref{def:toeplitz-semigroup}. Clearly $\jj_e=1$ since $\CC\times_{\CC^P}P$ is generated by $\jj(\CC^P)$ as a $\Cst$\nb-algebra. Also, notice that $\dot{j}_\alpha=0$ for every neutral word $\alpha\in\Hilm[W]$ such that $K(\alpha)=\emptyset$, because then $\dot{L}_\alpha = 0$ in~$D_r$, and $\dot{j}_\alpha=\dot{j}_\beta$ for every pair of neutral words  $\alpha, \beta\in\Hilm[W]$ such that $K(\alpha)=K(\beta)$,  because then $\dot{L}_\alpha - \dot{L}_\beta=0$. To see that $\jj$ also satisfies (T4),  suppose $K(\alpha)  =  \bigcup_{\beta\in A} K(\beta)$, where $\dot{\alpha}=e$ and $A$ is a finite collection of neutral words. Then $\prod_{\beta\in A} (\dot{L}_\alpha - \dot{L}_{\beta})  =  0$ in~$D_r$ and thus $\prod_{\beta\in A} (\dot{\jj}_\alpha - \dot{\jj}_{\beta})  =  0$ in the covariance algebra $\CC\times_{\CC^P}P$ by \defref{def:covariance-alg}. This shows that~$\jj$ also satisfies (T4) and hence $t_p\mapsto \jj_p$ gives rise to a surjective \Star homomorphism $q_u\colon \Toep_u(P)\to \CC\times_{\CC^P}P$ as wished.

It remains to establish the last assertion in the lemma. Clearly $\rho\big(\sum_{\substack{\alpha}\in A}\lambda_\alpha\dot{t}_\alpha\big)=0$ if $\rho\colon \Toep_u(P)\to B$ factors through $\CC\times_{\CC^P}P$ and $\sum_{\substack{\alpha}}\lambda_\alpha\dot{t}_\alpha\in D_u^{\rm{o}}$, where $A\subset\Hilm[W]$ is a finite collection of neutral words and the $\lambda_\alpha$'s are scalars. To prove the converse implication, let $\Hilm[C]\subset\J$ be an $\cap$\nb-closed finite collection of constructible right ideals of~$P$. For each $S\in \mc$, choose a neutral word $\alpha_S\in\Hilm[W]$ such that $K(\alpha_S)=S$. Then $$\ma(\mc)=\clsp\{\dot{t}_\alpha\mid K(\alpha)\in \Hilm[C],\, \dot{\alpha}=e\}=\mathrm{span}\{\dot{t}_{\alpha_S}\mid S\in\mc\}$$ is a finite dimensional \Star subalgebra of~$D_u$. By assumption, $$\rho\big(\sum_{\substack{S}}\lambda_S\dot{t}_{\alpha_S}\big)=0$$ in~$B$ whenever $\sum_{\substack{S}}\lambda_S\dot{t}_{\alpha_S}\in D_u^{\rm{o}}.$ We deduce that~$\rho$ vanishes on $\ma(\mc)\cap D_u^{\rm{o}}$. Hence it will factor through $\CC\times_{\CC^P}P$ because $\CC\times_{\CC^P}P$ is precisely the quotient of~$\Toepu(P)$ by the ideal generated by $D_u^{\rm{o}}$ and $D_u^{\rm{o}}=\lim_{\mc}\ma(\mc)\cap D_u^{\rm{o}}$.
\end{proof}
\end{lem}

\subsection{Foundation sets and generating projections of~$D_r^{\rm{o}}$} In order to find a more concrete description of the kernel of the quotient map $q_u\colon \Toep_u(P)\to \CC\times_{\CC^P}P$, we give in the next lemma a class of projections in~$D_u^{\rm{o}}$.

\begin{lem}\label{lem:charac-projections} Let $\alpha\in \Hilm[W]$ be a neutral word. Suppose that $A\subset\Hilm[W]$ is a finite collection of neutral words with $\bigcup_{\substack{\beta\in A}}K(\beta)\subset K(\alpha)$. Then $\prod_{\substack{\beta\in A}}(\dot{t}_\alpha-\dot{t}_{\beta})$ belongs to $D_u^{\rm{o}}$ if and only if for all $p\in K(\alpha)$, one has $$pP\cap \Big(\bigcup_{\substack{\beta\in A}}K(\beta)\Big)\neq\emptyset.$$
\begin{proof} We begin by observing that for each finite set $F\subset G$ such that 
$\prod_{\substack{\beta\in A}}(\1_{K(\alpha)}-\1_{K(\beta)}) $ does not vanish on~$\ell^2(\Delta_F)$, 
one has
$$\Big\|\lambda^+\Big(\prod_{\substack{\beta\in A}}(\dot{t}_\alpha-\dot{t}_{\beta})\Big)\restriction_{\ell^2(\Delta_F)}\Big\|=\Big\|\prod_{\substack{\beta\in A}}(\1_{K(\alpha)}-\1_{K(\beta)})
\1_{\Delta_F}\Big\|=1$$ 
 because $\prod_{\substack{\beta\in A}}(\1_{K(\alpha)}-\1_{K(\beta)})$ and $\1_{\Delta_F}$ are commuting projections. So $\prod_{\substack{\beta\in A}}(\dot{t}_\alpha-\dot{t}_{\beta})$ lies in~$D_u^{\rm{o}}$ if and only if there exists a finite set $F\subset G$ such that $\prod_{\substack{\beta\in A}}(\1_{K(\alpha)}-\1_{K(\beta)})$ vanishes on~$\ell^2(\Delta_F).$ Since $\prod_{\substack{\beta\in A}}(\1_{K(\alpha)}-\1_{K(\beta)}) $ already vanishes on the closed subspace of $\ell^2(P)$ determined by the union $\bigcup_{\substack{\beta\in A}}K(\beta)$,  the proof of the lemma reduces to showing that there is a finite set $F\subset G$  such that $\Delta_F\cap \big(K(\alpha)\setminus \bigcup_{\substack{\beta\in A}}K(\beta)\big)=\emptyset$ if and only if $pP\cap \big(\bigcup_{\substack{\beta\in A}}K(\beta)\big)\neq\emptyset$ for all $p\in K(\alpha)$.

Assume that $\Delta_F\cap \big(K(\alpha)\setminus \bigcup_{\substack{\beta\in A}}K(\beta)\big)=\emptyset$ for some finite set $F\subset G$. We will prove that $pP\cap \big(\bigcup_{\substack{\beta\in A}}K(\beta)\big)\neq\emptyset$ for all $p\in K(\alpha)$. In case $K(\alpha)=\bigcup_{\substack{\beta\in A}}K(\beta)$ we are done. Suppose that $K(\alpha)\setminus \bigcup_{\substack{\beta\in A}}K(\beta)\neq\emptyset$ and take $p\in K(\alpha)\setminus \bigcup_{\substack{\beta\in A}}K(\beta).$ It suffices to show that $pP\cap\Delta_F\neq\emptyset$ because $\Delta_F\cap \big(K(\alpha)\setminus \bigcup_{\substack{\beta\in A}}K(\beta)\big)=\emptyset$. Since $p\not\in\Delta_F$, we can find $g_0\in F$ with $pP\cap g_0P\neq\emptyset$ but $p\not\in g_0P$. Let $r_1\in pP\cap g_0P$. If $r_1\in\Delta_F$, we are done. Otherwise, the set $F_1\coloneqq F\setminus\{g_0\}$ must be nonempty and we have $r_1\not\in \Delta_{F_1}$ because $r_1\not\in \Delta_{F}$ and $r_1\in g_0P$. Then we can find $g_1\in F_1$ such that $r_1P\cap g_1P\neq\emptyset$ but $r_1\not\in g_1P$. Take $r_2\in r_1P\cap g_1P$. Again, if $r_2\in\Delta_{F}$, we are done. Otherwise, the set $F_2\coloneqq F\setminus\{g_0,g_1\}$ is nonempty. As above, we deduce that $r_2\not\in\Delta_{F_2}$ since $r_2\not\in\Delta_F$ and $r_2\in g_0P\cap g_1P$. Thus we can find $g_2\in F_2$ such that $r_2P\cap g_2P\neq\emptyset$ but $r_2\not\in g_2P$. Take $r_3\in r_2P\cap g_2P$. We are done in case $r_3\in\Delta_F$. Otherwise, notice that $F_2\subsetneq F_1\subsetneq F$. Since~$F$ is a finite set, this process must stop after finitely many steps, and so we can find $r_l\in \Delta_F\cap pP\subset K(\alpha)$. We conclude that $r_l\in  \bigcup_{\substack{\beta\in A}}K(\beta)$.

For the converse, assume that $pP\cap \big(\bigcup_{\substack{\beta\in A}}K(\beta)\big)\neq\emptyset$ for all $p\in K(\alpha)$. We have to find a finite set $F\subset G$ with $\Delta_F\cap \big(K(\alpha)\setminus \bigcup_{\substack{\beta\in A}}K(\beta)\big)=\emptyset$. Of course this holds whenever $K(\alpha)=\bigcup_{\substack{\beta\in A}}K(\beta)$. Otherwise, set $$F\coloneqq \bigcup_{\substack{\beta\in A}}Q(\beta).$$ We claim that $$\Delta_F\cap \Big(K(\alpha)\setminus \bigcup_{\substack{\beta\in A}}K(\beta)\Big)=\emptyset.$$ To see this, let $p\in K(\alpha)\setminus \bigcup_{\substack{\beta\in A}}K(\beta)$. Let $\beta'\in A$ be such that $pP\cap K(\beta')\neq \emptyset$. Since $p\not\in K(\beta')$, there must be $g\in Q(\beta')$ such that $p\not \in gP$. For such a~$g$, we have $pP\cap gP\neq\emptyset$ because $\emptyset\neq pP\cap K(\beta')\subset pP\cap gP$. Hence $p\not\in\Delta_F$. This completes the proof of the lemma.
\end{proof}
\end{lem}

 \lemref{lem:charac-projections} motivates the following generalization of the concept of a foundation set, originally defined in the context of quasi-lattice orders \cite{Sims-Yeend:Cstar_product_systems}, see also \cite[Definition~3.4]{Crisp-Laca}.

\begin{defn}\label{def:foundation-set} Let $P$ be a submonoid of a group~$G$ and let $S\in \J$ be a constructible right ideal of~$P$. We shall say that a finite collection $\mc\subset \J$ of constructible ideals  is a \emph{foundation set} for~$S$, or an \emph{$S$\nb-foundation set}, if $R\subset S$ for all $R\in\mc$ and for all~$p\in S$, one has $$pP\cap\Big(\bigcup_{\substack{R\in\mc}} R\Big)\neq\emptyset.$$ We will say that an $S$\nb-foundation set $\mc$ is \emph{proper} if $S\setminus\big(\bigcup_{\substack{R\in\mc}} R\big)\neq\emptyset$. In general, we will say that a finite collection of constructible ideals $\mc\in \J$ is a \emph{relative} foundation set for~$S$ if $\{S\cap R\mid R\in\mc\}$ is an~$S$\nb-foundation set.
\end{defn}

We can now describe a set of generating projections for~$D_u^{\rm{o}}$.

\begin{cor}\label{cor:generating-projections} The ideal $D_u^{\rm{o}}$ of~$D_u\subset\Toep_u(P)$ is the closed linear span of projections of the form $\prod_{\substack{\beta\in A}}(\dot{t}_{\alpha}-\dot{t}_{\beta})$, where $\alpha$ is a neutral word and
$A\subset\Hilm[W]$ is a finite collection of neutral words such that  $\{K(\beta)\mid \beta\in A\}$ is a proper foundation set for~$K(\alpha)$. Moreover, a map $w\colon P\to B$ into a $\Cst$\nb-algebra~$B$ is a strongly covariant isometric representation of~$P$ if and only if it satisfies the relations \textup{(T1)--(T4)} of \defref{def:toeplitz-semigroup} and the \emph{boundary relations}, that is,
 $$\prod_{\substack{\beta\in A}}(\dot{w}_{\alpha}-\dot{w}_{\beta})=0$$
for every neutral word $\alpha$ in $\W$ and   proper foundation set $\{K(\beta)\mid \beta\in A\}$ for~$K(\alpha)$, where $A$ is a finite collection of neutral words in $\W$.
\begin{proof}  Let $a\in D_u^{\rm{o}}$. Using that $D_u^{\rm{o}}=\lim_{\mc}(\ma(\mc)\cap D_u^{\rm{o}})$, where $\ma(\mc)$ is the (finite-dimensional) $\Cst$\nb-subalgebra of~$D_u$ spanned by the projections $\{\dot{t}_{\alpha}\mid \dot{\alpha}=e,\, K(\alpha)\in\mc\}$, we may assume that~$a$ is a finite linear combination~$\sum_{\substack{\alpha}\in F}\lambda_{\alpha}\dot{t}_\alpha$,  where $F\subset \Hilm[W]$ is a finite collection of neutral words and $\lambda_\alpha\in\CC$ for all~$\alpha\in F$. Thus by Lemma~\ref{lem:orthogonaldecomp}, $a$ can be decomposed as a finite linear combination of orthogonal projections $$a=\sum_{\emptyset\neq A\subset F}\lambda_AQ_A,$$ where $Q_A=\prod_{\alpha\in A}\dot{t}_\alpha \prod_{\beta\in F\setminus A} (t_e-\dot{t}_\beta)$ and $\lambda_A=\sum_{\substack{\alpha\in A}}\lambda_\alpha$. We rewrite $$Q_A= \prod_{\alpha\in A}\dot{t}_\alpha \prod_{\beta\in F\setminus A} (t_e-\dot{t}_\beta)= \prod_{\beta\in F\setminus A} \big(\dot{t}_{\prod_A\alpha}-\dot{t}_{\prod_A\alpha\beta}\big)$$ and use Lemma~\ref{lem:charac-projections} to conclude that $a\in D_u^{\rm{o}}$ if and only if $\{K(\prod_A\alpha\beta)\mid \beta\in F\setminus A\}$ is a foundation set for $K(\prod_A\alpha)$ whenever $\lambda_A\neq 0$. Now observe that if $K(\alpha)=\bigcup_{\substack{\beta\in A}}K(\beta)$ so that $\{K(\beta)\mid \beta\in A\}$ is obviously a foundation set for $K(\alpha)$, then $\prod_{\substack{\beta\in A}}(\dot{t}_{\alpha}-\dot{t}_{\beta})=0$ already holds in~$\Toep_u(P)$ by (T4). Hence the nonzero generators of~$D_u^{\rm{o}}$ arise from 
proper foundation sets. This proves the first assertion of the corollary. 
The second assertion follows because $\CC\times_{\CC^P}P$ is the quotient of~$\Toepu(P)$ by the ideal generated by $D_u^{\rm{o}}$. 
\end{proof}
\end{cor}

\begin{rem}
 Take an element $a=\sum_{\substack{\alpha}\in F}\lambda_{\alpha}\dot{t}_\alpha\in D_u$ and suppose that $a\not\in D_u^{\rm{o}}$. Using the same decomposition as in the first part of the proof of \corref{cor:generating-projections}, we see that the image of $a$ in the quotient $D_u/D_u^{\rm{o}}$ is the linear combination of the images of  mutually orthogonal basic projections that are not given by foundation sets. 
 \end{rem}

\begin{rem} In the case that $(G,P)$ is a quasi-lattice ordered group, the last statement of Corollary~\ref{cor:generating-projections} is proved in \cite[Proposition 5.6]{Sims-Yeend:Cstar_product_systems}. Indeed, in this case the nonempty constructible right ideals of~$P$ are the principal ideals. So let $p\in P$ and let $A\subset P$ be a finite set. Then $\{qP\mid q\in A\}$ is a foundation set for~$pP$ in the sense of Definition~\ref{def:foundation-set} if and only if~$\{p^{-1}qP\mid q\in A\}$ is a foundation set for~$P$ as defined in~\cite{Sims-Yeend:Cstar_product_systems}. A foundation set $\{qP\mid q\in A\}$ for~$pP$ is proper if and only if $p\not\in A$ or, equivalently, $e\not\in p^{-1}A\coloneqq \{p^{-1}q\mid q\in A\}$.
\end{rem}

We could deduce from \cite[Lemma~3.3]{SEHNEM2019558} that the ideal  of $\Toepu(P)$ generated by $D_u^{\rm{o}}$ is invariant under the canonical gauge coaction of~$G$ on~$\Toepu(P)$, and so an element $a\in D_u$ lies in the kernel of the quotient map $q_u\colon \Toepu(P)\to \CC\times_{\CC^P}P$ if and only if $a\in D_u^{\rm{o}}$ (see \cite[Lemma~3.4]{SEHNEM2019558}).  If we then apply  \lemref{lem:charac-projections}, we see that 
the kernel of the quotient map $q_u\colon\Toep_u(P)\to \CC\times_{\CC^P}P$ does not contain any projection of the form $\prod_{\substack{\beta\in A}}(\dot{t}_\alpha-\dot{t}_{\beta})$ when $\{K(\beta)\mid \beta\in A\}$ \emph{is not} a $K(\alpha)$\nb-foundation set.

We would like to give next a direct proof of  a stronger result that makes it clear why these projections cannot vanish  under any nontrivial representation of $\Toepu(P)$  when $\{K(\beta)\mid \beta\in A\}$ is not a foundation set for $K(\alpha)$. 
  
 \begin{prop}\label{prop:non-foundation-sets}
Let $\rho$ be a representation of $\Toepu(P)$ in a $\Cst$\nb-algebra~$B$. Let $\alpha\in\Hilm[W]$ be a neutral word and let $A\subset \Hilm[W]$ be a finite (possibly empty) collection of neutral words such that $K(\beta)\subset K(\alpha)$ for all $\beta\in A$. 
If $\{K(\beta)\mid \beta\in A\}$ is not a foundation set for $K(\alpha)$ and 
$$\rho\Big(\prod_{\beta\in A} (\dot{t}_{\alpha}-\dot{ t}_{\beta})\Big) =0,$$ 
then $\rho\equiv 0$. As a consequence, the restriction of $q_u$ to $D_u$ induces an embedding of $D_u/D_u^{\rm{o}}$  in~$\CC\times_{\CC^P}P$, and a \Star homomorphism $\hat{\rho}\colon\CC\times_{\CC^P}P\to B$ is faithful on $D_u/D_u^{\rm{o}}$ if and only if~$\hat{\rho}\neq 0$.
\begin{proof} The proof is based on an observation that goes back to \cite[Lemma~5.1]{purelinf}. 
 Regarding the first part of the proposition, it suffices to show that if $\{K(\beta)\mid \beta\in A\}$ is not a foundation set for $K(\alpha)$, then the projection $\rho\big(\prod_{\beta\in A} (\dot{t}_{\alpha}-\dot{ t}_{\beta})\big)$ dominates a
range projection of the form $t_pt_p^*$ for some~$p\in P$. Indeed, if $\{K(\beta)\mid \beta\in A\}$ is not a $K(\alpha)$\nb-foundation set, there must be $p\in K(\alpha)$ satisfying $$pP\cap\Big(\bigcup_{\beta\in A}K(\beta)\Big)=\emptyset.$$ It follows from the defining relation (T2) of~$\Toep_u(P)$ that $t_pt_p^*\dot{t}_{\beta}=0$ for all~$\beta\in A$. Also, from relations (T1) and (T3) of Definition~\ref{def:toeplitz-semigroup} we have $t_pt_p^*\dot{t}_\alpha=t_pt_p^*$. Hence $$t_pt_p^*\prod_{\beta\in A}(\dot{t}_{\alpha}-\dot{ t}_{\beta})=t_pt_p^*\dot{t}_\alpha=t_pt_p^*$$ as wished.  Thus $\rho\big(\prod_{\beta\in A}(\dot{t}_{\alpha}-\dot{ t}_{\beta})\big)=0$ forces $\rho(t_pt_p^*)=0$, and since $t_p$ is an isometry, this means that   $\rho$ has to be the zero representation.  

Suppose now that $\rho\colon\Toep_u(P)\to B$ is a nontrivial strongly covariant representation.
We will show that $D_u \cap \ker \rho = D_u^{\rm{o}}$.   By definition, $D_u^{\rm{o}} \subset D_u \cap \ker \rho$. For the reverse inclusion take $a\in D_u\cap\ker\rho$. As before, it suffices to consider elements of the form $a=\sum_{\alpha\in F}\lambda_\alpha\dot{t}_\alpha\in\mathrm{span}\{\dot{t}_\alpha\mid \alpha\in\Hilm[W],\dot{\alpha}=e\}$. Arguing as in the proof of \corref{cor:generating-projections}, we may write~$a$ as a finite linear combination of orthogonal projections $$a=\sum_{\emptyset\neq A\subset F}\lambda_AQ_A,$$ where $Q_A=\prod_{\alpha\in A}\dot{t}_\alpha \prod_{\beta\in F\setminus A} (t_e-\dot{t}_\beta).$ If we had $a\not\in D_u^{\rm{o}}$,  \corref{cor:generating-projections} would produce an $A\subset F$ such that $\lambda_A\neq 0$ and $\{K(\prod_{ A}\alpha\beta)\mid\beta\in F\setminus A\}$ is not a foundation set for $K(\prod_{A}\alpha)$.  In this case $\rho(a)=0$ would imply $\rho(Q_A)=0$ and thus $\rho\equiv 0$ by the first part.  To avoid the contradiction we must have $D_u\cap\ker\rho \subset D_u^{\rm{o}}$. 
 
 Taking now $\rho = q_u \colon \Toepu(P)\to\CC\times_{\CC^P}P$, which is a nontrivial strongly covariant representation by \cite[Theorem~3.10]{SEHNEM2019558}(C3),  by the preceding argument we conclude that $ D_u \cap \ker q_u =D_u^{\rm{o}}$.  This shows that $q_u$ induces an embedding of  $D_u/D_u^{\rm{o}}$  in~$\CC\times_{\CC^P}P$.
 
 Finally, if $\hat{\rho}\colon \CC\times_{\CC^P}P\to B$  is a nonzero representation, then the representation $\rho\coloneqq \hat\rho \circ q_u$ of $\Toepu(P)$ is strongly covariant and nontrivial, so $D_u \cap \ker \rho = D_u^{\rm{o}}$.
 Since the embedded copy of $D_u/D_u^{\rm{o}}$ in $\CC\times_{\CC^P}P$ is $q_u(D_u)$,  to finish the proof we only need to notice that if $\hat\rho(q_u(a)) = 0$, then $\rho(a) = 0$ and hence $a\in D_u^{\rm{o}}$.
\end{proof}
\end{prop}

\subsection{The covariance algebra as a full boundary quotient} We can now provide a few  characterizations of~$\CC\times_{\CC^P}P$. Among them, we show that $\CC\times_{\CC^P}G$ coincides with the full partial crossed product $\mathrm{C}(\partial \Omega_P)\rtimes G$, of the partial action of $G$ restricted to $\partial \Omega_P$, see \cite[Definition~5.7.8]{CELY} and \cite[Definition~5.7.9]{CELY}.  
In view of  \cite[Corollary~5.7.6]{CELY}, it is then natural to regard $\CC\times_{\CC^P}P$ as a \emph{full} boundary quotient for $\Toepu(P)$.
In order to establish the desired crossed product picture for $\CC\times_{\CC^P}P$, we need to show that the partial action $\gamma=(\{A_g\}_{g\in G},\{\gamma_g\}_{g\in G})$ of~$G$ on~$D_r$ from \thmref{thm:partial-picture} induces a partial action of~$G$ on the quotient $D_r/D_r^{\rm{o}}$.

Let $\gamma=(\{A_g\}_{g\in G},\{\gamma_g\}_{g\in G})$ be a partial action of~$G$ on a $\Cst$\nb-algebra~$A$. Recall from \cite{ELQ}*{Definition~2.7} that a $\Cst$\nb-subalgebra $D$ of~$A$ is \emph{invariant} under~$\gamma$ if $\gamma_g(D\cap A_{g^{-1}})\subset D$ for all~$g\in G$. If $I\idealin A$ is an invariant ideal, then $\gamma$ restricts to a partial action on~$I$ in a natural manner. The underlying collection of ideals is simply $\{I\cap A_g\}_{g\in G}$ and the isomorphism $I\cap A_{g^{-1}}\to I\cap A_{g}$ is simply the restriction of $\gamma_g$. The quotient $A/I$ also carries a partial action of~$G$ given by $\dot{\gamma}=(\{A_g/I\}_{g\in G},\{\dot{\gamma}\}_{g\in G})$, where $\dot{\gamma}_g\colon A_{g^{-1}}/I\to A_g/I$ is the isomorphism $a+I\mapsto \gamma_g(a)+I$. By \cite[Proposition 3.1]{ELQ}, there is a short exact sequence $$0\longrightarrow I\rtimes_{\gamma\restriction_I}G\longrightarrow A\rtimes_{\gamma}G\longrightarrow (A/I)\rtimes_{\dot{\gamma}}G\longrightarrow 0, $$ where the inclusion $ I\rtimes_{\gamma\restriction_I}G\hookrightarrow A\rtimes_{\gamma}G$ extends the embedding~$(I\cap A_g)\delta_g\hookrightarrow A_g\delta_g$ for~$g\in G$ while the second \Star homomorphism sends $a\delta_g\in A_g\delta_g$ to $(a+I)\delta_g\in (A_g/I)\delta_g$.

\begin{lem}\label{lem:invariance-Du} 
Let $P$ be a submonoid of a group $G$. Then the ideal $D_r^{\rm{o}}$ of $D_r$ is invariant under the partial action   $\gamma=(\{A_g\}_{g\in G},\{\gamma_g\}_{g\in G})$  from \proref{prop:partial-const}.
\begin{proof} Let $a\in D_r^{\rm{o}}\cap A_{g^{-1}}$. As usual we may assume that $$a=\sum_{\substack{\alpha\in F}}\lambda_\alpha\1_{K(\alpha)}=\sum_{\substack{\alpha\in F}}\lambda_\alpha\1_{K(\tilde{\alpha}\alpha)},$$ where $F\subset\Hilm[W]$ is a finite collection of words and $\lambda_\alpha\in\CC$ for all~$\alpha\in F$. What is not immediately clear is that we may choose such a linear combination so that~$\dot{\alpha}=g$. To see that we may, suppose first~$\dot{\alpha}=e$ for $\alpha \in F$ and consider the decomposition of the identity associated to~$F$ as in Lemma~\ref{lem:orthogonaldecomp}. Thus $a$ is a finite linear combination $a=\sum_{\substack{\emptyset\neq A\subset F}}\lambda_AQ_A$, where the $Q_A$'s are mutually orthogonal projections. 

Since $D_r^{\rm{o}}\cap A_{g^{-1}}$ is an ideal of~$D_r$,
it follows that $$Q_A=\prod_{\alpha\in A}\1_{K(\alpha)} \prod_{\alpha\in F\setminus A} (\1-\1_{K(\alpha)}) 
= \tfrac{1}{\lambda_\alpha} Q_A a \in  D_r^{\rm{o}}\cap A_{g^{-1}}$$ whenever $\lambda_A\neq 0$.
 Suppose $\emptyset\neq A\subset F$ satisfies $\lambda_A\neq 0$ and let $0<\varepsilon<\frac{1}{2}$. Let $b=\sum_{\beta\in E} \lambda_\beta\1_{K(\beta)}\in A_{g^{-1}}$ be such that $\|Q_A-b\|<\varepsilon$, where $E\subset\Hilm[W]$ is a finite collection of words with $\dot{\beta}=g$ for all $\beta\in E$. Using  \proref{pro:propertiesKalpha}(5) write $\prod_{\alpha\in A}\1_{K(\alpha)} =\1_{K(\alpha_A)}$ with  $\alpha_A\coloneqq\prod_A\alpha$,
 (the order of this concatenation is irrelevant because each $\alpha$ is neutral). Since we have chosen $\varepsilon<\frac{1}{2}$, the support of $Q_A$ is contained in $\bigcup_{\beta\in E} K(\beta)$, and we see that 
\begin{equation*}\begin{aligned}K\big(\alpha_A\big)&=K\big(\alpha_A\big)\bigcap\bigg(\bigg(\bigcup_{\alpha\in F\setminus A} K(\alpha)\bigg)\bigcup\bigg(\bigcup_{\beta\in E}K(\beta)\bigg)\bigg)\\
&=\bigg(\bigcup_{\alpha\in F\setminus A} K\big(\alpha\alpha_A\big)\bigg)\bigcup\bigg(\bigcup_{\beta\in E}K\big(\beta\alpha_A\big)\bigg),\qquad \text{(by (5) of \proref{pro:propertiesKalpha}).}
\end{aligned}
\end{equation*} 
Then 
\begin{equation}\label{eqn:1KalphaAexpansion}
\1_{K(\alpha_A)} = \sum_{\emptyset \neq B \subset ((F\setminus A)\cup E)\alpha_A}(-1)^{|B|+1}  \prod_{\beta'\in B}\1_{K(\beta')},
\end{equation}
where $((F\setminus A)\cup E)\alpha_A$ stands for the collection of words $\{\beta\alpha_A\mid \beta\in (F\setminus A)\cup E\}$. 
If $B\cap (F\setminus A)\alpha_A\neq\emptyset$, then one of the factors in the product  
$\prod_{\beta'\in B}\1_{K(\beta')} $ equals $\1_{K(\alpha_A)} \1_{K(\alpha)}$ because $\alpha_A$ is neutral, hence
$$\prod_{\beta'\in B}\1_{K(\beta')} \prod_{\alpha\in F\setminus A} (\1-\1_{K(\alpha)})=0.$$  
Thus only the factors of \eqref{eqn:1KalphaAexpansion} with $B\subset E\alpha_A$  contribute to $Q_A$, so that
\begin{equation*}\begin{aligned}Q_A= \1_{K(\alpha_A)} \prod_{\alpha\in F\setminus A} (\1-\1_{K(\alpha)}) =\sum_{\substack{\emptyset\neq B \subset  E\alpha_A}}(-1)^{|B|+1}  \prod_{\beta'\in B}\1_{K(\beta')} \prod_{\alpha\in F\setminus A} (\1-\1_{K(\alpha)}).\end{aligned}\end{equation*}

Considering now each term in the sum on its own, suppose $\emptyset\neq B \subset E\alpha_A$ and take $\beta'_B\in B$. Let $\beta_B\in E$ be such that $\beta'_B=\beta_B\alpha_A$. Define the concatenation
 \[
 \sigma_B \coloneqq \beta'_B\prod_{\beta'\in B\setminus\{\beta'_B\}}\tilde{\beta}'\beta'.
 \] 
 Then $\dot\sigma_B= \dot\beta_B' = \dot \beta_B\dot \alpha_A = \dot \beta_B=  g$ and
  $\prod_{\beta'\in B}\1_{K(\beta')} = \1_{K(\sigma_B)} $,  by properties (3) and (5) of 
 \proref{pro:propertiesKalpha}. Hence 
 $$ \prod_{\beta'\in B}\1_{K(\beta')} \prod_{\alpha\in F\setminus A} (\1-\1_{K(\alpha)})=\1_{K(\sigma_B)}\prod_{\alpha\in F\setminus A} (\1-\1_{K(\alpha)}) = 
 \sum_{A' \subset  F \setminus A}(-1)^{|A'|}  \1_{K(\sigma_B)}\prod_{\alpha\in A'}\1_{K(\alpha)} $$ 
 is a finite linear combination of projections of the form $\1_{K(\sigma)}$, where $\sigma\in\Hilm[W]$ satisfies $\dot{\sigma}=g$. So the same will be true for~$Q_A$ and for $a=\sum_{\substack{\emptyset\neq A\subset F}}\lambda_AQ_A$.

For the remainder of the proof,  assume $$a=\sum_{\substack{\alpha\in F}}\lambda_\alpha\1_{K(\alpha)},$$ where $F\subset\Hilm[W]$ is a finite collection of words with $\dot{\alpha}=g$ and $\lambda_\alpha\in\CC$ for all~$\alpha\in F$. We can see in the proof of Corollary~\ref{cor:generating-projections} using the identification $D_u\cong D_r$ that each projection $Q_A$ appearing in the decomposition of~$a$ will have the form $\prod_{i=1}^{n_A}(\1_{K(\alpha_A)}-\1_{K(\beta_{i,A})})$, where $\dot{\alpha}_A=\dot{ \beta}_{i,A}=g$ and $K(\beta_{i,A})\subset K(\alpha_A)$ for all $i=1,\ldots,n_A$, by another application of  (3) and (5) of \proref{pro:propertiesKalpha}. Since $a$ also lies in $D_r^{\rm{o}}$,  Lemma~\ref{lem:charac-projections} implies that $\{K(\beta_{i,A})\mid i=1,\ldots,n\}$ has to be a foundation set for $K(\alpha_A)$ whenever $\lambda_A\neq 0$. 

Thus in order to prove that $\gamma_g(a)\in D_r^{\rm{o}}$, all we have to show is that if $\lambda_A\neq 0$, then $$\gamma_g(Q_A)=\prod_{i=1}^{n_A}(\1_{K(\tilde{\alpha'}_A)}-\1_{K(\tilde{\beta}_{i,A})})\in D_r^{\rm{o}}.$$ But this happens if and only if $\{K(\tilde{\beta}_{i,A})\mid i=1,\ldots,n\}$ is a foundation set for $K(\tilde{\alpha}_A)$. So let $p\in K(\tilde{\alpha}_A)\setminus\bigcup_{\substack{i=1}}^{n_A}K(\tilde{\beta}_{i,A})$. Notice that $g^{-1}p \in K(\alpha_A)$. Using that $\{K(\beta_{i,A})\mid i=1,\ldots,n\}$ is a foundation set for $K(\alpha_A)$, we can find $r\in g^{-1}pP\cap\left(\bigcup_{\substack{i=1}}^{n_A}K(\beta_{i,A})\right)$. It follows that $$gr\in pP\cap g\left(\bigcup_{\substack{i=1}}^{n_A}K(\beta_{i,A})\right)=pP\cap \left(\bigcup_{\substack{i=1}}^{n_A}K(\tilde{\beta}_{i,A})\right).$$ So $\gamma_g(a)\in D_r^{\rm{o}}$ as wished.
\end{proof}
\end{lem}

Li introduced a notion of a boundary quotient $\partial\Toepr(P)$ of $\Toepr(P)$ \cite[Definition~5.7.9]{CELY}, cf. \cite[Definition~7.14]{Li:Semigroup_nuclearity}, based on a more general construction in the context of inverse semigroups due to Exel \cite{MR2419901, MR2534230}. Before we can relate the covariance algebra $\CC\times_{\CC^P}P$ to  $\partial\Toepr(P)$, we recall the description of the boundary $\partial\Omega_P$ of $\Omega_P$. 
 Let $\hat{\J}_{\max}$ be the subset of~$\hat{\J}$ formed by all characters $\chi\in \hat{\J}$ such that $\chi^{-1}(1)=\{S\in \J\mid \chi(S)=1\}$ is maximal among all characters. That is, if $\chi'\in\hat{\J}$, $\chi'\neq\chi$, then $\chi^{-1}(1)\not\subset\chi'^{-1}(1).$ By \cite[Lemma~5.7.7]{CELY}, the closure $\overline{\hat{\J}_{\max}}$ is contained in~$\Omega_P$, and it is invariant under the partial action $\hat{\gamma}=(\{U_g\}_{g\in G},\{\hat{\gamma}_g\}_{g\in G})$ of~$G$ on~$\Omega_P$ by \cite[Lemma~5.7.5]{CELY}. Set $\partial\Omega_P \coloneqq\overline{\hat{\J}_{\max}}$. Then $\mathrm{C}(\partial\Omega_P)$ is invariant under the partial action $\hat{\gamma}^*$ on~$\mathrm{C}(\Omega_P)$, see \eqref{eq:spec-formula}. By \cite[Corollary~5.7.6]{CELY}, the \emph{boundary quotient} $\partial\Toepr(P)$  of~$\Toepr(P)$ as defined in \cite[Definition~5.7.9]{CELY} is canonically isomorphic to the reduced partial crossed product $\mathrm{C}(\partial\Omega_P)\rtimes_r G $ of~$\mathrm{C}(\partial\Omega_P)$ by the partial action of~$G$ obtained by restricting $\hat{\gamma}^*$ to~$\mathrm{C}(\partial\Omega_P)$. We will now see that  under the canonical identification of $\Omega_P$ with the spectrum $\hat{D_r}$ of~$D_r$, we have $\Omega_P\setminus\partial\Omega_P\cong\hat{D_r^{\rm{o}}}$, and consequently $\partial\Toepr(P) \cong (D_r/D_r^{\rm{o}})\rtimes_{\dot{\gamma}, r} G$. Notice that  $\partial\Toepr(P)$ is indeed a quotient of $\Toepr(P)$ 
 by \cite[Proposition~21.3]{Exel:Partial_dynamical}.

\begin{lem}\label{lem:basis-boundary} Let $P$ be a submonoid of a group $G$.  Let $V(S;\mc)=\{\tau\in\Omega_P\mid \tau(S)=1;\,\tau(R)=0\text{ for }R\in\mc\}$ be a basic open set for the topology of~$\Omega_P$ as in \lemref{lem:basis-specP}. Then $V(S;\mc)\cap\partial\Omega_P\neq\emptyset$ if and only if $\mc$ is \emph{not} a relative $S$\nb-foundation set. In particular, the collection of subsets of~$\partial\Omega_P$ given by
\[
\{V(S;\mc)\cap\partial\Omega_P\mid \mc\text{ is not a relative $S$\nb-foundation set }\}
\]
is a basis of nonempty open sets for the topology of~$\partial\Omega_P$. Moreover,   $\partial\Omega_P=\{\tau\in\Omega_P\mid \tau\restriction_{D_r^{\rm{o}}}=0\}$ and $D_r/D_r^{\rm{o}}\cong \mathrm{C}(\partial \Omega_P)$.
\begin{proof}  In case $\emptyset\not\in \J$, that is, when $P$ is left reversible, then $\partial\Omega_P$ consists of a single point, given by the character $\tau\colon \J\to\{0,1\}$, $\tau(S)=1$ for all~$S\in\J$. So clearly $V(S;\mc)\cap \partial\Omega_P\neq\emptyset$ if and only if 
$\mc=\emptyset$. Since any nonempty collection $\mc $ of constructible ideals contained in $R\subset S$ produces a foundation set in this case, this proves the result for left reversible $P$.

Assume $\emptyset\in\J$. Let us argue by contradiction and suppose that $\mc$ is a relative foundation set for~$S$ and there is some $\tau\in \partial\Omega_P\cap V(S;\mc)$. We may assume $\tau\in\hat{\J}_{\max}$ because $V(S;\mc)$ is open. Let $R\in\mc$. Since $\tau(R)=0$, there is $\emptyset\neq S_R\in\J$ with $\tau(S_R)=1$ and $S_R\cap R=\emptyset$ by \cite[Lemma~5.7.4]{CELY}. Put $R'\coloneqq\bigcap_{R\in\mc}S_R\in\J$. Notice that $R'\cap \Big(\bigcup_{R\in \mc}R\Big)=\emptyset$ and $S\cap R'\neq\emptyset$ because $\tau(S\cap R')=\tau(S)\tau(R')=1$. Take $p\in S\cap R'$. Using that $\{S\cap R\mid R\in\mc\}$ is an $S$\nb-foundation set, we can find $r\in pP\cap \Big(\bigcup_{R\in\mc}R\Big)$. But $r\in R'$ because $R'$ is a right ideal. This gives a contradiction. Hence $V(S;\mc)\cap \partial\Omega_P\neq \emptyset$ only if $\mc$ is not a relative foundation set for~$S$.

To see that $V(S;\mc)\cap\partial\Omega_P\neq\emptyset$ if $\mc$ is not a relative foundation set for~$S$, notice that the image of the projection $$\1_{S}\prod_{R\in\mc}(\1-\1_{R})=\prod_{R\in\mc}(\1_S-\1_{S\cap R})$$ in~$\mathrm{C}(\partial\Omega_P)$ has to be nonzero by \proref{prop:non-foundation-sets}. This is so because $\partial\Toepr(P)$ is nonzero. Hence any character $\tau\in\partial\Omega_P$ with $$\tau\bigg(\prod_{R\in\mc}(\1_S-\1_{S\cap R})\bigg)\neq 0$$ will belong to $V(S;\mc)$ and so $V(S;\mc)\cap\partial\Omega_P$ is nonempty.

It remains to establish the identification $\partial\Omega_P=\{\tau\in\Omega_P\mid \tau\restriction_{D_r^{\rm{o}}}=0\}.$ We identify the spectrum~$\hat{D_r^{\rm{o}}}$ of~$D_r^{\rm{o}}$ with the open subspace~$U^{\rm{o}}$ of~$\Omega_P$ consisting of the characters $\tau\in\Omega_P$ satisfying the following property: there exist a nonempty constructible ideal $S$, and a finite collection of nonempty constructible ideals $\mc\subset\J$ such that $\mc$ is a proper foundation set for~$S$, with $\tau(S)=1$ and $\tau(R)=0$ for $R\in\mc$. From this identification and from the above we immediately see the inclusion $\partial\Omega_P\subset\{\tau\in\Omega_P\mid \tau\restriction_{D_r^{\rm{o}}}=0\}$. To prove the reverse inclusion, take $\tau \in \Omega_P$ with $\tau\equiv0$ on~$D_r^{\rm{o}}$. Let $V(S;\mc)\subset \Omega_P$ be a basic open set containing $\tau$. Then $\mc$ is not a relative foundation set for $S$ because $\tau$ vanishes on~$D_r^{\rm{o}}$. By the first part of the proof, we have $V(S;\mc)\cap\partial\Omega_P\neq\emptyset$. So $\tau\in\partial\Omega_P$ since $\partial\Omega_P$ is closed and $V(S;\mc)$ is an arbitrary basic open set around~$\tau$. This implies the identification $\Omega_P\setminus\partial\Omega_P\cong\hat{D_r^{\rm{o}}}$ and the isomorphism $D_r/D_r^{\rm{o}}\cong \mathrm{C}(\partial \Omega_P)$, and completes the proof of the lemma.
 \end{proof}
\end{lem}
 
\begin{thm}\label{thm: several-characteriz} Let $P$ be a submonoid of a group $G$.  The following $\Cst$\nb-algebras are isomorphic, via isomorphisms that identify the canonical generating elements:

\begin{enumerate}
\item[\rm{(1)}]  the covariance algebra $\CC\times_{\CC^P} P$ of the canonical product system $\CC^P$ over~$P$;

 \item[\rm{(2)}] the universal $\Cst$\nb-algebra with generating set $\{v_p\mid p\in P\}$ subject to the relations \textup{(T1)--(T3)} of \defref{def:toeplitz-semigroup} together with the  relations
\begin{equation}\label{eq:boundary-relations}
  \prod_{\substack{\beta\in A}}(\dot{v}_{\alpha}-\dot{v}_{\beta})=0
\end{equation}  for every neutral word $\alpha$ in $\W$ and  foundation set $\{K(\beta)\mid \beta\in A\}$ for~$K(\alpha)$, where $A$ is a finite collection of neutral words in $\W$;

\item[\rm{(3)}]  the universal $\Cst$\nb-algebra with generating set $\{v_p\mid p\in P\}$ subject to the relations   \textup{(T1)--(T4)} of \defref{def:toeplitz-semigroup} together with the boundary 
relations from \corref{cor:generating-projections};

\item[\rm{(4)}] the full partial crossed product $(D_r/D_r^{\rm{o}})\rtimes_{\dot{\gamma}} G$;

\item[\rm{(5)}] the full partial crossed product $\mathrm{C}(\partial \Omega_P)\rtimes G$ of~$\mathrm{C}(\partial \Omega_P)$ by $\hat{\gamma}^*\restriction_{\mathrm{C}(\partial \Omega_P)}$.
\end{enumerate}

\begin{proof} The  relation \eqref{eq:boundary-relations} corresponds to relation (T4) of Definition~\ref{def:toeplitz-semigroup} when $K(\alpha)=\bigcup_{\beta\in A}K(\beta)$. Hence the universal $\Cst$\nb-algebra from item (2) 
is canonically isomorphic to the one with presentation (T1)--(T4) of Definition~\ref{def:toeplitz-semigroup} and the  boundary relations. These in turn are isomorphic to $\CC\times_{\CC^P}P$ by Corollary~\ref{cor:generating-projections}.

In order to establish the isomorphism $\CC\times_{\CC^P}P\cong(D_r/D_r^{\rm{o}})\rtimes_{\dot{\gamma}} G$, observe that $\CC\times_{\CC^P}P$ is the quotient of $\Toep_u(P)$ by the ideal generated by~$D_u^{\rm{o}}\cong D_r^{\rm{o}}$. Such an ideal is simply $\overline{\bigoplus}_{\substack{g\in G}} (D_u^{\rm{o}}\cap A_g)\delta_g$ by $\gamma$\nb-invariance of~$D_r^{\rm{o}}$ obtained in Lemma~\ref{lem:invariance-Du}. Hence it is canonically isomorphic to $D_r^{\rm{o}}\rtimes_{\gamma\restriction_{D_r^{\rm{o}}}}G$ by the observation preceding the statement of Lemma~\ref{lem:invariance-Du} and so $$\CC\times_{\CC^P}P\cong \Toep_u(P)/\langle D_u^{\rm{o}}\rangle\cong (D_r\rtimes_\gamma G)/(D_r^{\rm{o}}\rtimes_{\gamma\restriction_{D_r^{\rm{o}}}}G)\cong (D_r/D_r^{\rm{o}})\rtimes_{\dot{\gamma}}G$$ as asserted.

Finally, to see that $\CC\times_{\CC^P}P$ is also canonically isomorphic to~$\mathrm{C}(\partial \Omega_P)\rtimes G $, we apply \lemref{lem:basis-boundary}.  The isomorphism $\Toep_u(P)\cong D_r\rtimes_{\gamma}G\cong \mathrm{C}(\Omega_P)\rtimes G$ identifies $D_r^{\rm{o}}\rtimes_{\gamma\restriction_{D_r^{\rm{o}}}}G$ with $\mathrm{C}(\Omega_P\setminus \partial\Omega_P)\rtimes G$ since 
$\Omega_P\setminus \partial\Omega_P\cong \hat{D_r^{\rm{o}}}$. Hence $$\CC\times_{\CC^P}P\cong (D_r/D_r^{\rm{o}})\rtimes_{\dot{\gamma}} G\cong \mathrm{C}(\partial \Omega_P)\rtimes G$$ canonically.
\end{proof}
\end{thm}

Let us recount the type of relations we have encountered thus far.  First we had the original relations, corresponding to (T3), saying that
 for neutral $\alpha$ the expression $\dot v_\alpha$ depends only on the  ideal $K(\alpha)$; these  give rise to 
 Li's $\Cst_s(P)$. We then introduced the additional relations (T4) saying that a product of defect projections must vanish if it does in the reduced diagonal; these give rise to our $\Toepu(P)$. When we add the  boundary relations, saying that products corresponding to proper  $S$-foundation sets must vanish, \thmref{thm: several-characteriz}(3), we obtain $\CC\times_{\CC^P}P$.  No more relations of this type can be added without causing a total collapse, by \proref{prop:non-foundation-sets}.
In view of \thmref{thm: several-characteriz}(5), we will say that $\CC\times_{\CC^P}P$  is the \emph{full boundary quotient} of $\Toepu(P)$.

\begin{note} 
Let  $\Lambda\colon \mathrm{C}(\partial \Omega_P)\rtimes G\to \mathrm{C}(\partial \Omega_P)\rtimes_r G$
be the canonical \Star homomorphism of the full crossed product of the partial action on the boundary onto the reduced one.
We will denote by $$\Lambda_\partial: \CC\times_{\CC^P}P\to\partial\Toepr(P)$$  the canonical \Star homomorphism obtained by combining the isomorphism $\CC\times_{\CC^P}P\cong \mathrm{C}(\partial \Omega_P)\rtimes G$ from \thmref{thm: several-characteriz} with $\Lambda$ and then with the isomorphism $\mathrm{C}(\partial \Omega_P)\rtimes_r G \cong \partial\Toepr(P)$ from \cite[Corollary~5.7.6]{CELY}.
\end{note}

\begin{cor}\label{cor:reduced-boundary} Let $P$ be a submonoid of a group $G$.  Suppose that~$G$ is exact. Then a \Star homomorphism $\rho\colon\Toepr(P)\to B$ factors through the boundary quotient $\partial\Toepr(P)$ if and only if $$\rho\Big(\prod_{R\in\mc}(\1_{S}-\1_{R})\Big)=0$$ whenever $\mc$ is a foundation set for~$S\in\J$.  
\begin{proof} Recall that $(D_r/D_r^{\rm{o}})\rtimes_{\dot{\gamma}, r} G $ is canonically isomorphic to  $\partial\Toepr(P) $, see \lemref{lem:basis-boundary} and the preceding comment.
View $\Toepr(P)$ as the reduced partial crossed product $D_r\rtimes_{\gamma,r}G$ via the isomorphism given in Theorem~\ref{thm:partial-picture}. Since $\Toepr(P)$ carries a faithful conditional expectation onto~$D_r$, the inclusion $D_r^{\rm{o}}\hookrightarrow D_r$ induces an embedding $D_r^{\rm{o}}\rtimes_{\gamma\restriction_{D_r^{\rm{o}}}, r}G\hookrightarrow\Toepr(P)$ by \cite[Proposition~21.3]{Exel:Partial_dynamical}. Now suppose that~$G$ is exact. By \cite[Theorem 21.18]{Exel:Partial_dynamical}, there is a short exact sequence $$0\longrightarrow D_r^{\rm{o}}\rtimes_{\gamma\restriction_{D_r^{\rm{o}}}, r}G\longrightarrow \Toepr(P)\longrightarrow (D_r/D_r^{\rm{o}})\rtimes_{\dot{\gamma}, r} G\longrightarrow 0.$$ Since $(D_r/D_r^{\rm{o}})\rtimes_{\dot{\gamma}, r} G\cong\mathrm{C}(\partial \Omega_P)\rtimes_rG\cong \partial\Toepr(P) $, the result follows from an application of \corref{cor:generating-projections} because $D_r^{\rm{o}}\rtimes_{\gamma\restriction_{D_r^{\rm{o}}}, r}G$ is precisely the ideal of~$\Toepr(P)\cong D_r\rtimes_{\gamma,r}G$ generated by $D_r^{\rm{o}}$.
\end{proof}
\end{cor}

As an application of \thmref{thm: several-characteriz}, we can characterize the situation in which the universal Toeplitz algebra coincides with the full boundary quotient.

\begin{cor}\label{cor:qu-isomor} Let $P$ be a submonoid of a group $G$.  The following are equivalent:
\begin{enumerate}
\item the quotient map $q_u\colon \Toepu(P)\to\CC\times_{\CC^P}P$ is an isomorphism;
\item no finite collection of proper constructible right ideals is a foundation set for~$P$;
\item for every nonempty constructible right ideal $S$ in $P$, no finite collection of constructible right ideals is a proper foundation set for~$S$.
\end{enumerate}

\begin{proof} That (1) is equivalent to (3) is a consequence of the presentation given in \thmref{thm: several-characteriz}(3) for $\CC\times_{\CC^P}P$, and that (3) implies (2) is obvious.

Next we prove that (2) implies (3) by contrapositive. Suppose there is a nonempty constructible ideal $S\in \J$ and a finite collection of ideals $\mc\subset\J$ such that $\mc$ is a proper foundation set for~$S$. Take 
$$p\in S\setminus\Big(\bigcup_{\substack{R\in\mc}}R\Big).$$ 
We claim that $\{p^{-1}R\cap P\mid R\in\mc\}$ is a proper foundation set for~$P$. Indeed, let $q\in P$. Then $pq\in S$ and so $pqP\cap \big(\bigcup_{\substack{R\in\mc}}R\big)\neq\emptyset$ since $\mc$ is a foundation set for $S$. 
Thus $$qP\cap \big(\bigcup_{\substack{R\in\mc}}p^{-1}R\cap P\big)\neq\emptyset,$$ and hence $\{p^{-1}R\cap P\mid R\in\mc\}$ is a proper foundation set for~$P$ because $e\in P\setminus \big(\bigcup_{\substack{R\in\mc}}p^{-1}R\cap P \big)$.
\end{proof}
\end{cor}

\subsection{Characterization of purely infinite simple boundary quotients.}
Let $P$ be a submonoid of~$G$ and let $G_0$ be the subgroup of~$G$ given by 
$$G_0=\{g\in G\mid gP\cap S\neq\emptyset,\, g^{-1}P\cap S\neq\emptyset, \text{ for all } \emptyset\neq S\in \J\}.$$ 
By \cite[Proposition 5.7.13]{CELY}, the partial action of~$G$ on $\partial\Omega_P$ is topologically free if and only if its restriction to~$G_0$ is so. We give next equivalent characterizations of topological freeness of this action in terms of constructible ideals of~$P$.

\begin{thm} \label{thm:equivalenttopfree} Let $P$ be a submonoid of a group~$G$. The following statements are equivalent:
\begin{enumerate}

\item[\rm(1)] The partial action of~$G$ on $\partial\Omega_P$ is topologically free;

\item[\rm(2)] The partial action of~$G_0$ on $\partial\Omega_P$ is topologically free;

\item[\rm(3)] For every $g\in G_0$, $g\neq e$, and for every $p\in P$, there is $q\in pP$ with $qP\cap gqP=\emptyset;$

\item[\rm(4)] For all $p, t\in P$ with $p\neq t$, there is $s\in P$ such that $psP\cap tsP=\emptyset$;

\item[\rm(5)] Every proper ideal of $\CC\times_{\CC^P}P$ is contained in the kernel of the canonical map $$\Lambda_\partial\colon\CC\times_{\CC^P}P\to\partial\Toepr(P).$$
\end{enumerate}
\begin{proof} The equivalence (1)$\Leftrightarrow$(2) is \cite[Proposition 5.7.13]{CELY} and (1)$\Leftrightarrow$(5) follows from~\cite[Theorem~4.5]{AbaAba}. So we prove (2)$\Leftrightarrow$(3)$\Leftrightarrow$(4). Let $g\in G_0$, $g\neq e$, and $p\in P$. Consider the open subset $$V(pP)\coloneqq V(pP;\emptyset)\cap\partial\Omega_P=\{\tau\in\partial\Omega_P\mid \tau(pP)=1\}.$$ We claim that $V(pP)\cap U_{g^{-1}}\neq\emptyset$. Indeed, considering that $pP$ is a nonempty constructible ideal and $g\in G_0$,  take $s\in  g^{-1}P \cap pP$ and let $t\in P$ be such that $s=g^{-1}t$. Put $\alpha\coloneqq (e,t,s,e)$. Then $\dot{\alpha}=g$ and $K(\alpha)=P \cap sP \cap st\inv P = sP$, so that $\1_{sP}\in A_{g^{-1}}$. We deduce that any character $\tau\in \partial\Omega_P$ satisfying $\tau(sP)=1$ lies in $V(pP)\cap U_{g^{-1}}$. Hence $V(pP)\cap U_{g^{-1}}$ is nonempty. 

Now using that the action of~$G_0$ on $\partial\Omega_P$ is topologically free and $g\neq e$, we can find $\tau\in V(pP)\cap U_{g^{-1}}$ 
 such that $\tau \neq  \hat{\gamma}_g(\tau)=\tau\circ\gamma_{g^{-1}}$. Thus for some 
 $\beta\in\Hilm[W]$ with $\dot{\beta}=g^{-1}$ we have $$\tau(K(\beta))\neq(\tau\circ\gamma_{g^{-1}})(K(\beta))=\tau(g^{-1}K(\beta))= \tau(K(\tilde{\beta})).$$   
 Since $\tau$ is a character and $\tau(pP)=1$, this implies that $\tau(pP\cap K(\beta))\neq \tau(pP\cap K(\tilde{\beta}))$. If $\tau(K(\beta))=1$, we see that $pP\cap K(\tilde{\beta})\cap K(\beta)$ is not a foundation set for $pP\cap K(\beta)$ since $\tau\in\partial\Omega_P$. So there must be $q\in pP\cap K(\beta)$ with $qP\cap K(\tilde{\beta})=\emptyset$. It follows that $qP\cap g^{-1}qP=\emptyset$ and thus $qP\cap gqP=\emptyset$ as wished, because $g^{-1}qP\subset g^{-1}K(\beta)=K(\tilde{\beta})$. In the case $\tau(K(\tilde{\beta}))=1$ and $\tau(K(\beta))=0$, 
we exchange the roles of $\beta$ and $\tilde{\beta}$ and reason as above to obtain an element $q\in pP\cap K(\tilde{\beta})$ with $qP\cap K(\beta)=\emptyset$. This implies $qP\cap gqP=\emptyset$, completing the proof of $(2)\Rightarrow(3)$.

Assume now that condition (3) holds. Let $e\neq g\in G_0$ and let $V(S;\mc)\cap\partial\Omega_P$ be a basic open set of the topology of~$\partial\Omega_P$ as in \lemref{lem:basis-boundary}. Because $\{ R\mid R\in\mc\}$ is not a relative $S$\nb-foundation set, there is $p\in S$ satisfying $$pP\cap \bigg(\bigcup_{R\in\mc}R\bigg)=\emptyset.$$ Take $s\in pP\cap g^{-1}P$ and let $t\in P$ be such that $s=g^{-1}t$. Then $\1_{sP}\in A_{g^{-1}}$ because $sP=K(\alpha)$, where $\alpha=(e,t,s,e)$. Using our hypothesis, we can find $q\in sP\subset pP$ such that $qP\cap gqP=\emptyset$. We have $\1_{qP}\in A_{g^{-1}}$ since $q\in sP$. Let $\tau\in\partial\Omega_P$ be a character with $\tau(qP)=1$. Then $\tau\in V(S;\mc)\cap U_{g^{-1}}$ because $q\in sP\subset pP\subset S$ . Also,  $qP\cap gqP=\emptyset$ gives $1=\hat{\gamma}_g(\tau)(gqP)=\tau(qP)\neq\tau(gqP)=0$. We conclude that for every $e\neq g\in G_0,$ the set $$\{\tau'\in U_{g^{-1}}\cap\partial\Omega_P\mid \hat{\gamma}_g(\tau')=\tau'\}$$ has empty interior. Thus the action of $G_0$ on~$\partial\Omega_P$ is topologically free. We have established (2)$\Leftrightarrow$(3).

In order to prove the implication (3)$\Rightarrow$(4), take $p, t\in P$ with $p\neq t$. Suppose first that there exists $q\in pP$ such that $qP\cap tP =\emptyset$. Setting $s\coloneqq p^{-1}q$, we obtain $$psP\cap tsP \subset qP\cap tP=\emptyset.$$ In case there is $q\in tP$ satisfying $qP\cap pP=\emptyset$, we also have $psP\cap tsP=\emptyset$ with $s\coloneqq t^{-1}q$. Otherwise, if those two previous cases do not occur, then $pP\cap tP$ is a foundation set for both $pP$ and $tP$. We claim that $p^{-1}t\in G_0$. Indeed, let $q\in P$. Using that $pP\cap tP$ is a $tP$\nb-foundation set, take $r, s\in P$ such that $ps=tqr$. Then $qr=t^{-1}ps\in t^{-1}pP$ and so $qP\cap t^{-1}pP\neq\emptyset$. Because $pP\cap tP$ is also a foundation set for $pP$, one can similarly show that $qP\cap p^{-1}tP\neq\emptyset$. Since $q\in P$ is arbitrary, we obtain $p^{-1}t\in G_0$, proving the claim. Applying (3) with $g=p^{-1}t$ and the identity~$e$ playing the role of~$p$, we can find $q\in P$ such that $p^{-1}tqP\cap qP=\emptyset$. Thus $pqP\cap tqP=\emptyset$ as wished. This gives (3)$\Rightarrow $(4).

Next we prove (4)$\Rightarrow$(3). Let $g\in G_0$, $g\neq e$, and $p\in P$. Using that $g\in G_0$, we can find $r\in pP\cap g^{-1}P$. Let $t\in P$ be such that $r=g^{-1}t$.  Thus $g=tr^{-1}$ and notice that $r\neq t$ because $g\neq e$. By (4), there is $s\in P$ such that $tsP\cap rsP=\emptyset$. Put $q\coloneqq rs$. Then $q\in pP$ and $gqP=tr^{-1}rsP=tsP$. So $gqP\cap qP=\emptyset$. This establishes the implication (4)$\Rightarrow $(3) and completes the proof of the proposition.
\end{proof}
\end{thm}

It follows from \cite[Corollary 5.7.17]{CELY} that $\partial\Toepr(P)$ is purely infinite simple provided the partial action of $G_0$ on~$\partial\Omega_P$ is topologically free and $P\neq\{e\}$. By recent work of Abadie--Abadie \cite{AbaAba}, there is a converse to \cite[Corollary 5.7.17]{CELY} whenever the full and reduced partial crossed products associated with the partial action of~$G$ on~$\partial\Omega_P$ are the same. Combining this with the criteria for topological freeness given in \thmref{thm:equivalenttopfree}, we then get a characterization of purely infinite simple boundary quotients in terms of properties of the semigroup. We specify this in the next corollary.

\begin{cor}\label{cor:charac-inf-simple} Let $P$ be a submonoid of a group~$G$ with $P\neq\{e\}$. Any of the conditions from \thmref{thm:equivalenttopfree} implies that $\partial\Toepr(P)$ is purely infinite simple. The converse implication also holds if we further have $\CC\times_{\CC^P}P\cong \partial\Toepr(P)$ via the canonical map $\Lambda_\partial$.
\begin{proof} If the partial action of~$G_0$ on~$\partial\Omega_P$ is topologically free, then $\partial\Toepr(P)$ is purely infinite simple by~\cite[Corollary~5.7.17]{CELY}.  This gives the first assertion in the statement. For the last assertion, suppose that $\Lambda_\partial$ is an isomorphism and that $\partial\Toepr(P)$ is purely infinite simple. It follows that the left regular representation $\Lambda\colon\mathrm{C}(\partial\Omega_P)\rtimes G\to\mathrm{C}(\partial\Omega_P)\rtimes_r G$ is an isomorphism and $\mathrm{C}(\partial\Omega_P)\rtimes G$ has no nontrivial proper ideal. Thus the partial action of~$G$ on~$\partial\Omega_P$ is topologically free by~\cite[Theorem~4.5]{AbaAba}. So condition (1) of \thmref{thm:equivalenttopfree} is satisfied. This completes the proof of the corollary.
\end{proof}
\end{cor}

\begin{rem} In the case that $P$ is a right LCM monoid that embeds in a group, the characterization of purely infinite simple boundary quotients in terms of constructible ideals from Corollary~\ref{cor:charac-inf-simple} follows from Theorem~4.12 and Theorem 4.15 of~\cite{STARLING}. Indeed, by \cite[Theorem~4.12]{STARLING} it suffices to verify condition (4) of~\thmref{thm:equivalenttopfree} for elements $p$ and $t$ in the \emph{core submonoid} $$P_0\coloneqq\{p\in P\mid pP\cap qP\neq\emptyset \text{ for all } q\in P\}\subset P$$ to deduce that the action of~$G_0$ on~$\partial\Omega_P$ is topologically free. Thus if $\CC\times_{\CC^P}P\cong\partial\Toepr(P)$, then $\partial\Toepr(P)$ is simple if and only if for all $p,t\in P_0$, $p\neq t$, there exists $s\in P$ with $psP\cap tsP=\emptyset$. This is so because if $p,t\in P$ are such that $e\neq p^{-1}t\in G_0$ and $r\in P$ satisfies $pP\cap tP=rP$, then $p^{-1}r, t^{-1}r\in P_0$. So $p^{-1}rsP\cap t^{-1}rsP=\emptyset $ yields $pt^{-1}rs\cap rsP=\emptyset$. Hence condition (4) of \thmref{thm:equivalenttopfree} restricted to pairs of elements in~$P_0$ implies \thmref{thm:equivalenttopfree}(3). In general, if $P$ is not a right LCM monoid we do not know whether \thmref{thm:equivalenttopfree}(4) restricted to elements in the core submonoid implies topological freeness of the partial action of~$G_0$ on~$\partial\Omega_P$.
\end{rem}

 \section{A numerical semigroup}\label{sec:numerical}
An additive submonoid $P$ of $\NN$ such that $\NN\setminus P$ is finite is called a \emph{numerical semigroup}. As pointed out by Li \cite[Section 5.6.5]{CELY} numerical semigroups do not satisfy independence. We wish to demonstrate the concrete  application of condition (T4) to the specific example $\Sigma\coloneqq \NN\setminus \{1\}$ studied in \cite{Rae-Vit}. We begin with an explicit description of the set $\J(\Sigma)$ of constructible ideals of $\Sigma$.

\begin{lem} \label{lem:numericalideallist} Every nonempty ideal of $\Sigma = \{0,2,3,4,5, \ldots\}$ is constructible, and 
\[
\J(\Sigma) =  \{p+\Sigma \mid p\in \Sigma\} \sqcup \{p+(2+\NN) \mid p\in \Sigma\} \sqcup \{3+\NN\}.
\]
\begin{proof}
We first  use \proref{pro:propertiesKalpha} to compute two key  constructible nonprincipal ideals: \begin{itemize}
\item[] $K(3,2,2,3) = \Sigma \cap (-3+2+\Sigma)  = \Sigma \cap (-1+\Sigma) = 2+\NN$;
\item[] $ K(2,3,3,2) =\Sigma \cap (-2+3+\Sigma)  = \Sigma \cap (1+\Sigma) = 3+\NN$.
\end{itemize}
If $I \subset \Sigma $ is a nonempty ideal and $m$ is its smallest element, then
 $m+\Sigma \subset I$. If $m+1 \notin I$, then $I = m+\Sigma$, which  is in the first set. If $m+1 \in I$ then $I = m+\NN =(m-2) + (2+\NN)$, which is in the second set unless $m = 3$, in which case $I = 3+\NN$. Since $p+(2+\NN) = K(0,p,3,2,2,3,p,0)$, the three sets on the right consist of constructible ideals.  Finally, $\emptyset \notin \J(\Sigma)$ because $\Sigma$ is abelian. 
\end{proof}
\end{lem}
It is clear that  independence can only fail at nonprincipal ideals. We choose the equality
\begin{equation}\label{eqn:failindepperf}
2+\NN = \{2, 3, 4, 5, \ldots\} = \{2, 4, 5, 6,\ldots \} \cup \{ 3, 5, 6,7, \ldots\} = (2+\Sigma) \cup (3+\Sigma),
\end{equation}
as the basic failure of independence, and  then show  that all other failures follow from this one. 

\begin{lem}\label{lem:num-failureduct}
Suppose that independence fails at the constructible ideal $S$ of $\Sigma$, in the sense that 
  $S= \bigcup _{R\in \mc} R$ for   a finite family $\mc$ of constructible ideals not containing $S$. Let $m = \min S$.
\begin{enumerate}
\item If $m\neq 3$, then   $S = p + (2+\NN)$ for $p\coloneqq m-2 \in \Sigma$, and there exist $R_1, R_2\in \mc$ such that  $R_1 = p +(2+\Sigma)$ and $R_2  \supset  p +(3+\Sigma)$. In this case 
\[
 p + (2+\NN) = S=  \bigcup _{R\in \mc} R = R_1 \cup R_2 = (p +(2+\Sigma)) \cup (p +(3+\Sigma)), 
\]
which is simply the $p$-translate of \eqref{eqn:failindepperf}.

\item  If $m=3$, then $S = 3+\NN $ and there exist $R_1, R_2\in \mc$ such that $R_1 =  (3+\Sigma) $ and $ R_2 \supset 4+\Sigma$.  In this case 
\[
3+\NN =S= \bigcup _{R\in \mc} R = R_1 \cup R_2 = (3+\Sigma) \cup (4+\Sigma),
\]
which is the $(-2+3)$\nb-translate of  \eqref{eqn:failindepperf}. 
\end{enumerate}
\begin{proof}
The ideal $ m+\NN$ is the only nonprincipal ideal containing $m$ as its smallest element, so it must be equal to $S$.  
Let $R_1$ be an ideal in $\mc$ that contains $m$, necessarily as its smallest element; since $R_1 \neq S$ we must have  $R_1 = m+\Sigma$. Since $m+1$ is in $S$ but not in $R_1$,  there must be another ideal $R_2 \in \mc$ that contains $m+1$, again, necessarily as its smallest element; hence $R_2$ is either  $m+1+\Sigma$ or $m+1+\NN$.
The rest consists of rewriting this in terms of $p=m-2$ when $m\neq 3$.
\end{proof}
\end{lem}

 \begin{prop}
Suppose $\{V_p\}_{p\in \Sigma}$  is a family of elements of a $\Cst$-algebra  satisfying the relations \textup{(T1)--(T3)} together with the  extra relation
 \begin{enumerate}
\item[] \textup{(T4)$_{2+\NN}$:} $\quad  V_3^*V_2 V_2^*V_3=   V_2V_2^* + V_3V_3^* - V_2V_2^* V_3V_3^*$. 
 \end{enumerate}
 Then $L_p \mapsto V_p$ extends to a \Star homomorphism $\pi_V\colon\Toepr(\Sigma) \to \Cst(V_p\mid p\in \Sigma)$.
 Moreover, $\pi_V$ is an isomorphism if and only it $V_3^*V_2 V_2^*V_3 \neq 1$.
  \begin{proof}
 Since $\Sigma$ is a submonoid of the abelian group $\ZZ$, we know that $ \Toepr(\Sigma)$ has the universal property of \defref{def:toeplitz-semigroup}, by \cite{CELY}*{Corollary 5.6.45} (see also \corref{cor:ifamenallequiv} above).   In order to conclude that $L_p \mapsto V_p$ extends,  it suffices 
 to show that $V$ satisfies  relation (T4) in that definition. 
 By \lemref{lem:equivalentto(T4)}(2) we only need to consider the cases in which independence fails. 
 So let $S= \bigcup _{R\in \mc} R$ and  $m =  \min S$ as in \lemref{lem:num-failureduct}. 
 
 If $m \neq 3$ we put $p = m-2$. If $I $  is a constructible ideal, let $E_I \coloneqq \dot V_\alpha$
 for any  neutral word $\alpha\in\W(\Sigma)$ such that $I = K(\alpha)$.  Then, using the earlier computation of $2+\NN$ and $3+\NN$ and  \proref{pro:propertiesKalpha}(6), we can write  $E_S = E_{p+(2+\NN)} = V_pE_{ 2+\NN}V_p^* = V_p(V_3^*V_2 V_2^*V_3)V_p^*$, and similarly  $E_{p+(2+\Sigma)} = V_pE_{(2+\Sigma)}V_p^* = V_p(V_2V_2^*) V_p^*$ and $E_{p+(3+\Sigma)} = V_p E_{(3+\Sigma)} V_p^* =V_p (V_3V_3^*)V_p^*$.  By \lemref{lem:num-failureduct}  we  have
 \begin{eqnarray*}
0 &\leq& \prod_{R\in \mc} (E_S -E_R) \leq (E_S -E_{R_1}) (E_S - E_{R_2}) \\&\leq&  V_p(E_{ 2+\NN} - E_{2+\Sigma}) V_p^* 
 V_p(E_{ 2+\NN} - E_{3+\Sigma}) V_p^*\\
 &=& V_p (V_3^*V_2 V_2^*V_3 -  V_2V_2^*)(V_3^*V_2 V_2^*V_3 -   V_3V_3^*) V_p^*,
  \end{eqnarray*}
which vanishes because of \textup{(T4)}$_{2+\NN}$. 
 If $m=3$ then 
 \begin{eqnarray*}
0 &\leq& \prod_{R\in \mc} (E_S -E_R) \leq (E_S -E_{R_1}) (E_S - E_{R_2}) \\
&\leq&  (E_{ 3+\NN} - E_{3+\Sigma}) (E_{ 3+\NN} - E_{4+\Sigma}) \\
&=& (V_2^*V_3V_3^*V_2 - V_3V_3^*) ( V_2^*V_3V_3^*V_2 - V_4V_4^*) \\
&=& \big(V_2^*V_3(V_3^*V_2 V_2^*V_3 -  V_2V_2^*) V_3^*V_2 \big) \big( V_2^*V_3 (V_3^*V_2 V_2^*V_3 -   V_3V_3^*)V_3^*V_2\big)\\
&=& V_2^*V_3(V_3^*V_2 V_2^*V_3 -  V_2V_2^*) (V_3^*V_2 V_2^*V_3 -   V_3V_3^*)V_3^*V_2,
\end{eqnarray*}
which, again, vanishes because of \textup{(T4)}$_{2+\NN}$.
This completes the proof that $V$ satisfies (T4), giving a canonical \Star homomorphism $\pi_V\colon\Toepr(\Sigma) \to \Cst(V)$.

If $\pi_V$ is injective, then $1 - V_3^*V_2 V_2^*V_3 \neq 0$ because $1 - L_3^*L_2 L_2^*L_3 $ is the projection onto the subspace generated by $\delta_0 \in \ell^2(\Sigma)$. 
For the converse, assume $1-V_3^*V_2 V_2^*V_3 \neq 0$ and let  $K(\alpha)$ be a proper ideal for each $\alpha$ in a finite collection $A\subset\W(\Sigma)$ of neutral words.  It is easy to see that $2+\NN$ is the largest proper ideal of $\Sigma$,  so $V_3^*V_2 V_2^*V_3 = E_{K(3,2,2,3)} \geq E_{K(\alpha)} = \dot V_\alpha$ for every  $\alpha \in A$. Hence $\prod_{\alpha \in A} (1- \dot V_\alpha) \geq 1-V_3^*V_2 V_2^*V_3 \neq 0$, proving that  $V$ is jointly proper. By \thmref{thm:uniqueiftopfreejointproper}, $\pi_V\colon\Toepr(\Sigma) \to \Cst(V)$ is faithful.
 \end{proof}
 \end{prop}

Since $\Sigma$ is generated by the elements $2$ and $3$, we would like to characterize representations satisfying (T1)--(T3) and (T4)$_{2+\NN}$ in terms of  generating isometries $W_2$ and $W_3$.  Notice that (T1) is obvious and (T2) never applies because $\Sigma$ is abelian, so the issue is to characterize pairs that generate a family satisfying (T3). 
Fortunately, in the present case we can rely on the spatial results from \cite{Rae-Vit} to characterize the pairs that yield  a presentation of $\Toepr(\Sigma)$ and thus obtain the following simplification of our uniqueness result.
\begin{cor} Suppose  $W_2$ and $W_3$ are two commuting  isometries in a $\Cst$\nb-algebra  having commuting range projections
and satisfying  $$\textup{(T3)}_\Sigma\colon\, W_2^3 = W_3^2\quad\text{ and }\quad\textup{(T4)}_{2+\NN}\colon\, W_3^*W_2 W_2^*W_3=   W_2W_2^* + W_3W_3^* - W_2W_2^* W_3W_3^*.$$
Then there exists a unique \Star homomorphism $\pi_W\colon \Toepr(\Sigma) \to \Cst(W_2, W_3)$ such that  $\pi_W(L_2) = W_2$ and $\pi_W(L_3) = W_3$, and $\pi_W$ is an isomorphism if and only if $W_3^*W_2 W_2^*W_3 \neq 1$. Therefore,  all the  $\Cst$-algebras generated by pairs $W_2$ and $W_3$ of isometries as above such that $W_3^*W_2 W_2^*W_3 \neq 1$ are canonically isomorphic.
\begin{proof}
By \cite[Proposition~1.5]{Rae-Vit} $W_2$ and $W_3$ generate a representation $V$ of $\Sigma$ by isometries with commuting range projections. By \cite[Theorem~2.1]{Rae-Vit} $V$ has a decomposition into three subrepresentations. Clearly the extra condition $W_3^*W_2 W_2^*W_3=   W_2W_2^* + W_3W_3^* - W_2W_2^* W_3W_3^*$ has to be satisfied separately in the three subrepresentations. But if we compute with the representation $S$ from \cite[Example~1.2]{Rae-Vit} we see that $S_3^*S_2 S_2^*S_3 = 1 \neq S_2S_2^* =   S_2S_2^* + S_3S_3^* - S_2S_2^* S_3S_3^*$. The only alternative is that the $S$\nb-component of $V$  is trivial, so $V$ is unitarily equivalent to a unitary representation and a multiple of the left regular representation of $ \Sigma$ on $\ell^2(\Sigma)$, the latter multiple being nontrivial if and only if $W_3^*W_2 W_2^*W_3 \neq 1$, in which case $\pi_W$ is injective.  
\end{proof}
\end{cor}

\section{$ax+b$-monoids of integral domains}\label{sec:integral}

\subsection{Basics on $ax+b$-monoids of integral domains} Let $R$ be an integral domain, which we will always assume to have a unit $1\neq 0$, and let $R^\times\coloneqq R\setminus\{0\}$ be its multiplicative semigroup. The $ax+b$-monoid associated to~$R$ is $R\rtimes R^\times$, where the action of~$R^\times$ by endomorphisms of~$R$ is given by multiplication. Hence the operation in~$R\rtimes R^\times$ is given by $$(b,a)(d,c)=(b+ad,ac),\qquad b,d\in R,\,  a,c\in R^\times.$$ The monoid $R\rtimes R^\times$ embeds in the group $G \coloneqq Q\rtimes Q^*$, where~$Q$ denotes the field of fractions of~$R$. 

\begin{note} Following \cite{Li:ax+b}, we will denote by $\I(R)$ the set of constructible ring-theoretic ideals of $R$, $$\I(R)=\big\{\bigcap_{\substack{g\in F}}gR\mid F\subset Q^\times,\, F\text{ is finite and }1\in F\big\}.$$
\end{note}
Thus, a nonzero ideal $I$ of $R$ lies in~$\I(R)$ if and only if it is the intersection of finitely many principal fractional ideals,
which is the same as saying that $I^\times $ is a constructible ideal in $R^\times$.

 By \cite[Lemma~2.11]{Li:ax+b}, the nonempty constructible right ideals in 
$ R\rtimes R^\times$ are indexed by pairs $(r,I)$ in which $I\in\I(R)$ and $r\in R$ represents a class in $R/I$. Specifically, when~$R$ is not a field,
\[
\J(R\rtimes R^\times) = \{ (r+I) \times I^\times \mid  r\in R   \text{ \ and \ }   I \in \I(R)\}\cup\{\emptyset\}.
\]
 \subsection{Topological freeness for $ax+b$-monoids of integral domains} 
The following proposition gives algebraic conditions on~$R$ that are equivalent to \thmref{thm:idealsifftopfree}(3) for the associated $ax+b$\nb-monoid $R\rtimes R^\times$, enabling the application of \thmref{thm:uniqueiftopfreejointproper}.

\begin{prop}\label{pro:manyideals} 
Let $R$ be an integral domain. Suppose that $R/I$ is finite for every  ideal~$I\in\I(R)$. Then $R\rtimes R^\times$ satisfies condition \textup{(3)} of \thmref{thm:idealsifftopfree} if and only if for every $x\in R^\times$ and every finite (possibly empty) collection $\mc \subset \I(R)$ of proper constructible ring-theoretic ideals there exists $a\in R \setminus \bigcup_{\substack{I\in\mc}} I$ such that $x\notin aR$.
\begin{proof} 
Condition (3)  of \thmref{thm:idealsifftopfree} says that for every unit $(x,u) \in  R\rtimes R^* \setminus \{(0,1)\}$ and
every finite collection $\mc_\rtimes = \{(r+ I) \times I^\times\mid r\in R, I\in\mc\}$ of proper constructible ideals in $R\rtimes R^\times$,
there exists $(s,a) \in R\rtimes R^\times \setminus \bigcup_{\substack{(r+I)\times I^\times\in\mc_{\rtimes}}} (r +I) \rtimes I^\times $ such that $(x,u)(s,a)P \neq (s,a)P$. 

We claim that the above condition is equivalent to the existence, for every nontrivial unit $(x,u)$ and every collection $\mc_\rtimes$ as above,  of $a\in R\setminus \bigcup_{\substack{I\in\mc}} I$ such that
 \begin{equation}\label{eqn:arithmprogr}
 (u-1)R +x \not\subset aR. 
 \end{equation} First notice that since the condition is to be satisfied by every family $\mc_\rtimes$, and since each $|R/I| <\infty$, we may assume without loss of generality that $\mc_\rtimes$ has been augmented so that all the classes modulo each of the $I$ are covered. Thus $(s,a) \in R\rtimes R^\times \setminus \bigcup_{\substack{I\in\mc}}\bigcup_ {r\in R/I}(r +I) \rtimes I^\times $ for some~$s$ if and only if $a\in R\setminus \bigcup_{I\in\mc} I$.
Recall from  \lemref{lem:minilemma} that $(x,u)(s,a)P \neq (s,a)P$ is equivalent to 
\[
(x+ us, ua) = (x,u)(s,a) \neq (s,a)(y,v)= (s+ay, av)  \quad \forall (y,v) \in P^*.
\]  
The multiplicative parts can always be matched by taking $v = u$ so this is really a condition on the additive parts, which says
that $a$ and $s$ satisfy $x+ (u - 1) s \neq ay$ for every $y\in R$. This completes the proof of the claim.

Suppose now that $R\rtimes R^\times$ satisfies condition (3) of \thmref{thm:idealsifftopfree} and, given $x\in R^\times$, consider the nontrivial unit $(x,1)$ of $R\rtimes R^\times$. By the equivalence proved above,  there exists $a\in R\setminus \bigcup_{\substack{I\in\mc}}I$ such that  \eqref{eqn:arithmprogr} holds, which in this case says that $\{x\} = (1-1) R + x \not \subset aR$, so $x\notin aR$.

To prove the converse, suppose that $\mc_\rtimes=\{(r+ I) \times I^\times\mid r\in R, I\in\mc\}$ is a finite collection of proper constructible ideals in $R\rtimes R^\times$, and let $(x,u)$ be a nontrivial unit. Assume first $x\neq 0$. Then there exists  $a\in R\setminus \bigcup_{\substack{I\in\mc}} I$ such that $x\notin aR$. Since $x = (u-1) 0 + x  \in (u-1) R +x$, we conclude that  \eqref{eqn:arithmprogr} holds. Assume now $x=0$. Then $u-1 \neq 0$, and there exists $a\in R\setminus \bigcup_{\substack{I\in\mc}} I$ such that $u-1\notin aR$.
Since $u-1 = (u-1) 1 + 0 \in (u-1) R +x$, we conclude that  \eqref{eqn:arithmprogr} holds in this case too.
\end{proof}
\end{prop}

\subsection{Boundaries for $ax+b$-monoids of integral domains} 
A characterization of pure infiniteness and simplicity of $\Cst$\nb-algebras constructed from rings that may have zero-divisors is given in \cite[Theorem~2]{Li:RingC*}. As a consequence, for $R$ an integral domain that is not a field, the boundary quotient $\partial\Toepr(R\rtimes R)$ is purely infinite simple \cite[Corollary~8]{Li:RingC*}. 
We can give a direct proof of this result, by verifying \thmref{thm:equivalenttopfree}(4) and then applying \corref{cor:charac-inf-simple}.

\begin{cor} Let $R$ be an integral domain that is not a field and let $R\times R^\times$ be the associated $ax+b$\nb-monoid. Then condition \textup{(4)} of \thmref{thm:equivalenttopfree} holds in~$R\rtimes R^\times$, and hence the boundary quotient $\partial\Toepr(R\rtimes R^\times)$ is purely infinite simple.
\begin{proof} In order to lighten the notation, we write $P\coloneqq R\rtimes R^\times$. Let $p,t\in P$ with $p\neq t$. We need to find $s\in P$ satisfying $psP\cap tsP=\emptyset$. Let $b,d\in R$, $a,c\in R^\times$ be such that $p=(b,a)$ and $t=(d,c)$.  It suffices to find such an element $s$ when  $b\neq d$ because the case $b=d$ entails $a\neq c$ and we may take a nonzero element $r\in R$
 and substitute $p$ and $t$ by 
  $$p(r,1)=(b+ar,a),\qquad t(r,1)=(d+cr,c),$$   which satisfy $b+ar\neq d+cr$.  Finding $s'$  for the pair of elements $p(r,1)$ and $t(r,1)$, gives $s =(r,1)s'$ for the pair of elements $p$ and $t$.
 Assuming thus $b-d\neq 0$, we separate the proof of existence of~$s$ into two cases.

\textbf{Case 1:} $b-d\not\in acR$. Set $s\coloneqq (0,ac)$. Then $$ps=(b,a^2c)\quad\text{ and }\quad ts=(d,c^2a).$$ We claim that $psP\cap tsP=\emptyset$. Indeed, looking for a contradiction, assume that there are $q_1=(f_1,e_1)$ and $q_2=(f_2,e_2)$ in $P$ such that $psq_1=tsq_2$. We would then have $$psq_1=(b+a^2cf_1,a^2ce_1)\quad\text{ and }\quad tsq_2=(d+c^2af_2,c^2ae_2).$$ Thus the equality $psq_1=tsq_2$ would imply $b+a^2cf_1=d+c^2af_2$ and hence $$b-d=ac(cf_2-af_1).$$ This contradicts our assumption that $b-d\not\in acR$.

\textbf{Case 2:} $b-d\in ac R^\times$. Let $\bar{x}\in R^\times$ be the unique element satisfying $b-d=ac\bar{x}$. Let $r\in R$ be a noninvertible element. We set $s\coloneqq (0,ac\bar{x}r)$. Thus $$ps=(b,a^2c\bar{x}r)\quad\text{ and }\quad ts=(d,c^2a\bar{x}r).$$ We claim that $psP\cap tsP=\emptyset$. Looking for a contradiction, assume that there are  $q_1=(f_1,e_1)$ and $q_2=(f_2,e_2)$ in $P$ such that $psq_1=tsq_2$. In this case we would have that $$b+a^2c\bar{x}rf_1=d+c^2a\bar{x}rf_2,$$ and hence $$b-d=ac\bar{x}(crf_2-arf_1)=(b-d)(crf_2-arf_1).$$  Since $R$ has no zero-divisors we would have  $crf_2-arf_1=1$ and thus $r(cf_2-af_1)=1$. We have arrived at a contradiction because~$r$ is not invertible in~$R^\times$. So we must have $psP\cap tsP=\emptyset$.

We have shown that \thmref{thm:equivalenttopfree}(4) holds, so \corref{cor:charac-inf-simple} gives the rest.
\end{proof}
\end{cor}

\section{$ax+b$-monoids of orders in number fields}\label{sec:orders}
   Recall that an algebraic number field~$K$ is a finite degree extension of the field~$\QQ$ of rational numbers. The  ring of algebraic integers $\OO_K$ of $K$ is the integral closure of $\ZZ$ in $K$; it is a free $\ZZ$-module of rank $d\coloneqq [K:\QQ]$. The rings of integers in algebraic number fields form an important class of integral domains, in fact of Dedekind domains. 
The Toeplitz $\Cst$\nb-algebras of the corresponding $ax+b$\nb-monoids and their equilibrium states were studied in \cite{CDLMathAnn2013}, and their ideal structure and $K$\nb-theory in \cite{EL2013} and \cite{Li:ax+b}, respectively. 

Here we will be interested in the more general class of  {\em orders}  in $K$. These are
subrings  of $\OO_K$ that are free of full rank~$d$ as $\ZZ$-modules. We will use the letter $\OO$ to denote an order in a number field; $\OO^*$ will denote the set of invertible elements of~$\OO$ and $\OO^\times$  the multiplicative monoid $ \OO\setminus \{0\}$.  The fraction field of $\OO$ equals that of~$\OO_K$, so that  $(\OO^\times)\inv \OO =K$.  We refer to \cite{Neu-99} for the basic definitions and results about rings of integers in number fields and to \cite{Ste-08} for orders; we also found K. Conrad's set of notes \cite{Con} very helpful.

\subsection{A uniqueness result for $\Toepr({\OO\rtimes\OO^\times})$} As an application of~\proref{pro:manyideals} and \thmref{thm:uniqueiftopfreejointproper}, we prove next a uniqueness result for the Toeplitz $\Cst$\nb-algebra $\Toep_\lambda({\OO\rtimes\OO^\times})$ of the $ax+b$\nb-monoid of an order~$\OO$ in a number field~$K$.

\begin{cor}\label{cor:uniqueness-orders}
Let $\OO$ be an order in an algebraic number field~$K$, and let $W \colon \OO \rtimes \OO^\times \to B$ be a map into a $\Cst$-algebra~$B$ satisfying relations \textup{(T1)--(T4)} from \defref{def:toeplitz-semigroup}. 
Then there is a  \Star homomorphism $ \pi_W\colon\Toepr(\OO\rtimes\OO^\times) \to \Cst(W)$ such that 
$\pi_W(L_p) =  W_p$ for all $p\in \OO\rtimes\OO^\times$, and $\pi_W$ is an isomorphism if and only if 
$W$ is jointly proper. 
\begin{proof} We will establish that  the assumptions of \thmref{thm:uniqueiftopfreejointproper} hold for $\OO\rtimes\OO^\times$. First, since the group $K\rtimes K^*$ is amenable, \corref{cor:ifamenallequiv} and \thmref{thm:modifiedresult} 
show that the left regular representation $\lambda^+$ is faithful. Then so is the conditional expectation $E_u\colon\Toepu(\OO\rtimes\OO^\times)\to D_u$ by \corref{cor:reg-kernel}.

Next we show that ${\OO\rtimes\OO^\times}$ satisfies condition (3) of  \thmref{thm:idealsifftopfree}. Since $\OO$ satisfies the assumptions of  \proref{pro:manyideals}, it suffices to show that for every~$x\in\OO^\times$ and  every finite collection $\mc\subset\I(\OO)$ of proper ideals there exists $a\in \OO\setminus \bigcup_{S\in\mc} S$ such that $x\notin a\OO$. Let $x\in \OO^\times$ and $\mc$ be such a collection.  Suppose first that $x$ is not invertible, that is, $x\OO$ is a proper ideal in $\OO$. The collection  $\mc_x\coloneqq \{ S'\mid S' \text{ is a proper ideal  in $\OO$ and } x\in S'\}$ is finite because $\OO/x\OO$ is finite. Consider now the finite collection
$\mc \cup \mc_x$. Since each proper ideal in $\OO$ can contain at most one rational prime, we may choose a prime number $p\in \NN$ such that 
\[
p\notin  \Big(\bigcup_{S \in \mc} S \Big) \cup \Big(\bigcup_{S' \in \mc_x} S' \Big).
\]
Then $p\in \OO\setminus \bigcup_{S\in\mc} S$, and $x\notin p\OO$ because, otherwise,
$p\OO$ would be a proper ideal between~$x\OO$ and~$\OO$,  hence in $\mc_x$ and this would  contradict
the choice of $p$.
Suppose now that $x$ is invertible, that is, $x\OO=\OO$ and just take a rational prime 
\[p\notin  \Big(\bigcup_{S \in \mc} S \Big).\]
Then $p\in \OO\setminus \bigcup_{S\in\mc} S$, and $x\notin p\OO$ because  $1/p$ is not an algebraic integer.

We have verified that the conditional expectation $E_u\colon\Toepu(\OO\rtimes\OO^\times)\to D_u$ is faithful, and that the equivalent conditions of \thmref{thm:idealsifftopfree} hold, so
 \thmref{thm:uniqueiftopfreejointproper}  completes the proof. \end{proof}
\end{cor}
\begin{rem} The key feature in the proof of \corref{cor:uniqueness-orders} is the availability of an infinite set of rational primes. So a slight modification of the proof  also yields a uniqueness theorem for representations of  the $ax+b$\nb-monoid of a congruence monoid. See \cite[Theorem~6.1]{MR4042889}  for a sharper result that only requires the jointly proper condition  at a restricted class of ideals. 
\end{rem}

\subsection{Independence fails for every nonmaximal order.}  The multiplicative monoid $\OO_K^\times$  and the $ax+b$-monoid $\OO_K\rtimes\OO_K^\times$ of the ring $\OO_K$ of algebraic integers in a number field~$K$ both satisfy independence  \cite[Lemma~5.6.36]{CELY}, see also \cite[Lemma~2.30]{Li:Semigroup_amenability}. We will show that, in contrast, when $\OO$ is a nonmaximal order  in~$K$, then $\OO^\times$ does not satisfy  independence, and  neither does 
$\OO\rtimes\OO^\times$, because \cite[Lemma~2.12]{Li:ax+b} shows that independence of $R^\times$ and of  $R\rtimes R^\times$ are equivalent for all integral domains. 
It follows that the behaviour exhibited by the example $\ZZ[\sqrt{-3}]$ given in \cite[Section~5.6.5]{CELY} is typical for nonmaximal orders. We will discuss this example later in detail following a recipe that could be used with every nonmaximal order, see Example~\ref{exa:failuresqrt-3}.
 
 Let $\OO$ be an order in an algebraic  number field~$K$. The  \emph{conductor} of $\OO$ is  given by $$\mfc\coloneqq(\OO :\OO_K)=\{x\in K\mid x\OO_K\subset\OO\},$$ 
 where we use $(R:I) \coloneqq \{x\in Q \mid xI \subset R\}$ to indicate  the ideal quotient of an integral domain $R$ with quotient field $Q$ over a fractional ideal $I$. Notice that $\mfc$ is an ideal in both~$\OO$ and~$\OO_K$, and is actually the largest ideal of~$\OO_K$ contained in~$\OO$.
 Since~$\OO$ has full rank in~$\OO_K$, there is a positive integer $m\in \ZZ$ such that $m\OO_K\subset \OO$. For example, one such integer is the index $[\OO_K\colon\OO]$. In particular, we may indeed regard $\OO_K$ as a fractional $\OO$\nb-ideal. Whenever $m\OO_K\subset \OO$, then $m\OO_K$ is an ideal of~$\OO_K$ contained in the conductor $\mfc$ of~$\OO$. 
 
  A key concept for us to analyze the failure of independence in $\OO$ is that of a divisorial fractional ideal. Recall that  a nonzero fractional ideal $I$ of an integral domain $R$ is said to be \emph{divisorial} whenever $(R:(R:I)) =I$;
equivalently, whenever $I$ is the intersection of an arbitrary collection of principal fractional ideals.  Notice that a nonzero ideal $I$ of a noetherian integral domain $R$ is divisorial if and only if it is a constructible ring-theoretic ideal, e.g. if and only if $I^\times$ is a constructible ideal in $R^\times$, because the arbitrary intersection can be replaced by a finite one.
Interested in the question of whether $m\OO_K$ is a constructible ring-theoretic ideal of~$\OO$, we prove next a number-theoretic result that is possibly well-known to the experts, but for which we have been unable to find a reference.   

\begin{lem}\label{lem:ok-divisorial} Let $K$ be an algebraic number field and $\OO_K$ the ring of integers in~$K$. Let $\OO$ be an order in~$K$. Then $\OO_K$ is divisorial as a fractional $\OO$\nb-ideal. 
\begin{proof} Of course we may assume that $\OO$ is a nonmaximal order in~$K$. We have to show that $$\OO_K=(\OO:(\OO:\OO_K))=(\OO:\mfc)=\{x\in K\mid x\mfc\subset\OO\}.$$ Let $g\in K$ be such that $g\mfc\subset\OO$. If $h\in \mfc$, then $h\OO_K\subset \mfc$ and so  $$gh\OO_K\subset g\mfc\subset\OO.$$ In particular, $gh\in\mfc$ and hence $g\mfc$ is contained in~$\mfc$, not just in~$\OO$. 

Now since $\OO_K$ is a Dedekind domain and $\mfc\subset\OO_K$ is a nonzero ideal, it follows that $\mfc$ is invertible (as a fractional $\OO_K$\nb-ideal) and so there exists a fractional $\OO_K$\nb-ideal $\mfc^{-1}$ with $\mfc^{-1}\mfc=\OO_K.$ Thus $$g\OO_K=\mfc^{-1}g\mfc\subset\mfc^{-1}\mfc=\OO_K.$$ This shows that $g\OO_K\subset \OO_K$ and therefore $g\in \OO_K$ as wished.
\end{proof}
\end{lem}

In order to prove that $\OO^\times$ and $\OO\rtimes\OO^\times$ do not satisfy independence whenever $\OO$ is a nonmaximal order in $K$, we introduce first some notation. Given $h\in \OO_K$, we denote by $n_h$ the positive generator of the ideal $h\OO_K \cap \ZZ$ of $\ZZ$; equivalently, $n_h$ is the least positive integer such that $n_h\OO_K\subset h\OO_K$. Then $h$ divides $n_h$ in $\OO_K$ and we let 
  \[
  h' \coloneqq \frac{n_h} {h}\in\OO_K.
  \]

\begin{prop}\label{prop:dependence} Let $\OO$ be a nonmaximal order in a number field~$K$. Let $m\in \ZZ$ be the least positive integer such that $m\OO_K\subset \OO$ and let $H\subset \OO_K$ be a complete set of representatives for the nontrivial cosets of $m\OO_K$ in~$\OO_K$. For each $h\in H$  define a fractional $\OO$-ideal $I_h$ by 
$$I_h\coloneqq  \begin{cases}
 \tfrac{1}{h'}m\OO\cap \OO_K & \text{ if }  h' \in \mfc,\\
\tfrac{1}{h'}\OO\cap \OO_K &\text{ if } h'\not\in\mfc.
\end{cases}$$ 
Then  $h + m\OO_K \subset I_h \subsetneq \OO_K$ and
$$\OO_K=\bigcup_{\substack{h\in H}}I_h.$$ As a consequence, the monoids $\OO^\times$ and ${\OO\rtimes\OO^\times}$ do not satisfy independence.

\begin{proof} We begin by proving that $\OO_K=\bigcup_{\substack{h\in H}}I_h$. Since $H$ is a complete set of representatives for the nontrivial cosets of~$m\OO_K$ in~$\OO_K$, we have $$\OO_K=m\OO_K\sqcup\Big(\bigsqcup_{\substack{h\in H}}\Big(h+m\OO_K\Big)\Big).$$
Next we show that each nontrivial class $h+m\OO_K$ can be replaced  by the corresponding $I_h$. We let $h\in H$ and compute
\begin{equation*}
\begin{aligned}
 h+m\OO_K&=\frac{n_hh}{n_h}+\frac{n_h}{n_h}m\OO_K\\&=\frac{n_hh}{n_h}+\frac{hh'}{n_h}m\OO_K\\&=\frac{h}{n_h}\Big(n_h+mh'\OO_K\Big)\\&=\frac{1}{h'}\Big(n_h+mh'\OO_K\Big)\subset \frac{1}{h'}(\ZZ +\OO) =\frac{1}{h'}\OO.
\end{aligned}
\end{equation*} 
Therefore $h+m\OO_K\subset  \frac{1}{h'}\OO\cap\OO_K$. When $h'$ belongs to the conductor~$\mfc$ of~$\OO$, we further have that $n_h = h h' \in \mfc \cap \ZZ = m\ZZ$ and $h'\OO_K\subset \mfc \subset \OO$. Thus the inclusion in the last line of the above computation can be strengthened to $\frac{1}{h'}\big(n_h+mh'\OO_K\big)\subset \frac{1}{h'}\big(m \ZZ + m\OO \big)= \frac{1}{h'}m\OO$. 
This yields  
$$ h+m\OO_K\subset  \frac{1}{h'}m\OO\cap\OO_K$$ 
in the case $h'\in\mfc$
and completes the proof  that $\OO_K= m\OO_K \cup \big(\bigcup_{\substack{h\in H}}I_h\big)$.
In order to see that the trivial class $m\OO_K$  is superfluous in this union, observe that $m\OO_K=h+m\OO_K-h\subset I_h$ for every $h\in H$. Hence $\OO_K= \bigcup_{\substack{h\in H}}I_h$. 

It remains to show that  $I_h$ is properly contained in~$\OO_K$ for every $h\in H$. Let $h\in H$ and suppose first that  $h'\not\in\mfc$. Looking for a contradiction, assume that  $$\OO_K= \tfrac{1}{h'}\OO\cap \OO_K.$$ Multiplying this equality by $h'\in\OO_K$, we deduce that $h'\OO_K\subset\OO$ and so $h'$ is in the conductor~$\mfc$. This contradicts our assumption that $h'\not\in\mfc$ and thus we must have $I_h\neq\OO_K$ whenever $h\in H$ satisfies $h'\not\in\mfc$.

Suppose next that $h'\in\mfc$.  Again looking for a contradiction,  assume that $$\OO_K= \tfrac{1}{h'}m\OO\cap \OO_K.$$ In this case we deduce that $\frac{h'}{m}\OO_K\subset \OO$, and hence $\frac{h'}{m}$ is in the conductor~$\mfc$ of~$\OO$. In particular $\frac{h'}{m}$ is in~$\OO_K$, so that
$\frac{n_h}{m} = h\frac{h'}{m} \in \QQ_+ \cap \OO_K $ is a positive integer such that $\frac{n_h}{m} \OO_K =  h\frac{h'}{m} \OO_K \subset  h\OO_K$. But since $$0\neq h\frac{h'}{m}=\frac{n_h}{m}< n_h$$ because $m\neq 1$, this contradicts our choice of~$n_h$ as the smallest positive integer satisfying $n_h\OO_K\subset h\OO_K$. Therefore $I_h\neq \OO_K$ also for $h\in H$ with $h'\in\mfc$.

Finally, by \lemref{lem:ok-divisorial} $\OO_K$ is divisorial as a fractional $\OO$-ideal, and thus it is a finite intersection of principal fractional $\OO$\nb-ideals because~$\OO$ is a noetherian domain. Hence $m\OO_K$ is a constructible ring-theoretic ideal of~$\OO$ and so is~$ mI_h$ for each $h\in H$; moreover, $mI_h\neq m\OO_K$. Thus  we have realized~$m\OO_K^\times$ as a union 
 $$m\OO_K^\times=\bigcup_{\substack{h\in H}} mI_h^{\times}$$ 
 of constructible ideals in $\OO^\times$ none of which is equal to $m\OO_K^\times$.
 We conclude that $\OO^\times$ and the corresponding $ax+b$\nb-semigroup $\OO\rtimes\OO^\times$ do not satisfy independence.
\end{proof}
\end{prop}
\begin{rem} In an algebraic number field $K$ with maximal order $\OO_K$, every nonzero fractional $\OO_K$-ideal is invertible, hence divisorial. By \cite[Corollary~7.1]{Li:ax+b}, among the orders in $K$ for which all   fractional ideals are divisorial, $\OO_K$ is the only one that satisfies independence. We have shown in \lemref{lem:ok-divisorial} that for an arbitrary order $\OO$ in $K$ the fractional $\OO$-ideal $\OO_K$ is divisorial, and this was enough to conclude that nonmaximal orders do not satisfy independence.  
We do not know whether all fractional $\OO$-ideals are divisorial, but many are. For example, since the map $\mfa \mapsto  \mfa\cap\OO$ is a bijection of the set of (integral)
 ideals in $\OO_K$ that are relatively prime to the conductor $\mfc$ onto the  ideals in $\OO$ relatively prime to $\mfc$,
 then the latter inherit from the former the property of being divisorial.
Since the fraction field of $\OO$  is the same as that of $\OO_K$, it also follows from \lemref{lem:ok-divisorial}
that the integral ideals of $\OO_K$ that are contained in $\OO$ are divisorial, in particular, the conductor is divisorial.  \end{rem}

\begin{example} \label{exa:failuresqrt-3} Let $K = \QQ(\sqrt{-3})$ and consider the order $\OO=\ZZ[\sqrt{-3}]$ in~$K$.  So $\OO$ is a nonmaximal order as the ring of algebraic integers~$\OO_K$ is given by
\[
\OO_K=\ZZ[\tfrac{1+\sqrt{-3}}{2}]=\ZZ + \tfrac{1+i\sqrt{3}}{2} \ZZ =  \big\{\tfrac{1}{2} (x  +y i\sqrt{3} ) \mid x, y \in \ZZ \text{ and } x-y = 0\pmod 2\big\}
\]
because $-3\equiv 1 \pmod 4$. We have $$2\OO_K= \big\{a  +b i\sqrt{3} \mid a, b \in \ZZ \text{ and } a-b= 0\pmod 2\big\}\subset\OO,$$ and notice that in this case $2\OO_K$ is the conductor~$\mfc$ of~$\OO$. Let $\omega\coloneqq\tfrac{1+\sqrt{-3}}{2}$. Then $$H\coloneqq\{1,\omega,\omega^2\}\subset\OO_K$$ is a complete set of representatives for the nontrivial cosets of $2\OO_K$ in $\OO_K$. Since each $h\in H$ is invertible in~$\OO_K$, it follows that $n_h=1$, $h=\tfrac{1}{h'}$ and so $h'\not\in\mfc$. Hence \proref{prop:dependence} gives $$\OO_K=\OO\cup\omega\OO\cup\omega^2\OO.$$ Multiplying this equality by~$2$ and removing the zero element from the ideals involved, we obtain \begin{equation}\label{eq:basic-failure-3} 2\OO_K^\times=2\OO^\times\cup2\omega\OO^\times\cup2\omega^2\OO^\times.\end{equation} Notice that $2\OO_K^\times$ corresponds to the constructible ideal $K(2,2\omega)=K(2,1+i\sqrt{3})=K(2\omega,2,2,2\omega)$ of $\OO^\times$ because $$2\OO_K^\times=2(\tfrac{1}{2}\OO^\times\cap\tfrac{1}{2\omega}\OO^\times)=\OO^\times\cap\tfrac{2}{2\omega}\OO^\times,$$ while the ideals appearing in the union above are simply principal ideals of~$\OO^\times$. We have arrived exactly at the instance of failure of independence for $\OO^\times=\ZZ[\sqrt{-3}]^\times$ provided in \cite[Section~5.6.5]{CELY}.
\end{example}

\begin{example} Let $K=\QQ(\sqrt{-1})$ and $\OO=\ZZ[2\sqrt{-1}]=\ZZ+2\ZZ i$. In this case $\OO_K$ is the ring of Gaussian integers $\ZZ[\sqrt{-1}]=\ZZ+\ZZ i$. The conductor of~$\OO$ is $\mfc=2\OO_K=2\ZZ+2\ZZ i$. One can show that $H=\{1,i,1+i\}$ is a complete set of representatives for the nontrivial cosets of~$2\OO_K$ in~$\OO_K$. We have $2\OO_K=(1+i)(1-i)\OO_K\subset (1+i)\OO_K$, so that $n_{1+i}=2$. Thus \proref{prop:dependence} gives $$\OO_K=\OO\cup i\OO\cup\big(\tfrac{1+i}{2}\OO\cap\OO_K\big).$$ We then obtain a concrete failure of independence in the monoid $\OO^\times$ given by $$2\OO_K^\times=\OO^\times\cap i\OO^\times=2\OO^\times\cup2i\OO^\times\cup\big((1+i)\OO^\times\cap\OO^\times\cap i\OO^\times\big).$$
\end{example}
 Orders in quadratic fields have the special feature that the conductor $\mfc$ coincides with the ideal $m\OO_K$, so it is interesting to analyze a more generic example, in which  $m\OO_K$ is properly contained in $\mfc$.
 \begin{example}
Consider the cubic field $K=\QQ(\sqrt[3]{19})=\QQ+\QQ\sqrt[3]{19}+\QQ\sqrt[3]{19^2}$, and let $\OO$ be the order $\ZZ[\sqrt[3]{19}]=\ZZ+\ZZ\sqrt[3]{19}+\ZZ\sqrt[3]{19^2}$ in~$K$, see \cite[Example~2.3]{Con}. The ring of algebraic integers $\OO_K$ is $$\OO_K=\ZZ+\ZZ\sqrt[3]{19}+\ZZ\frac{1+\sqrt[3]{19}+\sqrt[3]{19^2}}{3}.$$ Clearly $3$ is the smallest positive integer~$m$ with the property that $m\OO_K\subset\OO$. The conductor $\mfc$ of~$\OO$ is $$\mfc=\{a+b\sqrt[3]{19}+c\sqrt[3]{19^2}\mid a+b+c\equiv 0\pmod3\},$$ while $3\OO_K$ consists of the ideal in~$\OO$ given by all elements of the form $a+b\sqrt[3]{19}+c\sqrt[3]{19}$ with $a,b,c$ integers that are equal modulo~$3$. Hence  $3\OO_K\subsetneq\mfc$. 

We set $\omega\coloneqq \frac{1+\sqrt[3]{19}+\sqrt[3]{19^2}}{3}$. So $\{1,\sqrt[3]{19},\omega\}$ is a $\ZZ$\nb-basis for~$\OO_K$. One can show that the set $$H\coloneqq\{q_1+q_2\sqrt[3]{19}+q_3\omega\mid q_i\in\{0,1,2\}\text{ for $i=1,2,3$ and }q_1+q_2+q_3\neq 0\}$$ is a complete set of representatives for the nontrivial cosets of~$3\OO_K$ in~$\OO_K$. Let $h=q_1+q_2\sqrt[3]{19}+q_3 \omega\in H$ and let $n\in \ZZ$. Then $h$ divides~$n$ in~$\OO_K$ if and only if there are integers $r_1,r_2$ and $r_3$ such that $$h(r_1+r_2\sqrt[3]{19}+r_3\omega)=n.$$ Thus, given $h$, we seek the smallest positive integer $n_h$, for which the system
\begin{equation}\begin{array}{rcrcrcc}\label{eq:linear-system}
q_1x&+&(6q_3-q_2)y&+&(6q_2+4q_3)z&=&n_h\\
q_2x&+&(q_1-q_2)y&+&2q_3z&=&0\\
q_3x&+&(3q_2+q_3)y&+&(q_1+q_2+q_3)z&=&0
\end{array}
\end{equation}
has a solution $(r_1,r_2,r_3) \in \ZZ^3$, which gives $h' = r_1+r_2\sqrt[3]{19}+r_3\omega$.

All the possible fractional $\OO$\nb-ideals obtained from the elements in~$H$ as in \proref{prop:dependence} are displayed below in Table~\ref{tab:table1}. We have eliminated the obvious repetitions resulting from pairs
$h_1, h_2\in H$ with $h_2=2h_1$, which yield  $h_1'=h_2'$ and hence $I_{h_1}=I_{h_2}$. Notice that the table has three elements $h'$ that lie in the conductor, namely $-1+\dn$, $-4-2\dn+3\omega$ and $-37-2\dn+18\omega$; their respective 
rows are indexed by $(0,0,1)$, $(2,1,1)$, and $(1,2,2)$.
{\small
\begin{table}[h!]\def\arraystretch{1.6}
    \caption{}
    \label{tab:table1}
    \begin{tabular}{|c|c|c|c|c|} 
             \hline
         \textbf{$(q_1,q_2,q_3)$} &  \textbf{$n_h$} & \textbf{$h'$} & \textbf{$I_h$}\\
          \hline
 $(1,0,0)$ &   $1$ & $1$ & $\OO$ \\
      \hline
$(0,1,0)$ &   $19$ & $-1-\sqrt[3]{19}+3\omega$ & $\tfrac{\sqrt[3]{19}}{19}\OO\cap \OO_K$ \\
      \hline
$(1,1,0)$ &   $20$ & $-2\sqrt[3]{19}+3\omega$ & $\tfrac{1+\sqrt[3]{19}}{20}\OO\cap \OO_K$ \\
      \hline
$(2,1,0)$ &   $9$ & $1-\sqrt[3]{19}+\omega$ & $\tfrac{2+\sqrt[3]{19}}{9}\OO\cap \OO_K$ \\
      \hline
$(1,2,0)$ &   $51$ & $-1-2\sqrt[3]{19}+4\omega$ & $\tfrac{1+2\sqrt[3]{19}}{51}\OO\cap\OO_K$ \\
      \hline
 $(0,0,1)$ &   $6$ & $-1+\dn$ & $\tfrac{\omega}{2}\OO\cap\OO_K$ \\
      \hline
 $(1,0,1)$ &   $8$ & $2\dn-\omega$ & $\tfrac{1+\omega}{8}\OO\cap\OO_K$ \\
      \hline
 $(2,0,1)$ &   $6$ & $2+\dn-\omega$ & $\tfrac{2+\omega}{6}\OO\cap\OO_K$ \\
      \hline
$(0,1,1)$ &   $10$ & $-2+\omega$ & $\tfrac{\dn+\omega}{10}\OO\cap\OO_K$ \\
      \hline
$(1,1,1)$ &   $27$ & $-8-\dn+4\omega$ & $\tfrac{1+\dn+\omega}{27}\OO\cap\OO_K$ \\
      \hline
$(2,1,1)$ &   $12$ & $-4-2\dn+3\omega$ & $\tfrac{2+\dn+\omega}{4}\OO\cap\OO_K$ \\
      \hline
$(0,2,1)$ &   $60$ & $-5-\dn+4\omega$ & $\tfrac{2\dn+\omega}{60}\OO\cap\OO_K$ \\
      \hline
$(1,2,1)$ &   $66$ & $-6-2\dn+5\omega$ & $\tfrac{1+2\dn+\omega}{66}\OO\cap\OO_K$ \\
      \hline
$(2,2,1)$ &   $82$ & $-7-4\dn+7\omega$ & $\frac{2+2\dn+\omega}{82}\OO\cap\OO_K$ \\
          \hline
 $(1,0,2)$ &   $75$ & $-5+8\dn-2\omega$ & $\tfrac{1+2\omega}{75}\OO\cap\OO_K$ \\
      \hline
$(0,1,2)$ &   $153$ & $-23+5\dn+7\omega$ & $\tfrac{\dn+2\omega}{153}\OO\cap\OO_K$ \\
      \hline
$(1,1,2)$ &   $94$ & $-20+4\dn+5\omega$ & $\tfrac{1+\dn+2\omega}{94}\OO\cap\OO_K$ \\
      \hline
$(2,1,2)$ &   $15$ & $-5+\dn+\omega$ & $\tfrac{2+\dn+2\omega}{15}\OO\cap\OO_K$ \\
      \hline
$(1,2,2)$ &   $303$ & $-37-2\dn+18\omega$ & $\tfrac{1+2\dn+2\omega}{101}\OO\cap\OO_K$ \\
      \hline
    \end{tabular}
\end{table}
}

We already know from \lemref{lem:ok-divisorial} that $\OO_K$ is a divisorial fractional $\OO$\nb-ideal, and hence $3\OO_K$ is a  constructible ring-theoretic ideal of~$\OO$. But it is not difficult to describe them explicitly as 

$$ \OO_K = \tfrac{1}{1-\sqrt[3]{19}}\OO\cap \tfrac{1}{3}\OO \qquad \text{and} \qquad 3\OO_K=\tfrac{3}{1-\sqrt[3]{19}}\OO\cap \OO .$$ 
Indeed, take an arbitrary element $x=a+b\sqrt[3]{19}+c\sqrt[3]{19^2}\in\OO$ and compute $$(1-\sqrt[3]{19})x=(1-\sqrt[3]{19})(a+b\sqrt[3]{19}+c\sqrt[3]{19^2})=a-19c+(b-a)\sqrt[3]{19}+(c-b)\sqrt[3]{19^2}.$$ Observe that $a-19c\equiv a-c\pmod 3$, and so $\frac{1-\sqrt[3]{19}}{3}x\in \OO$ if and only if $x\in 3\OO_K$. Consequently, $$3\OO_K=\tfrac{3}{1-\sqrt[3]{19}}\OO\cap \OO$$ 
as claimed  and thus $\OO_K=\tfrac{1}{1-\sqrt[3]{19}}\OO\cap \tfrac{1}{3}\OO$. This also shows that
$3\OO_K^\times=K(3,1-\sqrt[3]{19})$ as a constructible ideal of $\OO^\times$.
By \proref{prop:dependence}, after multiplying by~$3$ and removing the zero element, each fractional $\OO$\nb-ideal in the last column of Table~\ref{tab:table1} yields a constructible ideal of~$\OO^\times$ that is properly contained in~$3\OO_K^\times$.  The union of all these ideals is~$3\OO_K^\times$. This illustrates a concrete failure of independence~in~$\OO^\times$.
\end{example}

\subsection{Reduction of relation (T4) for $\ZZ[\sqrt{-3}]\rtimes\ZZ[\sqrt{-3}]^\times$}  In analogy to what we saw in \secref{sec:numerical},  it sometimes suffices to verify relation (T4) from \defref{def:toeplitz-semigroup} on a particular instance of failure of independence also for $ax+b$-monoids of orders.
We would like to illustrate this by examining in more detail the example $\ZZ[\sqrt{-3}]$. 

 It follows from \cite[Lemma~6.3]{MR3474102} that if $\mathcal{F}$ is a finite collection of fractional $\OO$\nb-ideals such that $S\subsetneq\OO_K$ for each $S\in \mathcal{F}$ and $$\OO_K=\bigcup_{S\in \mathcal{F}} S,$$ then we must have $\{\OO,\omega\OO,\omega^2\OO\}\subset\mathcal{F}$. The constructible ideals of~$\OO^\times$ are  given in \cite{MR3474102}, see also \cite[Example~4.2]{Ste-08}; specifically, $$\J(\OO^\times)=\{x\OO^\times\mid x\in \OO^\times\}\cup\{2y\OO_K^\times\mid y\in\OO_K^\times\}.$$ Yet another description of $\J(\OO^\times)$ will be more convenient for our purposes: because $\OO_K=\bigcup_{\substack{j=0}}^2\omega^j\OO$, an arbitrary element $y\in\OO_K^\times$ has the form $y=\omega^jx$ for some $j\in\{0,1,2\}$ and $x\in\OO^\times$. Using that $\omega\in\OO_K^*$ and thus $\omega\OO_K=\OO_K$, we may rewrite $\J(\OO^\times)$ as $$\J(\OO^\times)=\{x\OO^\times\mid x\in\OO^\times\}\cup\{2x\OO_K^\times\mid x\in \OO^\times\}.$$ It will also be convenient to represent $\J(\OO^\times)$ in terms of neutral words in~$\mathcal{W}(\OO^\times)$ by $$\J(\OO^\times)=\{K(1,x,x,1)\mid x\in\OO^\times\}\cup\{K(1,x,2\omega,2,2,2\omega,x,1)\mid x\in \OO^\times\}.$$

We observe that \cite[Lemma~6.3]{MR3474102} also implies that if $x\in\OO^\times$ and $\mc\subset\J(\OO^\times)$ is a finite collection of constructible ideals such that $2x\OO_K^\times=\cup_{\substack{S\in\mc}}S$ and $S\subsetneq 2x\OO_K^\times$ for all~$S\in\mc$, then $\mc$ must contain $\{2x\OO^\times, 2x\omega\OO^\times,2x\omega^2\OO^\times\}$. Therefore, all instances of failure of independence for $\OO^\times$ can be obtained from \eqref{eq:basic-failure-3} via translations by elements in~$\OO^\times$. This leads to the following proposition at the level of representations of~$\OO^\times$. 

\begin{prop}\label{prop:t4-otimes} Let $K=\QQ(\sqrt{-3})$ and $\OO=\ZZ[\sqrt{-3}]$. Let $w\colon \OO^\times\to B$ be a map into a $\Cst$\nb-algebra~$B$ satisfying relations \textup{(T1)--(T3)} from \defref{def:toeplitz-semigroup}. Then $w$ satisfies relation \textup{(T4)} if and only if $w$ satisfies \textup{(T4)} at \eqref{eq:basic-failure-3}, that is, $$\prod_{\substack{j=0}}^2(w_{2\omega}^*w_2w_2^*w_{2\omega}-w_{2\omega^j}w_{2\omega^j}^*)=0.$$ 

\begin{proof} The `only if' direction is clear. In order to prove the converse, suppose $$\prod_{\substack{j=0}}^2(w_{2\omega}^*w_2w_2^*w_{2\omega}-w_{2\omega^j}w_{2\omega^j}^*)=0.$$ By \lemref{lem:equivalentto(T4)}, all we need to show is that~$w$ satisfies (T4) at special cases in which the independence condition fails. Let $\alpha\in\mathcal{W}(\OO^\times)$ be a neutral word and let $F\subset \mathcal{W}(\OO^\times)$ be a finite set of neutral words such that $K(\beta)\subsetneq K(\alpha)$ for all $\beta\in F$ and $$K(\alpha)=\bigcup_{\substack{\beta\in F}}K(\beta).$$ We may assume that $\alpha=(1,x,2\omega,2,2,2\omega,x,1)$ for some $x\in \OO^\times$ because~$w$ satisfies relations (T1)--(T3) from \defref{def:toeplitz-semigroup} and the independence condition can only fail at the nonprincipal ideals of~$\OO^\times$. By the discussion preceding the statement of the proposition, we deduce that $$\{2x\OO^\times, 2x\omega\OO^\times, 2x\omega^2\OO^\times\}\subset\{K(\beta)\mid \beta\in F\}.$$ Thus, for each $j\in \{0,1,2\}$, we can find a word $\beta_j\in F$ with $K(\beta_j)=2x\omega^j\OO^\times$. By relations (T1) and (T3), we have $$\dot{w}_{\beta_j}=w_{2x\omega^j}w_{2x\omega^j}^*=w_xw_{2\omega^j}w_{2\omega^j}^*w_x^*,\qquad (j=0,1,2).$$ Also, $\dot{w}_\alpha=w_xw_{2\omega}^*w_2w_2^*w_{2\omega}w_x^*.$ Then $$\prod_{\substack{j=0}}^2(\dot{w}_\alpha-\dot{w}_{\beta_j})=w_x\Big(\prod_{\substack{j=0}}^2(w_{2\omega}^*w_2w_2^*w_{2\omega}-w_{2\omega^j}w_{2\omega^j}^*)\Big)w_x^*=0.$$ Hence $$\prod_{\substack{\beta\in F}}(\dot{w}_\alpha-\dot{w}_\beta)= \prod_{\substack{j=0}}^2(\dot{w}_\alpha-\dot{w}_{\beta_j})\prod_{\substack{\beta\in F}}(\dot{w}_\alpha-\dot{w}_\beta)=0$$ as wished. So $w$ satisfies relation (T4) from \defref{def:toeplitz-semigroup}.
\end{proof}
\end{prop}

We now turn our attention to the $ax+b$\nb-monoid $\OO\rtimes\OO^\times$. It follows from \cite[Lemma~2.11]{Li:ax+b} that the family of constructible right ideals of $\OO\rtimes\OO^\times$ is given by \begin{equation*}
\begin{aligned}
\J(\OO\rtimes\OO^\times)&=\{(r+x\OO)\times x\OO^\times\mid x\in\OO^\times, r\in\OO\}\cup\{(r+2x\OO_K)\times 2x\OO_K^\times\mid x\in\OO^\times, r\in\OO\}\cup\{\emptyset\}\\
&=\{(r,x)(\OO\rtimes\OO^\times)\mid r\in\OO, x\in\OO^\times\}\cup\{(r,x)(2\OO_K\rtimes 2\OO_K^\times)\mid r\in\OO, x\in \OO^\times\}\cup\{\emptyset\}.
\end{aligned}
\end{equation*} The nonprincipal ideal $(r,x)(2\OO_K\rtimes 2\OO_K^\times)$ equals $K(\alpha)$, where $\alpha\in\mathcal{W}^4(\OO\rtimes\OO^\times)$ is the neutral (and symmetric) word $$\alpha=((0,1),(r,x),(0,2\omega),(0,2),(0,2),(0,2\omega),(r,x),(0,1)).$$ 

In $\OO\rtimes\OO^\times$, we find a failure of independence at the ideal $2\OO_K\rtimes 2\OO_K^\times$ because $$2\OO_K\times2\OO_K^\times=\Big(\bigcup^2_{\substack{j=0}}2\omega^j\OO\times 2\omega^j\OO^\times\Big)\cup\Big(\bigcup^2_{\substack{j=0}}(r_j+2\omega^j\OO)\times 2\omega^j\OO^\times\Big),$$ where $r_j\in 2\OO_K$ is any element in the unique nontrivial class of the quotient ring $2\OO_K/2\omega^j\OO$ for each $j=0,1,2$. Because $2\OO_K=2\omega^j\OO\sqcup (r_j+2\omega_j\OO)$ for each $j\in \{0,1,2\}$, another application of \cite[Lemma~6.3]{MR3474102} shows that if $\mathcal{F}\subset\J(\OO\rtimes\OO^\times)$ is a finite collection of constructible right ideals such that $2\OO_K\times2\OO_K^\times=\bigcup_{\substack{S\in\mathcal{F}}}S$, then 
$\mathcal{F}$ contains $$\{2\omega^j\OO\times 2\omega^j\OO^\times, (r_j+2\omega^j\OO)\times 2\omega^j\OO^\times\}$$ for $j=\{0,1,2\}$. An analogue of this fact also holds with an arbritary constructible ideal $(r+2x\OO_K)\times 2x\OO_K^\times\in\J(\OO\rtimes\OO^\times)$ in place of $2\OO_K\times2\OO_K^\times$ and $$\{2x\omega^j\OO\times 2x\omega^j\OO^\times, (r+xr_j+2x\omega^j\OO)\times 2x\omega^j\OO^\times\}$$ in place of~$\{2\omega^j\OO\times 2\omega^j\OO^\times, (r_j+2\omega^j\OO)\times 2\omega^j\OO^\times\}$ for $j=0,1,2$.

Before giving a simplification of relation (T4) for the $ax+b$\nb-monoid $\OO\rtimes\OO^\times$, let us first introduce some notation, following~\cite{CDLMathAnn2013}. Given an isometric representation $w\colon \OO\rtimes \OO^\times\to B$ in a $\Cst$\nb-algebra~$B$, we regard $\OO$ as a group with its additive operation and let $u\colon \OO\to B$ be the unitary representation given by $r\mapsto w_{(r,1)}$. Similarly, $w$ gives rise to an isometric representation $s\colon \OO^\times\to B$ via $x\mapsto w_{(0,x)}$. The next result is at the same time a simplification of relation (T4) for $\OO\rtimes\OO^\times$  and an application of \corref{cor:uniqueness-orders}.

\begin{cor} Let $K=\QQ(\sqrt{-3})$ and $\OO=\ZZ[\sqrt{-3}]$. Let $W\colon \OO\rtimes\OO^\times\to B$ be a map into a $\Cst$\nb-algebra satisfying relations \textup{(T1)--(T3)} from \defref{def:toeplitz-semigroup}. Let~$s$ and~$u$ be the restrictions of~$W$ to the multiplicative and additive parts of $\OO\rtimes\OO^\times$, respectively, and suppose that 
$$
\prod_{\substack{j=0}}^2(s_{2\omega}^*s_2s_2^*s_{2\omega}-s_{2\omega^j}s_{2\omega^j}^*)\prod_{\substack{j=0}}^2(s_{2\omega}^*s_2s_2^*s_{2\omega}-u_{r_j}s_{2\omega^j}s_{2\omega^j}^*u_{r_j}^*)=0,
$$ 
where $r_j$ is any element in the nontrivial class of $2\OO_K/2\omega^j\OO$ for $j=0,1,2$. Then there is a \Star homomorphism $\pi_W\colon \Toep_{\lambda}(\OO\rtimes\OO^\times)\to B$ that sends~$L_{(r,x)}$ to~$W_{(r,x)}$. Moreover, $\pi_{W}$ is faithful if and only if for every finite set $F$ of primes  in~$\OO \setminus 2\OO_K$, one has \begin{equation}\label{eq:prime-joint-proper}Q_F\coloneqq\Big(1-(\sum_{\substack{r\in \OO/2\OO_K}}u_rs_{2\omega}^*s_2s_2^*s_{2\omega}u_r^*)\Big)\prod_{p\in F}\Big(1-(\sum_{\substack{r\in\OO/p\OO}}u_rs_ps_p^*u_r^*)\Big)\neq 0.\end{equation}

\begin{proof} In order to see that~$w$ satisfies relation (T4) from \defref{def:toeplitz-semigroup}, observe that if~$\alpha$ is a neutral word in~$\mathcal{W}(\OO\rtimes\OO^\times)$ with $K(\alpha)=(r+2x\OO_K)\times 2x\OO_K^\times,$ where $r\in\OO$ and $x\in\OO^\times$, then $$\dot{W}_\alpha= u_rs_xs_{2\omega}^*s_2s_2^*s_{2\omega}s_x^*u_r^*$$ since $w$ satisfies relations (T1)--(T3) from \defref{def:toeplitz-semigroup}. Similarly, if $\beta\in\mathcal{W}(\OO\rtimes\OO^\times)$ is neutral and $K(\beta)=(r+xr_j+2x\omega^j\OO)\times 2x\omega^j,$ then $$\dot{W}_\beta=u_rs_xu_{r_j}s_{2\omega^j}s_{2\omega^j}^*u_{r_j}^*s_x^*u_r^*.$$ Hence, the same reasoning used to prove \proref{prop:t4-otimes} shows that $w$ satisfies relation (T4). By \corref{cor:uniqueness-orders}, there exists a \Star homomorphism $\pi_W\colon \Toep_{\lambda}(\OO\rtimes\OO^\times)\to B$ mapping a canonical generator $L_{(r,x)}$ to $W_{(r,x)}=u_rs_x$, and $\pi_W$ is faithful if and only if~$W$ is jointly proper. Thus all we need to prove is that $W$ is jointly proper if and only if \eqref{eq:prime-joint-proper} holds for every finite set $F$ of primes  in~$\OO \setminus 2\OO_K$.

The `only if' direction is clear. For the converse, suppose that \eqref{eq:prime-joint-proper} holds  for every finite set $F$ of primes  in~$\OO \setminus 2\OO_K$. Since~$2\OO_K$ is a maximal ideal in~$\OO$, it follows that an ideal in~$\OO$ is either contained in~$2\OO_K$ or it is relatively prime to~$2\OO_K$. If $I$ is a proper ideal in~$\OO$ that is relatively prime to~$2\OO_K$, then $I$ is contained in a prime ideal $\mfp$ of~$\OO$ that is itself relatively prime to~$2\OO_K$ because $$\OO\supset \mfp+2\OO_K\supset I+2\OO_K=\OO.$$ Since the ideals of~$\OO$ that are relatively prime to~$2\OO_K$ are principal ideals, we have $\mfp=p\OO$ for a prime element $p\in\OO$.

Now take a constructible ideal $S=(r+I)\times I^\times$ in $\J(\OO\rtimes\OO^\times)$ with $S\subsetneq \OO\rtimes\OO^\times$. By the above paragraph, we have either $S\subset (r+2\OO_K)\times2\OO_K^\times$ or $S\subset(r+p\OO)\times p\OO^\times$ for some prime element $p\in\OO$ that is relatively prime to $2\OO_K$. Then if $\alpha\in\mathcal{W}(\OO\rtimes\OO^\times)$ is a neutral word with $K(\alpha)=S$, it follows that 
$$
1-\dot{W}_\alpha\geq\Big(1-\big(\sum_{\substack{r\in \OO/2\OO_K}}u_rs_{2\omega}^*s_2s_2^*s_{2\omega}u_r^*\big)\Big)\Big(1-\big(\sum_{\substack{r\in\OO/p\OO}}u_rs_ps_p^*u_r^*\big)\Big)=Q_{\{p\}}.
$$ for some prime  $p\in \OO \setminus 2\OO_K$.
So if we  take a finite set~$A$ of neutral words in~$\OO\rtimes\OO^\times$ such that $K(\alpha)\subsetneq\OO\rtimes\OO^\times$ for all $\alpha\in A$, we can find a finite set $F\subset\OO \setminus 2\OO_K$ of primes such that $$\prod_{\alpha\in A}(1-\dot{W}_\alpha)\geq \prod_{\alpha\in A}(1-\dot{W}_\alpha)Q_F =Q_F\neq 0.$$ This shows that $W$ is jointly proper.
\end{proof}
\end{cor}

\section{Right LCM semigroups} \label{sec:rightLCM}
Our  results  give new insight also in the  particular -- and important -- case of right LCM semigroups.  In order  to demonstrate this point we present in this section two applications of our uniqueness results to the universal Toeplitz algebras of right LCM submonoids of groups.
\subsection{Topological freeness and uniqueness} The first application stems from the observation that there is an obvious  parallel between \thmref{thm:uniqueiftopfreejointproper} and \cite[Theorem~4.3]{NNN1}
because the condition given in \cite[equation~(4.1)]{NNN1} corresponds to our joint properness condition, \defref{def:jointlyproper}, 
when applied to right LCM semigroups. We would like to elaborate on this parallel here in order to provide a simplification of the hypothesis  and a strengthening  of \cite[Theorem~4.3]{NNN1}. Along the way, we also shed conceptual light on the technical condition (D2) from \cite[Definition 4.1]{NNN1}:
 \begin{enumerate}
\item [(D2)] if $s_0 \in P$, $s_1\in s_0P$ and $F\subset P$ is a finite subset with $s_1P\cap \Big(P\setminus \bigcup_{q\in F}qP \Big) \neq \emptyset$, then for every $x\in P^*\setminus \{e\}$, there is $s_2 \in s_1P$ satisfying 
\[
s_2P \cap \Big(P\setminus \bigcup_{q\in F}qP \Big) \neq \emptyset\quad \text{ and }\quad s_0\inv s_2 P \cap x s_0\inv s_2 P =\emptyset.
\]
\end{enumerate}

Recall that when $P$ is a right LCM semigroup, the nonempty constructible right ideals of $P$ are all of the form $qP$, $q\in P$, hence an ideal $qP\in\J$ is proper if and only if $q\in P\setminus P^*$. Thus, it seems worth recasting \thmref{thm:idealsifftopfree} specifically for right LCM monoids, giving algebraic conditions on $P$ that are equivalent to topological freeness of the partial action of the underlying group.

\begin{thm}\label{thm:rightlcm-topfree}
 Let $P$ be a right LCM submonoid of a group~$G$. The following are equivalent:
 \begin{enumerate}
 \item the partial action of $G$ on $\Omega_P$ is topologically free;
 \item the action of $P^*$ on $\Omega_P$ is topologically free;
 \item if $u\in P^*\setminus\{e\}$ and $F\subset P\setminus P^*$ is a finite set,
  then there exists $t \in P \setminus \bigcup_{q\in F} qP$ such that $utP \neq tP$ (or, equivalently, such that $ut\notin tP^*$);
\item every  ideal of $\Toepu(P)$ that has trivial intersection with $D_u$ is contained in the kernel of the left regular representation.
 \end{enumerate}
\end{thm}
 When we apply this characterization, we see that (D2) implies topological freeness.
\begin{cor}\label{cor:D2-topfree}
Let $P$ be a right LCM submonoid of a group. If $P$ satisfies \textup{(D2)}, then the action of $P^*$ on  $\Omega_P$ is topologically free.
\begin{proof} It suffices to show that condition (3) of \thmref{thm:rightlcm-topfree} holds. Suppose $u \in P^*\setminus \{e\}$ and let $F\subset P\setminus P^*$ be a finite set.
 When we apply (D2) to the set $F$ and to $s_0 = e$,  $s_1 = e$ and $x = u$ we obtain
 an element  $s_2 \in s_1P = P$ such that $s_2P \cap (P\setminus \bigcup_{q\in F} qP) \neq \emptyset$ 
 and $u s_2P \cap s_2P = \emptyset$. That is, $us_2P$ and $s_2P$ are disjoint subsets of~$P$. This  implies condition (3) of \thmref{thm:rightlcm-topfree}, which only requires those subsets to be different.
\end{proof}
\end{cor}

If~$P$ is a right LCM monoid, an isometric representation $w\colon P\to B$ induces a representation of $\Toepu(P)$ if and only if it satisfies the relation
\begin{equation}\label{eq:nicacovariance} w_pw_p^*w_qw_q^* = \begin{cases} w_rw_r^*&\text{ if }pP\cap qP=rP,\\
0 & \text{ if } pP\cap qP=\emptyset.
\end{cases}
\end{equation} The diagonal subalgebra~$D_u\subset \Toepu(P)$ is the closed linear span of the range projections $\{t_pt_p^*\mid p\in P\}$. We simplify \thmref{thm:uniqueiftopfreejointproper} when $P$ is a right LCM monoid, and obtain, in particular, the group-embeddable case of \cite[Theorem~4.3]{NNN1} as a consequence of \corref{cor:D2-topfree}.

\begin{thm}[cf.\ \cite{NNN1}*{Theorem~4.3}]\label{thm:uniqueiftopfreerightlcm}
Let $P$ be a right LCM submonoid of a group~$G$. Suppose that any of the equivalent conditions of \thmref{thm:rightlcm-topfree} holds 
and that the conditional expectation $E_u\colon\Toepu(P) \to D_u$ is faithful.
Let $W\colon P\to B$ be an isometric representation of~$P$ satisfying \eqref{eq:nicacovariance}. Then the canonical map $\rho_W\colon\Toepu(P) \to \Cst(W)$ is an isomorphism if and only if $$\prod_{p\in F}(1-W_pW_p^*)\neq 0$$ for every finite subset $F\subset P\setminus P^*$.
\end{thm}

It is important to keep in mind that \cite[Theorem 4.3]{NNN1} also applies to $\Cst$\nb-algebras of 
 right LCM semigroups that do not embed in groups, which are not covered by our results.  It is  nonetheless plausible 
 that replacing (D2) by \thmref{thm:rightlcm-topfree}(3) would still produce a version of \thmref{thm:uniqueiftopfreerightlcm} also for semigroups that do not embed in groups. 

\subsection{Pure infiniteness and simplicity}
We now turn our attention to the main findings in~\cite{NNN1} concerning pure infiniteness and simplicity of the Toeplitz algebra of a right LCM semigroup. Specifically, we aim to show that for~$P$  a right LCM submonoid of a group, the conclusion of \cite[Theorem~5.3]{NNN1} follows from a combination of \thmref{thm:idealsifftopfree}, \corref{cor:qu-isomor} and \corref{cor:charac-inf-simple}. We start by interpreting \corref{cor:qu-isomor} in the special case of a right LCM monoid. 
 
\begin{cor}\label{cor:tu-boundarylcm} Let $P$ be a right LCM submonoid of a group. The following are equivalent:
\begin{enumerate}
\item the quotient map $q_u\colon \Toepu(P)\to\CC\times_{\CC^P}P$ is an isomorphism;
\item for every finite set $F\subset P\setminus P^*$, there exists $s'\in P$ such that $s'P\cap qP=\emptyset$ for all $q\in F$;
\item  if $s\in P$ and $F$ is a finite subset of~$P$ with $sP\cap (P\setminus \bigcup_{\substack{q\in F}}qP)\neq\emptyset$, then there is $s'\in sP$ such that $s'P\cap qP=\emptyset$ for all $q\in F$. \textup{(This is condition (D3) of \cite{NNN1})}.
\end{enumerate}
\end{cor}

\begin{cor}\label{cor:toep-purely-simple} Let $P$ be a right LCM submonoid of a group~$G$ with $P\neq\{e\}$. Suppose that~$P$ satisfies any of the conditions of  \thmref{thm:rightlcm-topfree} and any of the conditions of \corref{cor:tu-boundarylcm}. Then $\Toepr(P)$ is purely infinite simple.

\begin{proof} By \corref{cor:tu-boundarylcm}(3), we have $D_r^{\rm{o}}=\{0\}$ and so $\Omega_P=\partial\Omega_P$ by \lemref{lem:basis-boundary}. Thus we have a canonical isomorphism 
$
\Toepr(P)\cong \mathrm{C}(\Omega_P)\rtimes_r G=\mathrm{C}(\partial\Omega_P)\rtimes_r G\cong\partial\Toepr(P).
$
 By  \thmref{thm:rightlcm-topfree}(1), the partial action of~$G$ on~$\Omega_P$ is topologically free, and so $\Toepr(P)$ is purely infinite simple by \corref{cor:charac-inf-simple}.
\end{proof}
\end{cor}

\begin{rem}  We emphasize that \cite[Theorem~5.3]{NNN1} applies to not-necessarily group embeddable cancellative right LCM semigroups.  Nevertheless, we would like to clarify 
the relationship between \corref{cor:toep-purely-simple} and \cite[Theorem~5.3]{NNN1} when  applied to right LCM submonoids of a group.

\textup{(i)} Our assumption $P\neq\{e\}$ in \corref{cor:toep-purely-simple} is necessary  to exclude the case $\Toepr(P)=\CC$, which is not purely infinite. This assumption is not explicitly stated in \cite[Theorem~5.3]{NNN1}.

\textup{(ii)} Case (1) of \cite[Theorem~5.3]{NNN1} assumes $P^*=\{e\}$, and thus (D2) holds vacuously.
In addition, (D2) also holds in case (3) of \cite[Theorem~5.3]{NNN1}, because of \cite[Lemma~5.1]{NNN1}.
Since (D2) implies topological freeness by \corref{cor:D2-topfree}, we see that 
 our \corref{cor:toep-purely-simple} recovers \cite[Theorem~5.3]{NNN1} for right LCM submonoids of groups. 
\end{rem}

\end{document}